\documentclass[11pt]{amsart}
\usepackage{amsaddr}
\usepackage{amsmath,amssymb,amsfonts,amsthm,mathrsfs}
\allowdisplaybreaks
\usepackage{mathrsfs}
\usepackage{dsfont}
\usepackage{algorithm}  
\usepackage{mdwlist}
\usepackage{algpseudocode}  
\usepackage[sort&compress,numbers]{natbib}
\usepackage{enumerate}
\usepackage[T1]{fontenc} % Use 8-bit encoding that has 256 glyphs
\usepackage{graphicx} % Required for including images
\graphicspath{{Figures/}} % Set the default folder for images
\usepackage{subfigure} 
\usepackage{float}
\usepackage{soul}
\usepackage{paralist}
\usepackage{varioref} % More descriptive referencing
\usepackage{hyperref}
\usepackage{mathtools}
\hypersetup{hidelinks,
	colorlinks=true,
	allcolors=blue,
	pdfstartview=Fit,
	breaklinks=true}

\allowdisplaybreaks

\newtheorem{theorem}{Theorem}[section]
\newtheorem{lemma}[theorem]{Lemma}
\newtheorem{corollary}[theorem]{Corollary}
\newtheorem{proposition}[theorem]{Proposition}

\newtheorem{remark}[theorem]{Remark}
\newtheorem{claim}[theorem]{Claim}

\newcommand{\vep}{\varepsilon}

\DeclareMathOperator{\const}{const}
\DeclareMathOperator{\dist}{dist}
\DeclareMathOperator{\re}{Re}
\DeclareMathOperator{\im}{Im}
\DeclareMathOperator{\Ker}{Ker}

\newcommand{\sgn}{\operatorname{sgn}}
\newcommand{\triplenorm}[1]{\left|\!\left|\!\left| #1 \right|\!\right|\!\right|}
\begin{document}
\let\oldsection\section
\renewcommand\section{\setcounter{equation}{0}\oldsection}
\renewcommand\thesection{\arabic{section}}
\renewcommand\theequation{\thesection.\arabic{equation}}

\newcommand{\beq}{\begin{equation}}
\newcommand{\noi}{\noindent}
\newcommand{\eeq}{\end{equation}}
\newcommand{\dis}{\displaystyle}
\newcommand{\mint}{-\!\!\!\!\!\!\int}

\title[transition layer solutions to MCRD system]{Radially symmetric transition-layer solutions in mass-conserving reaction-diffusion systems with bistable nonlinearity}
\author{Xiaoqing He}
\address{School of Mathematical Sciences, Key Laboratory of MEA (Ministry of Education) \& Shanghai Key Laboratory of PMMP,  East China Normal University, Shanghai, China}
\email[Corresponding author]{xqhe@cpde.ecnu.edu.cn}
\thanks{Corresponding author: Xiaoqing He, Email: xqhe@cpde.ecnu.edu.cn}
\author{Quan-Xing Liu}
\address{School of Mathematical Sciences, Shanghai Jiao Tong University, Shanghai, China}
\email{qx.liu@sjtu.edu.cn}
\author{Dong Ye}
\address{School of Mathematical Sciences,  Key Laboratory of MEA (Ministry of Education) \& Shanghai Key Laboratory of PMMP,  East China Normal University, Shanghai, China}
\email{dye@math.ecnu.edu.cn}



\date{}
\begin{abstract}
Mass-conserving reaction-diffusion (MCRD) systems are widely used to model phase separation and pattern formation in cell polarity, biomolecular condensates, and ecological systems. Numerical simulations and formal asymptotic analysis suggest that such models can support stationary patterns with sharp internal interfaces. 
In this work, we establish for a general class of bistable MCRD systems the existence of nonconstant radially symmetric stationary solutions with a single internal transition layer on an $N$-dimensional ball, for general spatial dimension $N$.
Our approach incorporates the global mass constraint directly into a refined matched-asymptotic framework complemented by a uniform spectral/linear analysis. Beyond mere existence, our framework yields arbitrarily high-order asymptotic approximations of the constructed solutions together with quantitative uniform error estimates, which provides a quantitative higher-dimensional theory of transition-layer patterns in MCRD systems and a rigorous justification for their use in modeling phase separation and pattern formation in biological and ecological settings.
\end{abstract}

\maketitle
%-------------------------------------------------------------------------------------
\section{Introduction and Statement of Results}

\subsection{Mass-conserving reaction-diffusion system}

Pattern formation is a fundamental phenomenon underlying a wide range of processes in biological development and ecological systems. Over the past several decades, various theoretical frameworks have been proposed to elucidate the mechanisms driving such self-organized dynamics. Among these, mass-conserving reaction-diffusion (MCRD) systems have emerged
as a powerful paradigm for studying pattern formation in systems where the total quantity of certain chemical species is conserved. A pioneering contribution in this direction was made by Otsuji {\it et al.} \cite{OICetal} and Ishihara {\it et al.} \cite{IOM} in 2007, who demonstrated that MCRD systems can successfully capture the spatiotemporal dynamics of cell polarity. Subsequently, in 2008, Mori {\it et al.} \cite{MJE_2008} proposed a related two-component MCRD model with a distinct form of reaction term to further investigate cell polarization. In dimensionless form, their one-dimensional system reads:
\begin{align}\label{eqn:MCRD_mori}
\begin{cases}
\vep u_t = \vep^2 u_{xx} + f(u,v)  & \text{\upshape{in} }(0,1)\times(0,\infty),\\
\vep v_t = d v_{xx} - f(u,v)   & \text{\upshape{in} }(0,1)\times(0,\infty),\\
u_x(0)=0=u_x(1), ~ v_x(0)=0=v_x(1),\\
u(x,0) =u_0(x), \quad v(x,0) =v_0(x)   & \text{\upshape{in} } (0,1),
\end{cases}
\end{align}
where $u(x,t)$ and $v(x,t)$ denote the concentrations of two interconverting chemical species. The diffusion coefficients $\vep^2$ (for $u$) and $d$ (for $v$) satisfy $0<\vep^2\ll d$; for instance, Mori {\it et al.} \cite{MJE_2011}  considered $\vep\approx 0.03$ and $d\approx 0.1$. The nonlinear function $f(u, v)$ models the interconversion from $v$ to $u$ and is assumed to be \textit{bistable} in $u$ for fixed $v$. A canonical example \cite{MJE_2008, MJE_2011} is
\begin{align}\label{eqn:f(u,v)}
	f(u,v) &=  -u+\left(\delta+\frac{\gamma u^2}{1+u^2}\right)v,
\end{align}
where $\gamma, \delta>0$ satisfy $\gamma>8\delta$ to ensure bistability. Using formal asymptotic analysis combined with a  wave-pinning argument and numerical computations, 
the authors of \cite{MJE_2011} showed that the MCRD system \eqref{eqn:MCRD_mori} supports propagating traveling fronts that eventually arrest to form stationary internal interfaces. Related work by Gomez {\it et al.} \cite{GLM} further explored wave-pinning behavior via analysis of the associated shadow system of \eqref{eqn:MCRD_mori}.

Building upon these foundations, Frey's group~\cite{HF_2018, HBF_2018, BHF_2020, BHF_2021} applied two-component MCRD frameworks to investigate mechanisms of intracellular pattern formation, including cell polarization. Siteur {\it et al.} \cite{Siteur2023} applied MCRD models to describe vegetation patterns in ecosystems. These studies, although primarily focused on modeling and applications rather than rigorous mathematical analysis, highlight the broad applicability of MCRD systems across diverse fields, ranging from cellular dynamics to ecological pattern formation. In parallel, mass-conserving structures also arise in epidemic reaction-diffusion models \cite{ABLN, CS1, CS2, Peng, PSW}, where the authors focus primarily on thresholds (e.g., basic reproduction numbers) and equilibrium stability rather than spatial pattern formation.

\smallskip

More recently, experimental studies of biomolecular condensates---a class of mesoscale liquid-liquid phase-separated structures---have provided new biological motivation for MCRD modeling. In particular, the authors of \cite{LGHLQ} proposed a data-driven MCRD model to describe phase separation phenomena arising in dsDNA-protein interactive co-condensates. In this setting, 
$u(x,t)$ and $v(x,t)$ denote the concentrations of dsDNA-bound protein and free protein, respectively. Experimental measurements \cite{LB-Q, LGHLQ} show that the diffusion coefficient of free protein $v$ is approximately $18.0\pm1.1\,\mu \text{m}^2/\text{s}$, whereas that of the dsDNA-bound protein $u$ is only about $5.1\times 10^{-5}$ times as large, thereby revealing a pronounced multiscale diffusion hierarchy. Furthermore, the rapid reaction kinetics converting $v$ to $u$, governed by biochemical interactions, renders the system strongly reaction-dominated. Motivated by these experimental evidence and observations, we propose the following \textit{generalized} MCRD system:
\begin{align}\label{eqn:MCRD}
\begin{cases}
\vep u_t = \vep^2 \Delta u + f(u,v)  & \text{\upshape{in} } \Omega\times(0,\infty),\\
\vep v_t = \vep D \Delta v - f(u,v)   & \text{\upshape{in} } \Omega\times(0,\infty),\\
\partial_{\nu} u = 0 = \partial_{\nu}v & \text{\upshape{on} } \partial \Omega\times(0,\infty),\\
u(x,0) =u_0(x), \quad v(x,0) =v_0(x)   & \text{\upshape{in} } \Omega,
\end{cases}
\end{align}
where $\Omega\subset \mathbb R^N$ $(N\geq 1)$ is a bounded smooth domain, $\Delta$
%= \sum_{i=1}^N\frac{\partial^2~~}{\partial x_i^2}$ 
is the Laplace operator, and $\partial_{\nu} :=\nu \cdot \nabla$
denotes the outward normal derivative on $\partial \Omega$. Here,
$0<\vep\ll 1$ and $D >0$ are dimensionless parameters that characterize the relative timescales of diffusion and reaction. Setting $D = d/\vep$ recovers the classical scaling in \eqref{eqn:MCRD_mori}. Thus one may regard \eqref{eqn:MCRD} as a \emph{generalized} MCRD system that reduces to \eqref{eqn:MCRD_mori} under this scaling. Moreover, allowing
$D$ to range from $O(1)$ to $O(1/\vep)$ provides a mathematically unified framework that connects the biologically motivated multiscale model \eqref{eqn:MCRD} with the canonical MCRD model \eqref{eqn:MCRD_mori}; see Remark~(d) at the end of this section.

Crucially, for any classical solution \((u(x,t), v(x,t))\) of \eqref{eqn:MCRD}, the total mass is conserved:
\begin{align*}
	\int_{\Omega} [u(x,t)+v(x,t)]\,dx \equiv \int_{\Omega} [u_0(x)+v_0(x)]\,dx, \quad \forall t>0,
\end{align*}
a defining feature that distinguishes MCRD systems from classical non-mass-conserving reaction-diffusion (non-MCRD) systems. \textit{This nonlocal mass constraint poses a major analytical difficulty: it couples local dynamics to a global integral condition, obstructing a direct use of matched-asymptotic, spectral techniques and many other methods developed for classical (non-MCRD) singular perturbation problems.} Therefore, despite substantial advances in the modeling of MCRD systems for pattern formation in biological and ecological contexts, establishing a rigorous general existence and stability theory for patterned stationary solutions in MCRD systems is a challenging problem.

\textit{Our goal in this work is to construct patterned stationary solutions to system \eqref{eqn:MCRD} that exhibit internal transition layers, for $0<\varepsilon \ll 1$ and $D>0$, under a prescribed mass constraint}. Specifically, we consider the steady-state problem:
\begin{align}\label{eqn:MCRD-ss}
	\begin{cases}
		0 = \vep^2 \Delta u + f(u,v)  &\quad  \text{in }  \Omega, \\
		0 = \vep D \Delta v - f(u,v)   &\quad \text{in } \Omega, \\
		\partial_{\nu} u = 0 = \partial_{\nu}v &\quad \text{on } \partial \Omega,
	\end{cases}
\end{align}
subject to the mass constraint
\begin{align}\label{eqn:M}
	\frac{1}{|\Omega|}\int_{\Omega} [u(x)+v(x)]\,dx = M,
\end{align}
where $M>0$ denotes the prescribed averaged total mass.

It is straightforward to verify that system \eqref{eqn:MCRD-ss} is equivalent to the following reformulated system: 
\begin{align}\label{eqn:MCRD-ss-T}
	\begin{cases}
		0 = \vep^2 \Delta u + f(u,v)  &  \quad \text{in }   \Omega,  \vspace{1ex}\\
		\displaystyle	0 = \Delta\Big(\frac{\vep}{D}u+v\Big)    & \quad \text{in }  \Omega, \vspace{1ex}\\
		\partial_{\nu} u = 0 = \partial_{\nu}v &  \quad \text{on } \partial \Omega.\\
	\end{cases}
\end{align}
Let $(u, v)$ be a classical solution of \eqref{eqn:MCRD-ss-T} satisfying the mass constraint \eqref{eqn:M}. The Neumann boundary condition then implies that
$$\frac{\vep}{D}u + v\equiv \const \text{ on } \bar \Omega,$$
which, together with \eqref{eqn:M}, yields the identity
\begin{align}\label{eqn:v=mathcal S^vep-..}
	v =  \mathcal S^{\vep}[u]-\frac{\vep}{D} u,
\end{align}
where the nonlocal operator $\mathcal S^{\vep}[u]$ is defined by
\begin{align}\label{eqn:mathcal S^vep}
	\mathcal S^{\vep}[u] := M - \frac{1}{|\Omega|}\left(1-\frac{\vep}{D}\right)\int_\Omega u(x)\,dx.
\end{align}
Substituting \eqref{eqn:v=mathcal S^vep-..} into the first equation of \eqref{eqn:MCRD-ss-T}, we see that $u$ satisfies the following nonlocal elliptic scalar equation:
\begin{align}\label{eqn:MCRD_nonlocal_scalar}
	\begin{cases}
\displaystyle		\vep^2 \Delta u 
		+ f\left(u, \mathcal S^{\vep}[u]-\frac{\vep}{D} u \right) = 0
		&\quad \text{\upshape{in} }  \Omega,\vspace{1ex}\\
		\partial_{\nu} u = 0  &\quad \text{\upshape{on} } \partial  \Omega.
	\end{cases}
\end{align}

\emph{In summary, under the mass constraint \eqref{eqn:M}, system \eqref{eqn:MCRD-ss} is equivalent to the nonlocal scalar elliptic problem \eqref{eqn:MCRD_nonlocal_scalar}.}
More precisely, a pair \((u,v)\) is a classical solution of \eqref{eqn:MCRD-ss} satisfying \eqref{eqn:M} if and only if \(u\) is a classical solution of \eqref{eqn:MCRD_nonlocal_scalar} and \(v\) is given by \eqref{eqn:v=mathcal S^vep-..}. This equivalence reduces the construction of nonconstant solutions of \eqref{eqn:MCRD-ss} under \eqref{eqn:M} to the analysis of the nonlocal scalar equation \eqref{eqn:MCRD_nonlocal_scalar}.

%\textcolor{red}{In the one-dimensional case, in \cite{KTI,IK} the authors proved the existence and stability of stationary solutions with a single internal transition layer for the MCRD system \eqref{eqn:MCRD_mori}. However, in the existence part, they construct by matched asymptotic expansions an approximate stationary profile of order $O(\vep)$ which satisfies the prescribed total mass constraint \eqref{eqn:M} only up to an error of order $O(\vep)$. Therefore, the approximate solution does not satisfy the exact total mass constraint. They then apply directly the Implicit Function Theorem to obtain ``a genuine solution" near this approximate profile, but the total mass constraint is not incorporated as an additional equation during this step. Hence there is no guarantee that the resulting solution satisfies the exact total mass constraint, which, in our view, leaves a serious gap in the existence proof.} 

We now proceed to introduce the structural assumptions on the nonlinear reaction term \( f(u,v) \). Throughout this paper, we assume that the nonlinear reaction term \( f(u,v) \) satisfies the following conditions:
\begin{enumerate}
	\item [(A1)](\textbf{Bistability}). The function $f(u,v)$ is smooth ($C^\infty$) on $\mathbb R^2$, and there exist two constants $\underline v < \overline v$ such that for each fixed \( v \in (\underline{v}, \overline{v}) \), the equation $f(u,v) = 0$ has exactly three distinct real roots denoted by $h^-(v) < h^0(v) < h^+(v)$. Moreover,
	\begin{align*}
		f_u(h^-(v),v)<0,\quad f_u(h^0(v),v)>0\quad\text{and}\quad f_u(h^+(v),v)<0.
	\end{align*}  
	\item [(A2)](\textbf{Transversality}). For each $v\in (\underline v, \bar v)$, the following inequality holds:
	\begin{align*}
		\Big(\frac{\partial f}{\partial u}- \frac{\partial f}{\partial v}\Big)\Big|_{(u,v) = (h^{\pm}(v), v)} < 0.
	\end{align*}
	\item [(A3)](\textbf{Mass Balance Condition}). For each $v\in (\underline v, \bar v)$, define the function
	\begin{align*}
		J(v) := \int_{h^-(v)}^{h^+(v)} f(s,v)\,ds.
	\end{align*}
	There exists a $v^{\ast} \in (\underline v, \bar v)$ such that
	\begin{align*}
		J(v^{\ast}) = 0 \quad  \text{and} \quad  J'(v^{\ast}) \neq 0.
	\end{align*}
\end{enumerate}

We remark that condition (A2) ensures the local asymptotic stability of the spatially homogeneous steady states $(u,v) \equiv (h^{\pm}(v), v)$ of system \eqref{eqn:MCRD} for all $v\in (\underline v, \bar v)$. Rigorous justifications of this stability result can be found in \cite{MJE_2011, GLM}.

In biological applications, the quantities $u$ and $v$ often represent concentrations of chemical species and are typically required to be positive. However, our analysis does not impose positivity constraints on $u$ and $v$, and the mathematical framework remains valid even when $u$ and $v$ take nonpositive values.

\smallskip
In the one-dimensional case, the authors \cite{KTI,IK} proved the existence of stationary solutions with an internal transition layer for the MCRD system \eqref{eqn:MCRD_mori}, where the $u$-component undergoes a jump across a thin layer. Technically, they used an intrinsic patching/gluing construction by introducing an \emph{a priori} unknown matching point $x^{(\vep)}\in(0,1)$ and decompose the interval into $(0,x^{(\vep)})$ and $(x^{(\vep)},1)$. They then solved the boundary value problems on each subinterval and enforce $C^1$-matching at $x=x^{(\vep)}$, together with the global mass constraint via the Implicit Function Theorem.

The authors claimed in \cite[Remark 2.2]{KTI} that one can obtain a $k$-th order approximation of the solutions for any $k$. However, extending the construction in \cite[Theorem 2.1]{KTI} to higher orders does not appear to be straightforward: one would need to carry out careful recursive computations to verify that the Implicit Function Theorem argument closes at each order, which is a substantial task. For instance, without an explicit understanding of the higher-order terms, it is unclear whether the nondegeneracy condition $h^{+}(v^{})-h^{-}(v^{})\neq 0$ alone suffices to match an arbitrary higher-order expansion.

\smallskip

In this work, we develop a unified high-order expansion framework for the general MCRD system \eqref{eqn:MCRD} on $N$-dimensional balls, thereby overcoming the above challenges. By incorporating the global mass constraint into a refined matched-asymptotic construction and establishing a uniform spectral/linear theory for the resulting singularly perturbed operators, we derive \emph{arbitrarily high-order} asymptotic approximations for radially symmetric single transition-layer solutions with explicit error bounds. These approximations and error controls will serve as the principal technical ingredient in the stability analysis of the solutions constructed here, which will be carried out in forthcoming work.

\subsection{Assumptions and statement of main result}
In this subsection, we state the precise statement of our result concerning existence of stationary single transition-layer solutions to the MCRD system~\eqref{eqn:MCRD} with prescribed total mass, i.e., solutions to \eqref{eqn:MCRD-ss}-\eqref{eqn:M} or equivalently \eqref{eqn:MCRD_nonlocal_scalar}.

For our purposes, we assume throughout the remainder of this paper that
\begin{align*}
\Omega = B_1,
\end{align*}
where $B_r:= B_r(0)$ denotes the ball of radius $r>0$  centered at the origin in $\mathbb R^N$. Our objective is to construct radially symmetric solutions to system~\eqref{eqn:MCRD-ss} subject to the mass constraint \eqref{eqn:M} in $\Omega$. Specifically, we aim to construct solutions exhibiting an internal transition layer across the sharp interface
$$\Gamma_* =  \{x\in\Omega : |x|=R_*\} =\partial B_{R_*},$$
where $R_*\in(0,1)$ denotes the interface location parameter. 
As will be derived in \eqref{eqn:R_*}, $R_*$ is determined in the course of the construction, by the prescribed averaged total mass $M$ in \eqref{eqn:M}. For notational convenience, we set
\begin{align*}
	I_{v^*} := \big(v^* + h^-(v^*), v^* + h^+(v^*)\big)
\end{align*}
to denote the admissible range of the parameter $M$ throughout the paper. 

Since we work with radially symmetric solutions, we shall, by a slight abuse of notation, write $g(x)$ and $g(r)$ interchangeably for a radial function $g$ in $\Omega$, where $r=|x|$.

The main result of this work is stated below.

\begin{theorem}\label{thm:main}
Assume that conditions {\upshape (A1)-(A3)} hold. Then, for each given
\begin{align*}
M\in  I_{v^*}  \text { and }  	D_0> 0,
\end{align*}
there exists a constant $\bar\vep = \bar\vep(M, D_0)>0$ such that system \eqref{eqn:MCRD-ss} admits a family of spherically symmetric classical solutions $(u^{\vep, D}(r), v^{\vep, D}(r))$ defined for $\vep\in(0, \bar\vep]$ and $D\geq D_0$ and satisfying the mass constraint
\begin{align*}
\frac{1}{|\Omega|}\int_{\Omega} [u^{\vep,D}(x) + v^{\vep,D}(x)] dx = M.
\end{align*}  
Moreover, these solutions satisfy the following properties:
\begin{enumerate}
\item  There exists a constant $C = C(M, D_0, \bar\vep)>0$ such that
\begin{align*}
	\|u^{\vep,D}\|_{L^{\infty}(\Omega)}+\|v^{\vep,D}\|_{L^{\infty}(\Omega)}\leq C, \quad \forall \;\vep\in(0,\bar\vep] \text{ and }D\geq D_0.
\end{align*}
\item For each $\eta>0$, there exist two constants $K=K(\eta)
>0$ and $\bar\vep'=\bar\vep'(M, D_0,\eta)\in(0, \bar\vep]$ such that
\begin{align*}
	\begin{cases}
		|u^{\vep, D}(r)-h^+(v^*)|<\eta    	& \text{on }  [0, R_*-\vep K] \vspace{1ex}\\
		|u^{\vep, D}(r)-h^-(v^*)|<\eta  	&\text{on }  [R_*+\vep K, 1]\\
	\end{cases}\quad  \forall\; \vep\in (0, \bar\vep'] \text{ and }D\geq D_0, 
\end{align*}
where 
\begin{align}\label{eqn:R_*}
	R_* =R_*(M):= \Big[\frac{M-v^*-h^-(v^*)}{h^+(v^*)-h^-(v^*)}\Big]^{1/N}\in(0,1).
\end{align}
\item  $v^{\vep,D}(x)=\mathcal S^{\vep}[u^{\vep,D}]-\frac{\vep}{D} u^{\vep,D}(x)$ and
\begin{align*}
	\lim\limits_{\vep\to0} \|v^{\vep, D}- v^*\|_{L^{\infty}(\Omega)} =0 \text{ uniformly in }D\geq D_0.
\end{align*}
\end{enumerate}
\end{theorem}

Theorem~\ref{thm:main} has several immediate consequences, which we summarize below. 

\begin{remark}
\begin{enumerate}
\item[\upshape{(a)}] Part (ii) of Theorem \ref{thm:main} establishes that the transition layer which separates the high- and low-concentration regions in $u^{\vep, D}$, is sharply localized near the interface $r = R_*$, which provides directly an accurate characterization of a width of order $O(\vep)$ for the transition layer.
\item[\upshape{(b)}] Furthermore, for any fixed $\delta > 0$ with $\delta < \min\{R_*, 1 - R_*\}$, we see the uniform convergence
	\begin{align*}
		\lim_{\vep\to0} u^{\vep, D}(r) =
		\begin{cases}
			h^+(v^*) &\text{ uniformly on } [0, R_*-\delta ] \text{ and }D\geq D_0\\
			h^-(v^*) &\text{ uniformly on } [R_*+\delta, 1] \text{ and }D\geq D_0.
		\end{cases}
	\end{align*}
\item[\upshape{(c)}] By setting $D = d/\vep$ in Theorem~\ref{thm:main}, one recovers the canonical MCRD scaling considered in~\cite{MJE_2008, MJE_2011}, thereby obtaining a family of spherically symmetric stationary solutions to system~\eqref{eqn:MCRD_mori}.  
\item[\upshape{(d)}] It is also worth mentioning that the one-dimensional case \(N=1\) is considerably simpler; see, for example, the expansion in \((2.17)\) below. The main reason is that for \(N\ge 2\) (in the radially symmetric setting) one must deal with differential operators with variable coefficients, such as the radial Laplacian \(u_{rr}+\frac{N-1}{r}u_r\), which significantly complicates the analysis. Furthermore, as discussed in Subsection~1.1, deriving \emph{high-order} asymptotic approximations together with uniform, explicit error bounds requires rather involved computations. It is therefore advantageous to work with a single global closed-form expression, as in our approach.
\end{enumerate}
\end{remark}

To establish Theorem~\ref{thm:main}, we exploit the equivalence shown at the end of the preceding subsection: it suffices to establish the following existence result for the nonlocal scalar problem \eqref{eqn:MCRD_nonlocal_scalar} and then set
$$v^{\vep,D}(x): =\mathcal S^{\vep}[u^{\vep,D}]-\frac{\vep}{D} u^{\vep,D}(x).$$

\begin{theorem}\label{thm:main-scalar}
Assume that conditions {\upshape (A1)-(A3)} hold. Then, for each given
\begin{align*}
	M\in  I_{v^*}  \text { and }  	D_0> 0,
\end{align*}
there exists a constant $\bar\vep = \bar\vep(M, D_0)>0$ such that system \eqref{eqn:MCRD_nonlocal_scalar} admits a family of spherically symmetric classical solutions $u^{\vep, D}(r)$ defined for $\vep\in(0, \bar\vep]$ and $D\geq D_0$. 
Moreover, these solutions satisfy the following properties:
\begin{enumerate}
\item  There exists a constant $C = C(M, D_0, \bar\vep)>0$ such that
\begin{align*}
	\|u^{\vep,D}\|_{L^{\infty}(\Omega)}\leq C \quad \forall \;\vep\in(0,\bar\vep] \text{ and }D\geq D_0.
\end{align*}
\item For each $\eta>0$, there exist two constants $K=K(\eta)>0$ and $\bar\vep'=\bar\vep'(M, D_0,\eta)\leq \bar\vep$ such that
	\begin{align*}
		\begin{cases}
			|u^{\vep, D}(r)-h^+(v^*)|<\eta    	& \text{on }  [0, R_*-\vep K] \vspace{1ex}\\
			|u^{\vep, D}(r)-h^-(v^*)|<\eta  	&\text{on }  [R_*+\vep K, 1]\\
		\end{cases}\quad  \forall \;\vep\in (0, \bar\vep'] \text{ and }D\geq D_0, 
	\end{align*}
where $R_*$ is defined as in \eqref{eqn:R_*}.
\end{enumerate}
\end{theorem}

Our approach to proving Theorem~\ref{thm:main-scalar} consists of two main steps. First, we construct a family of smooth ($C^{\infty}$)
radially symmetric approximate solutions to problem \eqref{eqn:MCRD_nonlocal_scalar}, characterized by an internal transition layer located near the sharp interface $\Gamma_*$. These approximate solutions are obtained via matched asymptotic expansions, with particular attention to enforcing the mass constraint~\eqref{eqn:M}, which is now encoded in the nonlocal term $\mathcal S^{\vep}[u]$. In the second step, we apply a contraction mapping argument, after a finite-dimensional spectral reduction (via spectral projection), to perturb the approximate solutions to exact ones for sufficiently small $\vep > 0$, using a Lyapunov-Schmidt-type decomposition to reduce the problem. Crucial ingredients of the analysis include precise asymptotic matching near the transition layer, careful handling of the nonlocal term $\mathcal S^{\vep}[u]$---which governs the solvability of the coupled inner-outer system---and sharp spectral estimates for the linearized operator around the approximate solutions.

\smallskip

Finally, we conclude this subsection with a discussion of the modeling insights and broader implications of Theorem~\ref{thm:main}, with particular emphasis on its connections to phase separation, analytical challenges, and multiscale pattern formation in biological systems.

(a) Theorem \ref{thm:main} demonstrates that under general structural assumptions (A1)-(A3), MCRD systems admit spatially nonhomogeneous stationary states characterized by phase separation with a droplet-like interface. Such solutions arise when the averaged total mass $M$ lies strictly between the two critical thresholds $v^*+ h^-(v^*)$ and $v^*+ h^+(v^*)$. Moreover, varying $M$ within this interval affects only the volume fraction 
between the high- and low-concentration regions, through the interface location $R_*=R_*(M)$ defined in~\eqref{eqn:R_*}. This agrees qualitatively with experimental findings on liquid-liquid phase separation \cite{AGM, BBH} and supports the use of MCRD systems as a theoretical model for biomolecular condensate formation.

(b) Patterned solutions with internal/boundary transition layers or spikes have been extensively studied in classical non-mass-conserving reaction-diffusion (Non-MCRD) systems, e.g., in the following form:
\begin{align*}
	\begin{cases}
		\vep  u_t = \varepsilon^2\Delta u  + f(u,v)   \quad & \text{\upshape{in} } \Omega\times(0,\infty),\\
		v_t = d \, \Delta v + g(u,v)                \quad  & \text{\upshape{in} } \Omega\times(0,\infty),\\
		\partial_{\nu}u = \partial_{\nu}v=0  \quad  & \text{\upshape{in} } \partial\Omega\times(0,\infty),\\
		u(x,0) =u_0(x), \quad v(x,0) =v_0(x)   & \text{\upshape{in} } \Omega.
	\end{cases}
\end{align*} 
See, for instance, \cite{Fife, Ito, Lin, MTH,  NF, Taniguchi, S} for results in one spatial dimension, and \cite{delPino, HS, NT_1995, SS1, SS2} for the multi-dimensional setting. However, the analytical techniques developed in these works do not directly apply to MCRD systems due to the presence of the nonlocal mass constraint~\eqref{eqn:M}, which fundamentally alters the structure of the problem by coupling local dynamics to a global integral condition.

(c) A classical theory for pattern formation is Turing instability \cite{Turing}, which typically predicts the emergence of spatially periodic patterns via linear instability of homogeneous steady states \cite{Murray}. In contrast, MCRD systems often exhibit spatiotemporal coarsening dynamics\cite{HBF_2018, IOM}, wherein transient patterns evolve toward polarized states consisting of two distinct phases---a dense phase and a dilute phase. A rigorous understanding of the full dynamical pathway leading to such polarized steady states remains a significant and challenging open problem.

(d) Our generalized MCRD model~\eqref{eqn:MCRD} accommodates a range of scaling regimes through the parameter $D$. When $D$ is treated as independent of  $\varepsilon$, the model captures a multiscale hierarchy between diffusion and reaction processes. Alternatively, choosing $D = d/\varepsilon$ recovers the classical scaling used in \eqref{eqn:MCRD_mori}. 
Given that natural pattern-forming systems often involve widely separated spatial and temporal scales, the flexibility in scaling endowed by MCRD models enables them to serve as a unified cross-scale framework for investigating pattern formation across diverse biological and ecological contexts.

\smallskip

The remainder of the paper is organized as follows. In Section 2, we construct a family of smooth ($C^{\infty}$)
radially symmetric approximate solutions exhibiting internal transition layers and derive sharp asymptotic estimates. Section 3 develops the spectral analysis of the linearized operator about the approximate solutions. In Section 4 we employ a finite-dimensional spectral reduction together with a contraction mapping argument to obtain exact solutions near the approximate family for sufficiently small $\vep$, thereby proving the main results.

\medskip

\section{Construction of the approximate solutions}
In this section, we construct a family of smooth, spherically symmetric approximate solutions to problem \eqref{eqn:MCRD_nonlocal_scalar} that exhibit an internal transition layer across the interface $\Gamma_*$. These approximate solutions can achieve arbitrarily high order of accuracy measured in terms of
$\vep^k$ for any prescribed $k\in\mathbb N$. The width of the transition layer is of order $O(\vep)$.

\subsection{Preliminaries}
Under the conditions (A1) and (A3), for each $s\in(\underline v, \bar v)$, we consider the determination of the pair $(Q(z;s), c(s))$ as the solution to the boundary value problem:
\begin{align}\label{eqn:Q(z;v)}
\begin{cases}
Q_{zz} - c(s)Q_z + f(Q, s) = 0  \quad z\in \mathbb R,\\
Q(0; s) = h^0(s), \quad \lim_{z\to\mp\infty}Q(z; s) = h^{\pm}(s).
\end{cases}
\end{align}

The following result follows from \cite{Fife1977}. See also \cite[Proposition 2.3]{HS}.

\begin{proposition}\label{pro:Q(z;v)}
Under the condition {\upshape(A1)}, for each $s\in (\underline v, \bar v)$, the problem \eqref{eqn:Q(z;v)} has a unique solution $(Q(z;s), c(s))$ satisfying the following properties:
\begin{enumerate}
\item The function $Q(z;s)$ (respectively $Q_z$ and $Q_{zz}$) approaches the limit $h^{\pm}(s)$ (respectively $0$) at an exponential rate as $z\to\mp\infty$, and $Q_z(z;s)<0$ for $z\in\mathbb R$.
\item 
%Denote $$m(s) := \int_{\mathbb R}Q^2_z(z;s)\,dz,$$ 
Then the function $c(s)$ is explicitly given by 
\begin{align*}
c(s) = -\frac{J(s)}{m(s)},\quad \mbox{where } m(s) := \int_{\mathbb R}Q^2_z(z;s)\,dz.
\end{align*}
If in addition, {\upshape(A3)} is satisfied, then 
\begin{align*}
c'(v^*) = -\frac{J'(v^*)}{m(v^*)}.
\end{align*}
\end{enumerate}
\end{proposition}

\subsection{Statements on existence of approximate solutions}
We begin by constructing a family of approximate solutions to problem~\eqref{eqn:MCRD_nonlocal_scalar} that exhibit internal transition layers. The central result of this section is stated as follows.

\begin{theorem}\label{thm:approx_sol_scalarEQ}
Assume that conditions {\upshape (A1)-(A3)} hold. Then for each given $M\in  I_{v^*}$, there exists a family of smooth $(C^{\infty})$ radially symmetric functions $u^{\vep,D}_k(r)$ defined for $\vep>0$, $D>0$ and $k\in\mathbb N$ such that the following statements hold true:
\begin{enumerate}
\item For each $D_0>0$, $k\in\mathbb N$ and $\eta>0$, there exist two constants $K=K(\eta)>0$ and $\bar\vep = \bar\vep(M, D_0, k, \eta)>0$ such that
\begin{align*}
\begin{cases}
	|u^{\vep, D}_k(r)-h^+(v^*)|<\eta    	& \text{on }  [0, R_*-\vep K] \vspace{1ex}\\
	|u^{\vep, D}_k(r)-h^-(v^*)|<\eta  	&\text{on }  [R_*+\vep K, 1]
\end{cases}\quad  \forall \;\vep\in (0, \bar\vep] \text{ and }D\geq D_0, 
\end{align*}
where $R_*=R_*(M)$ is defined as in \eqref{eqn:R_*}.
\item For each $D_0>0$, $\vep_0>0$ and $k\in\mathbb N$, there exists a constant  $C_k = C_k(M, D_0, \vep_0)>0$ such that
\begin{align}\label{eqn:u^vep,D_k unif.bdd}
	\|u^{\vep,D}_k\|_{L^{\infty}(\Omega)}\leq C_k,  \quad \forall \;\vep\in(0,\vep_0] \text{ and }D\geq D_0
\end{align}
and
\begin{align}\label{eqn:R^vep}
	\big\|\mathcal R^\vep[u^{\vep,D}_k]\big\|_{L^{\infty}(\Omega)} \leq C_k\vep^{k+1}, \quad \forall \;\vep\in(0,\vep_0] \text{ and }D\geq D_0, 
\end{align}
where 
\begin{align*}
	\mathcal R^\vep[u^{\vep,D}_k] := & ~\vep^2 \Delta u^{\vep, D}_k + f(u^{\vep, D}_k,\mathcal S^{\vep}[u^{\vep, D}_k] - \frac{\vep}{D} u^{\vep, D}_k).
\end{align*}
\item Set $v^{\vep,D}_k := \mathcal S^{\vep}[u^{\vep, D}_k]  -\frac{\vep}{D} u^{\vep, D}_k$, then  $$\displaystyle\frac{1}{|\Omega|}\int_{\Omega} [u^{\vep,D}_k(x) + v^{\vep,D}_k(x)] dx = M$$
and for each $k\in\mathbb N$,
	\begin{align*}
\lim\limits_{\vep\to0} \big\|v_k^{\vep, D}- v^*\big\|_{L^{\infty}(\Omega)} =0 \text{ uniformly in }D\geq D_0.
	\end{align*}
\end{enumerate}	
\end{theorem}

Note that the results of Theorem \ref{thm:approx_sol_scalarEQ} can also be reformulated within the framework of system \eqref{eqn:MCRD-ss}, due to the equivalence of system \eqref{eqn:MCRD-ss} under the mass constraint \eqref{eqn:M} and problem \eqref{eqn:MCRD_nonlocal_scalar}.

\subsection{Outer and inner expansions}
Let $u^\vep$ be a radially symmetric solution with an internal transition layer to equation \eqref{eqn:MCRD_nonlocal_scalar}, i.e., $u^\vep$ satisfies
the following equation
\begin{align}\label{eqn:u^vep for expansion}
	\begin{cases}
		\vep^2 \Delta u^\vep 
		+ f\left(u^\vep, \mathcal S^{\vep}[u^\vep]-\frac{\vep}{D} u^\vep \right) = 0
		&\quad \text{\upshape{in} }  \Omega,\vspace{1ex}\\
		\partial_{\nu} u^\vep = 0  &\quad \text{\upshape{on} } \partial  \Omega.
	\end{cases}
\end{align}
Note that we have suppressed the explicit dependence of 
$u^{\vep}$ on the parameters $D$ and $M$ for notational simplicity. Our objective in this subsection is to construct the outer expansion of $u^\vep$ in regions away from the interface $\Gamma_* = \partial B_{R_*}$ and the inner expansion in a neighborhood of $\Gamma_*$, where $R_*\in(0,1)$ denotes the interface location parameter and is to be determined during the construction (cf. Proposition \ref{pro:ajbjR_*determine} below).

Let us denote 
\begin{align}\label{eqn:Omega^pm}
\Omega^- := B_{R_*}\quad \text{and}\quad \Omega^+ := B_1\backslash \overline{B}_{R_*}.
\end{align}
We first construct the outer expansions $U^j$ for $j\in\mathbb N$ by formally expanding the solution $u^\vep$ of \eqref{eqn:u^vep for expansion} into a power series in $\vep$ of the form
\begin{align}\label{eqn:outer-expan}
u^{\vep}(x) \sim \sum_{j\geq 0}\vep^jU^{j}(x),
\end{align}
such that the truncated expansion asymptotically satisfies equation \eqref{eqn:u^vep for expansion} in the outer regions $\Omega\backslash\Gamma_* = \Omega^-\cup \Omega^+$,
where 
\begin{align}\label{eqn:U^j DEF}
U^{j}(x) :=
\begin{cases}
U^{-,j}(x)  \quad \quad \text{\upshape{in} } \Omega^-\\
U^{+,j}(x) \quad \quad \text{\upshape{in} } \Omega^+
\end{cases}
\end{align}
and is required to be discontinuous at $|x|= R_*$, thus exhibiting a sharp transition layer across $\Gamma_*$. 

To derive the equations satisfied by $U^j(x)$, it is necessary to also expand the nonlocal term $S^{\vep}[u^\vep]$ in equation \eqref{eqn:u^vep for expansion}, which involves the integral of $ u^\vep$, into a formal power series in $\vep$. However, since $u^\vep$ is not yet known, this expansion cannot be computed explicitly. To proceed, we formally postulate an expansion of the form
\begin{align}\label{eqn:mathcal S^vep_sim_vep^jA_j}
	\mathcal S^{\vep}[u^\vep] \sim \sum_{j\geq 0}\vep^j A_j,
\end{align}
where the coefficients $A_j$ are constants to be determined in conjunction with the outer and inner expansions of $u^\vep$ in Proposition \ref{pro:ajbjR_*determine} below. The expansion \eqref{eqn:mathcal S^vep_sim_vep^jA_j}, when applied to the approximate solution obtained via matched asymptotic expansions, will be shown to hold to arbitrarily high order in $\vep$ in Proposition \ref{pro:ajbjR_*determine}. This underscores a key analytical challenge in the construction of approximate solutions to~\eqref{eqn:u^vep for expansion}---namely, the proper handling of the mass constraint \eqref{eqn:M} which is now encoded in the nonlocal term $\mathcal S^{\vep}[u^\vep]$.

Substituting \eqref{eqn:outer-expan} and \eqref{eqn:mathcal S^vep_sim_vep^jA_j} into \eqref{eqn:u^vep for expansion} and equating coefficients of like powers of $\vep$, one obtains equations for $U^{\pm, j}$ given by
\begin{align}\label{eqn:outer-u-equs}
\begin{cases}
\text{(i)}& \hspace{-1ex} 0 = f(U^{\pm,0}, A_0)\vspace{1ex}\\
\text{(ii)}&\hspace{-1ex}  0 = f_u^{\pm,0}U^{\pm,1} +f_v^{\pm,0}(A_1-\frac{1}{D}U^{\pm,0}) \vspace{1ex}\\ 
\text{(iii)}&\hspace{-1ex} 0 = f_u^{\pm,0}U^{\pm,j} + f_v^{\pm,0}(A_j-\frac{1}{D}U^{\pm,j-1}) + f^{\pm,j} \hspace{2ex}(j\geq2)
\end{cases} \text{ in }\Omega^{\pm},
\end{align}
where $f_u^{\pm,0}: =f_u(U^{\pm,0}, A_0)$, $f_v^{\pm,0}: =f_v(U^{\pm,0}, A_0)$ and $f^{\pm,j} (j\geq2)$ are given by
\begin{align*}
f^{\pm,j} = ~ &\Delta U^{\pm,j-2} - \begin{bmatrix}f_u^{\pm,0}U^{\pm,j} + f_v^{\pm,0}(A_j-\frac{1}{D}U^{\pm,j-1})\end{bmatrix} \\
&+ \frac{1}{j!}\frac{\partial^j}{\partial\vep^j}f\Big(
	\sum_{k\geq 0}\vep^kU^{\pm,k}, \, A_0+\sum_{k\geq 1}\vep^k (A_k-\frac{1}{D}U^{\pm,k-1})\Big)\Big|_{\vep=0}.
\end{align*}
Note that $f^{\pm,j}$ depends only on $(U^{\pm,k}, A_k)$ with $0\leq k\leq j-1$.

Since we are looking for solutions with a sharp internal transition layer across $\Gamma_*$, in accordance with the condition (A1), we choose
\begin{align}\label{eqn:U^{pm,0} = h^mp(A_0)}
U^{\pm,0}(x) \equiv h^\mp(A_0)  \qquad \text{\upshape{in} } \Omega^\pm
\end{align}
as solution of \eqref{eqn:outer-u-equs}-(i). This implies that $f_u^{\pm,0}<0$ by the condition (A1). Note that an alternative choice of \eqref{eqn:U^{pm,0} = h^mp(A_0)} would be $U^{\pm,0}(x) \equiv h^\pm(A_0)$ in $\Omega^\pm$ receptively. Then we will obtain a different family of transition layer solutions which is structurally symmetric with our current choice. See Theorem \ref{thm:main2} below. Then by induction on $j\geq1$, we see from \eqref{eqn:outer-u-equs} that 
\begin{align}\label{eqn:outer-u-solutions}
\begin{cases}
\text{(i)}&\hspace{-1ex}	U^{\pm,0}(x) \equiv h^\mp(A_0)  \vspace{1ex}\\
\text{(ii)}	&\hspace{-1ex} U^{\pm,1}(x) \equiv h_v^{\mp}(A_0)(A_1-\frac{1}{D}U^{\pm,0})    \vspace{1ex}\\
\text{(iii)}&\hspace{-1ex}	U^{\pm,j}(x) \equiv h_v^{\mp}(A_0)(A_j-\frac{1}{D}U^{\pm,j-1}) - (f_u^{\pm,0})^{-1}f^{\pm,j}\quad  (j\geq2) 
\end{cases}\text{in } \Omega^\pm.
\end{align}
Therefore, $U^{\pm,j}~(j\geq 0)$ is determined by $A_k$ with $0\leq k\leq j$. Note that $U^{\pm,j}$ are constant in $\Omega^\pm$ for all $j\geq 0$. Henceforth, for notational convenience, we will use $U^{\pm,j}$ to denote either the constant value or the corresponding function defined on $\Omega^\pm$, without risk of confusion. The expression of $f^{\pm,j}$ $(j\geq2)$ can now be simplified as
\begin{align}\label{eqn:f^pm,j}
\begin{split}
f^{\pm,j} & = - \Big[f_u(h^\mp(A_0), A_0)U^{\pm,j} + f_v(h^\mp(A_0), A_0)(A_j-\tfrac{1}{D}U^{\pm,j-1})\Big]\\
& \quad + \frac{1}{j!}\frac{\partial^j}{\partial\vep^j}f\Big(
\sum_{k\geq 0}\vep^kU^{\pm,k},\,A_0+\sum_{k\geq 1}\vep^k (A_k-\tfrac{1}{D}U^{\pm,k-1})
\Big)\Big|_{\vep=0}. 
\end{split} 
\end{align}

To analyze the sharp transition near the interface $\Gamma_*$, we introduce the stretched variable $z = \frac{r-R_*}{\vep}$. Under this transformation, equation \eqref{eqn:u^vep for expansion} is rewritten as:
\begin{align}\label{eqn:MCRD-streched}
u_{zz}^{\vep} + \vep \tfrac{N-1}{R_*+\vep z} u_z^{\vep}+ f\begin{pmatrix}u^{\vep}, \mathcal S^{\vep}[u^\vep]-\frac{\vep}{D} u^{\vep}\end{pmatrix} =0 
\qquad z\in \begin{pmatrix}-\frac{R_*}{\vep}, \frac{1-R_*}{\vep}\end{pmatrix}.
\end{align} 
We will determine the inner expansions $w^j$ for $j\in\mathbb N$ by formally expanding the solution $u^\vep$ of \eqref{eqn:MCRD-streched} into a power series in $\vep$ of the form
\begin{align}\label{eqn:inner-expan}
u^{\vep}(z) \sim \sum_{j\geq 0}\vep^jw^j(z),
\end{align}
such that the truncated expansion  
asymptotically satisfies \eqref{eqn:MCRD-streched}. 
Substituting \eqref{eqn:inner-expan} and \eqref{eqn:mathcal S^vep_sim_vep^jA_j} into \eqref{eqn:MCRD-streched} and equating coefficients of like powers of $\vep$, we obtain equations for $w^j$ as follows:
\begin{align}\label{eqn:w^0_eqn}
0 = w^0_{zz} + f(w^0, A_0)
\end{align}
and for $j\geq1$, 
\begin{align}\label{eqn:w^j_eqn}
0 = w^j_{zz} + f_u(w^0,A_0)w^j + f_j,
\end{align}
where 
\begin{align}\label{eqn:f_j}
f_j & := \frac{1}{j!}\frac{\partial^j}{\partial\vep^j}f\Big(\sum_{k\geq 0}\vep^kw^k, A_0+\sum_{k\geq 1}\vep^k (A_k-\frac{1}{D}w^{k-1})\Big)\Big|_{\vep=0}\\
&\quad -f_u(w^0,A_0)w^j+(N-1)\sum_{k=0}^{j-1}(-1)^k\frac{z^k}{R_*^{k+1}}w^{j-1-k}_z. \notag
\end{align} 
Note that $f_j$ depends only on lower order terms $w^k$ $(0\leq k \leq j-1)$ and $A_k$ $(0\leq k \leq j)$. 

For our purposes, we fix a positive constant $d_0$ throughout this paper such that
\begin{align}\label{eqn:d_0}
	0<d_0<\min\big\{\sqrt{-f_u(h^-(v^*), v^*)}, \sqrt{-f_u(h^+(v^*), v^*)}  \big\}.
\end{align}
The constant $d_0$ depends solely on the nonlinearity $f$ and is independent of the parameters $M$, $\vep$ and $D$. Hence, in what follows, we shall not explicitly trace the dependence of other quantities on $d_0$.

After both the outer expansions $U^j(x)$ and the inner expansions $w^j(z)$ are obtained, we will smoothly match them across the interface $\Gamma_*$ using a standard cut-off function technique. To facilitate this matching, we extend the domain of each inner expansions $w^j$ from $\left(-\frac{R_*}{\vep}, \frac{1-R_*}{\vep}\right)$ to $\mathbb R$, and require that $w^j$ satisfies the following exponential matching condition for each $j\in\mathbb N$:
\begin{align}\label{eqn:bc_for_w^j}
|w^j(z)-U^{+,j}|\cdot\chi_{[0,\infty)}(z)+|w^j(z)-U^{-,j}|\cdot\chi_{(-\infty,0]}(z)\leq c_j e^{-d_0|z|} \text{ in }\mathbb R,
\end{align}
where $\chi$ denotes the characteristic function of the set indicated, $U^{\pm,j}$ are constants (cf. the comment after \eqref{eqn:outer-u-solutions}),
$c_j>0$ is some constant to be determined along with $w^j$.

\medskip

We now seek solutions of \eqref{eqn:w^0_eqn} and \eqref{eqn:w^j_eqn} under the condition \eqref{eqn:bc_for_w^j}. 

For equation \eqref{eqn:w^0_eqn} under the assumption \eqref{eqn:bc_for_w^j}, it follows from Proposition \ref{pro:Q(z;v)}, \eqref{eqn:outer-u-solutions}-(i) and  assumption (A3) that we must choose
\begin{align}\label{eqn:A_0=v^*}
A_0:=v^*
\end{align}
in order to admit a solution for equation \eqref{eqn:w^0_eqn}, which is given by $w^0(z) = Q(z+a_0; v^*)$, where $a_0\in\mathbb R$ is a translation parameter to be determined later in Proposition \ref{pro:ajbjR_*determine}.
For notational simplicity, we will henceforth write 
$Q(z+a_0)$ in place of $Q(z+a_0; v^*)$ and set 
\begin{align}\label{eqn:w^0(z) = Q(z+a_0; v^*)}
	w^0(z) := Q(z+a_0).
\end{align}
Since $w^0(z)-h^\mp(v^*)$, $w^0_z(z)$ and $w^0_{zz}(z)$ decay to zero exponentially as $z\to\pm\infty$ by Proposition \ref{pro:Q(z;v)}, there exists a constant $c_0 = c_0(a_0)>0$ such that 
\begin{align}\label{eqn:w^0-U0,w^0_z,w^0_zz}
	|w^0(z)-h^{-}(v^*)|\cdot\chi_{[0,\infty)}(z) &+|w^0(z)-h^+(v^*)|\cdot\chi_{(-\infty,0]}(z)\\
	&+|w^0_z(z)| + |w^0_{zz}(z)| \leq c_0 e^{-d_0|z|} \quad \text{ in }\mathbb R,\notag
\end{align}
which guarantees that condition \eqref{eqn:bc_for_w^j} holds for $j=0$. 

It follows from \eqref{eqn:A_0=v^*} and \eqref{eqn:outer-u-solutions}-(i) that $U^{0}$ is now determined as 
\begin{align}\label{eqn:U^0=h^pm(v*)}
U^{0}(x) = \begin{cases}
U^{-,0} \equiv h^{+}(v^*)  \quad  \text{ in }\Omega^-\\
U^{+,0} \equiv h^{-}(v^*)   \quad \text{ in }\Omega^+.
\end{cases}
\end{align}

Next, we solve equation \eqref{eqn:w^j_eqn} subject to the condition \eqref{eqn:bc_for_w^j}. We need to establish a technical lemma as a preparation. We begin with introducing the following Banach space
\begin{align*}
	\mathcal X:=\Big\{g\in C(\mathbb R) \,|\,&g_{\pm\infty} := \lim_{z\to\pm \infty} g(z) \text{ exist and there exists a constant }C_{g}>0
\\&\text{such that } |g(z) -g_{\sgn(z)\infty}|\leq C_{g}e^{-d_0|z|} \text{ in }\mathbb R\Big\},
\end{align*}
equipped with the norm
\begin{align*}
\|g\|_{\mathcal X} := |g_{\infty}|+|g_{-\infty}|+\sup_{z\geq 0}\big|(g(z) -g_{\infty})e^{d_0z}\big|+\sup_{z\leq 0}\big|(g(z) -g_{-\infty})e^{-d_0z}\big|.
\end{align*}
We also define the closed subspace $\mathcal X_0 \subset \mathcal X$ consisting of functions that decay exponentially fast to zero at both infinities:
\begin{align*}
\mathcal X_0 := \{g\in \mathcal X \,|\, g_{\pm\infty}  = 0\}.
\end{align*}
Define the linear operator $L^0$ by 
\begin{align}\label{eqn:L^0}
	L^0\varphi := \varphi_{zz} + f_u(w^0(z), v^*)\varphi.
\end{align}
It is easy to see from \eqref{eqn:w^0(z) = Q(z+a_0; v^*)}, \eqref{eqn:w^0-U0,w^0_z,w^0_zz} and Proposition \ref{eqn:Q(z;v)} that 
$$w^0_z\in \mathcal X_0 \quad \text{ and }\quad  L^0 w^0_z =0.$$ 
We then define the closed subspace of $\mathcal X$:
\begin{align*}
	[w^0_z]^{\perp} := \{	g \in \mathcal X \,|\, \langle g, w^0_z\rangle =0 \}, \quad \mbox{where } \langle \varphi, \psi\rangle:=\int_{\mathbb R} \varphi\overline\psi\,dz.
\end{align*}
It inherits the Banach space structure from  $\mathcal X$.

It is straightforward to verify that solving equation~\eqref{eqn:w^j_eqn} subject to the boundary condition~\eqref{eqn:bc_for_w^j} is equivalent to solving~\eqref{eqn:w^j_eqn} in $\mathcal{X}$. We now claim

\begin{lemma}\label{lem:L^0-solvabilty}
Consider the linear inhomogenouse equation 
\begin{align}\label{eqn:L^0_varphi=g}
	L^0 \varphi = g, \quad \text{ where }g\in \mathcal X. 
\end{align} 
Then the following statements hold.
\begin{enumerate}
\item The problem \eqref{eqn:L^0_varphi=g} has a solution $\varphi\in \mathcal X$ if and only if $g\in [w^0_z]^{\perp}$.
\item The operator $L^0$ has a bounded right inverse 
$$(L^0)^{-1}: [w^0_z]^{\perp}\to [w^0_z]^{\perp}.$$
That is, for each $g\in[w^0_z]^{\perp}$, the problem \eqref{eqn:L^0_varphi=g} has a unique solution in $[w^0_z]^{\perp}$, which we denote as
$\varphi:=(L^0)^{-1}g$ and it satisfies
\begin{align}\label{eqn:w_+-infty}
\varphi_{\pm\infty} = \frac{g_{\pm\infty}}{f_u(h^{\mp}(v^*), v^*)} \text{ and } \varphi_z, \varphi_{zz}\in\mathcal X_0.
\end{align}
Moreover,
\begin{align}\label{eqn:w,w_z,w_zz<C|g|}
		\|\varphi\|_{\mathcal X} +\|\varphi_z\|_{\mathcal X} +\|\varphi_{zz}\|_{\mathcal X}\leq C\|g\|_{\mathcal X},
\end{align}
where $C>0$ is a constant depending only on $w^0$.
In fact, $\varphi$ can be explicitly expressed as
\begin{align}\label{eqn:w=hat_w+alphg*w^0_z}
	\varphi(z) = \alpha w^0_z(z) +\bar \varphi(z),
\end{align}
where 
\begin{align*}
	\bar \varphi(z) =  w^0_z(z)\int_{0}^z\frac{d\xi}{(w^0_z(\xi))^2}\int_{\pm\infty}^{\xi}w^0_z(\eta)g(\eta)\,d\eta, \quad 
	\alpha = -\frac{\langle \bar \varphi, w^0_z\rangle}{\langle w^0_z, w^0_z\rangle}.
\end{align*}
\end{enumerate}
\end{lemma}

\begin{remark}
A similar result to Lemma \ref{lem:L^0-solvabilty} has been obtained in \cite[Lemma 3.1]{HS}. However, a key distinction in our result is that we state it in a suitable functional framework, which confirms that the constant $C$ in \eqref{eqn:w,w_z,w_zz<C|g|} depends only on $w^0$ and is independent of the inhomogeneity $g$. This refinement is crucial for the subsequent analysis, where \eqref{eqn:w,w_z,w_zz<C|g|} is applied systematically to monitor parameter dependence across a sequence of estimates.
\end{remark}

\begin{proof}[Proof of Lemma \ref{lem:L^0-solvabilty}]
We first prove Part (i). Assume that $g\in[w^0_z]^{\perp}$. It follows directly from variation of constant formula that $\varphi$ defined by \eqref{eqn:w=hat_w+alphg*w^0_z} satisfies equation \eqref{eqn:L^0_varphi=g}. By applying L'H\^{o}sptical's rule repeatedly, we can easily show that
\begin{align*}
\Big|\varphi - \frac{g_{\pm\infty}}{f_u(h^{\mp}(v^*), v^*)}\Big| \leq C(\|g\|_{\mathcal X}, \|w^0\|_{\mathcal X}) \cdot e^{-d_0|z|}.
\end{align*}
where $C(\|g\|_{\mathcal X}, \|w^0\|_{\mathcal X}) >0$ is a constant depending only on the variables indicated.
This implies the first identity in \eqref{eqn:w_+-infty}, which combined with the equation for $\varphi$ implies that $(\varphi_{zz})_{\pm\infty} =0$. Then it follows from the interpolation inequality that $(\varphi_{z})_{\pm\infty} =0$. Therefore, $\varphi_z, \varphi_{zz}\in\mathcal X_0$. This finishes the proof of \eqref{eqn:w_+-infty}.

On the other hand, suppose \eqref{eqn:L^0_varphi=g} has a solution $\varphi\in\mathcal X$. Since $g\in\mathcal X$, $w^0_z\in\mathcal X_0$ and $L^0w^0_z =0$, multiplying both sides of $L^0\varphi =g$ by $w^0_z$ and integrating over $\mathbb R$, using integration by parts, we see that $\langle g, w^0_z\rangle =0 $, which implies that $g
\in [w^0_z]^{\perp}$. This finishes the proof of Part (i).

To finish the proof for Part (ii),
it only remains to prove \eqref{eqn:w,w_z,w_zz<C|g|}.
In fact, it is easy to verify that $(L^0)^{-1}: [w^0_z]^{\perp} \to [w^0_z]^{\perp}$ is a continuous operator with closed graph. Then by the closed graph theorem, there exists a constant $C>0$ depending only on $w^0$ such that 
$$\|\varphi\|_{\mathcal X} \leq C\|g\|_{\mathcal X} \quad \forall \;g\in[w^0_z]^{\perp}.$$
Then it follows from the equation for $\varphi$ that $\|\varphi_{zz}\|_{\mathcal X}\leq C\|g\|_{\mathcal X}$. Now \eqref{eqn:w,w_z,w_zz<C|g|} follows from the above two estimates and the interpolation inequality. This finished the proof of the lemma.
\end{proof}

Now we are ready to solve equation \eqref{eqn:w^j_eqn} under the condition \eqref{eqn:bc_for_w^j}.
\begin{proposition}\label{pro:sov of uj_condi on A_j_aj}
Assume that conditions {\upshape (A1)-(A3)} hold. Let $A_0$, $U^0$ and $w^0$ be defined by \eqref{eqn:A_0=v^*}, \eqref{eqn:U^0=h^pm(v*)} and \eqref{eqn:w^0(z) = Q(z+a_0; v^*)} respectively, where $R_*\in(0, 1)$ is the interface location parameter appearing in the definition of $\Omega^{\pm}$ (cf. \eqref{eqn:Omega^pm}) in \eqref{eqn:U^0=h^pm(v*)}, and $a_0\in\mathbb R$ is the free parameter in \eqref{eqn:w^0(z) = Q(z+a_0; v^*)}. Then, for each $D>0$, the following statements hold:
\begin{enumerate}
\item For $j=1$, the equation \eqref{eqn:w^j_eqn} admits a solution $w^1\in\mathcal X$ if and only if 
\begin{align}\label{eqn:A_1}
	A_1= \frac{G_0(R_*, D)}{J'(v^*)},
\end{align}
where 
\begin{align*}
	G_0(R_*, D):= \frac{N-1}{R_*}\int_{\mathbb R} Q_z(z)^2dz-\frac{1}{D}\int_{\mathbb R}f_v(Q(z),v^*)Q(z)Q_z(z)dz.
\end{align*}
When condition \eqref{eqn:A_1} holds, the corresponding solution $w^1$ can be written explicitly as
\begin{align}\label{eqn:w^1}
	w^1 = a_1w^0_z + \widehat w^1,
\end{align} 
where $a_1$ is a free parameter and $\widehat w^1 = -(L^0)^{-1}f_1 \in [w^0_z]^{\perp}$.
\item For each $j\geq 2$, the equation \eqref{eqn:w^j_eqn} admits a solution $w^j\in\mathcal X$ if and only if 
\begin{align}\label{eqn:A_j}			
	A_j = \frac{G_{j-1}(R_*, D, a_0,\cdots, a_{j-2})}{J'(v^*)},
\end{align}
where $G_{j-1}$ is a smooth function of the variables indicated and $a_0, \cdots, a_{j-2}$ are the free parameters in the expressions of $w^0, \cdots, w^{j-2}$ respectively. When condition \eqref{eqn:A_j} holds, the corresponding solution
$w^j$ can be written explicitly as
\begin{align}\label{eqn:w^j}
	w^j = a_jw^0_z + \widehat w^j,
\end{align} 
where $a_j$ is a free parameter and $\widehat w^j = -(L^0)^{-1}f_j\in[w^0_z]^{\perp}$.
\end{enumerate}
Furthermore, for each $D_0>0$ and $j\geq 1$, there exists a constant 
$$c_j = c_j(R_*, D_0, a_0, \cdots, a_j)>0$$ such that for all $D\geq D_0$,
\begin{align}\label{eqn:w^j-Uj,w^j_z,w^j_zz}
	|w^j(z)-U^{+,j}|\cdot\chi_{[0,\infty)}(z)&+|w^j(z)-U^{-,j}|\cdot\chi_{(-\infty,0]}(z)\\
	& + |w^j_{z}(z)|+|w^j_{zz}(z)|\leq c_je^{-d_0|z|} \quad \text{ in } \mathbb R.\notag
\end{align}
\end{proposition}

Proposition \ref{pro:sov of uj_condi on A_j_aj} implies that, \textit{for any prescribed averaged total mass $M\in I_{v^*}$ and fixed $D>0$, the construction of outer and inner expansions reduces to determining the interface location parameter $R_*$ and the ``free'' parameters $a_j$ for $j\geq 0$}. Once these quantities are specified, all associated quantities---including the constants $A_j~(j\geq1)$ (see \eqref{eqn:A_1} and \eqref{eqn:A_j}), the outer expansions $U^j~(j\geq 0)$ (see \eqref{eqn:outer-u-solutions}) and the inner expansions $w^j~(j\geq0)$ (see \eqref{eqn:w^0(z) = Q(z+a_0; v^*)}, \eqref{eqn:w^1} and \eqref{eqn:w^j}), can then be determined inductively.

We now proceed to determine the interface location parameter $R_*\in(0,1)$ and the free parameters $a_j\in\mathbb R$ for $j \geq 0$ by constructing an approximate solution built from the outer expansions $U^j$ and the inner expansions $w^j$. To this end, we introduce the following shorthand notation for the outer and inner expansions truncated at order  $k\in\mathbb N$:
\begin{align}\label{eqn:U^vep_k,u^vep,k}
U^{\vep}_k(r):=\sum_{j=0}^{k}\vep^j U^{j}(r), \quad 
w^{\vep}_k(z):=\sum_{j=0}^{k}\vep^j w^j(z).
\end{align}  
Next, we select a smooth cutoff function $\theta(z)$ satisfying
\begin{gather}\label{eqn:theta}
0\leq \theta(z)\leq 1 \text { for }z\in\mathbb R, \quad
\theta(z)= 1\text { for } |z|\leq 1\text{ and } 
\theta(z)= 0\text { for }  |z|\geq 2.
\end{gather} 
Let 
\begin{align}\label{eqn:r_0}
r_0:=\frac{1}{4}\min\{R_*, 1-R_*\}.
\end{align}
We then define the $k$-th order approximate solution $u^{\vep,D}_k(r)$ by smoothly gluing the outer and inner expansions near the interface $r=R_*$ as follows:
\begin{align}\label{eqn:u_k^{vep,D}-formula}
u^{\vep,D}_k(r) := U^{\vep}_k(r) + \theta\Big(\frac{r-R_*}{r_0}\Big)\Big[ w^{\vep}_k\Big(\frac{r-R_*}{\vep}\Big) - U^{\vep}_k(r)\Big].
\end{align}

\begin{proposition}\label{pro:ajbjR_*determine}
Assume that conditions {\upshape (A1)-(A3)} hold. Let $A_0$ be defined by \eqref{eqn:A_0=v^*}. Then, for each given
$$M\in  I_{v^*}\;\mbox{ and } \; D>0 $$ 
there are constants, $R_*=R_*(M)$ defined by \eqref{eqn:R_*}, $a_0= a_0(M, D)$, $a_1= a_1(M, D, a_0)$, $\cdots$, $a_j = a_j(M, D, a_0,\cdots, a_{j-1})$ for all $j\geq 1$, which are determined inductively together with the following quantities:
\begin{itemize}
	\item the constants $A_j$ for $j\geq 1$ (cf. \eqref{eqn:A_1} and \eqref{eqn:A_j}),
	\item  the outer solutions $U^j$ for $j\geq 0$ (cf.  \eqref{eqn:outer-u-solutions}), and
	\item the inner solutions $w^j$ for $j\geq 0$ (cf. \eqref{eqn:w^0(z) = Q(z+a_0; v^*)}, \eqref{eqn:w^1} and \eqref{eqn:w^j}),
\end{itemize}
such that for any $\vep>0$ and $k\in\mathbb N$, the following expansion holds:
\begin{align}\label{eqn:mathcal S^vep=sum_A_j...}
	\mathcal S^{\vep}[u^{\vep,D}_k] =v^*+\sum_{j=1}^{k+1}\vep^j A_j + \vep^{k+2}A_{k+2}^{\vep, D},
\end{align}
where $\mathcal S^{\vep}[\,\cdot\,]$ is defined in \eqref{eqn:mathcal S^vep},  $u^{\vep,D}_k$ is given by \eqref{eqn:u_k^{vep,D}-formula}, and $A_{k+2}^{\vep, D}$ is the coefficient of the remainder term which depends only on 
$M$, $D$, $\vep$ and $k$.

Moreover, there hold:
\begin{enumerate}
	\item For each $D_0>0$ and $k\in\mathbb N$, there exists a constant $C'_k = C'_k(M, D_0)>0$ such that
\begin{align}\label{eqn:a_j+A_j unif_bdd}
	\sum_{j=0}^{k}|a_j|+\sum_{j=0}^{k+2}|A_j|\leq C'_k, \quad \forall \;D\geq D_0,
\end{align}
which further implies the uniform bounds
\begin{align}\label{eqn:w^k, U^k unif-bdd}
\sum_{j=0}^{k+2}\|U^k\|_{L^{\infty}}\leq C'_k, \quad 	\sum_{j=0}^{k}\Big(\|w^k\|_{\mathcal X}+\|w^k_{z}\|_{\mathcal X}+\|w^k_{zz}\|_{\mathcal X}\Big)\leq C'_k\quad \forall \;D\geq D_0.
\end{align}
\item For each $\vep_0 >0$, $D_0>0$ and $k\in\mathbb N$, there exists a constant $C_k = C_k(M, D_0, \vep_0 )>0$ such that
\begin{align}\label{eqn:A^vep,D_k+2 unif_bdd}
	|A_{k+2}^{\vep, D}|\leq C_k, \quad \forall \;\vep\in(0, \vep_0] \text{ and } D\geq D_0.
\end{align}
\end{enumerate}
\end{proposition}

Proposition \ref{pro:ajbjR_*determine} establishes that for each prescribed $M\in  I_{v^*}$ and $D>0$, the interface location parameter $R_*$ and the ``free'' parameters $a_j~(j\geq 0)$---along with the associated quantities $A_j$, $w^j$ and $U^j$ defined in Proposition \ref{pro:sov of uj_condi on A_j_aj} and \eqref{eqn:outer-u-solutions}---can all be determined inductively so that the asymptotic expansion \eqref{eqn:mathcal S^vep_sim_vep^jA_j}, when applied to the approximate solution $u^{\vep,D}_k$, holds to any prescribed order $k+1$ in powers of $\vep$, with a remainder term of the form $\vep^{k+2} A_{k+2}^{\vep, D}$. Moreover, the uniform bounds stated in \eqref{eqn:a_j+A_j unif_bdd}-\eqref{eqn:A^vep,D_k+2 unif_bdd} guarantee that the outer expansions $U^j$, the inner expansions $w^j$ and the approximate solutions $u_k^{\vep,D}$ remain uniformly controlled as $\vep \to 0$ and $D \to \infty$. These results provide a solid foundation for the perturbative analysis carried out in subsequent sections to construct exact solutions in a neighborhood of $u_k^{\vep,D}$.

\smallskip
We now prove Propositions \ref{pro:sov of uj_condi on A_j_aj} and \ref{pro:ajbjR_*determine}.

\begin{proof}[Proof of Proposition \ref{pro:sov of uj_condi on A_j_aj}]
We first prove Part (i) for the case $j=1$. Then equation \eqref{eqn:w^j_eqn} with $j=1$ reads $L^0w^1=-f_1$ with
\begin{align*}
	f_1= \Big(A_1-\frac{1}{D}w^0\Big) f_v(w^0, v^*) +\frac{N-1}{R_*} w^0_z.
\end{align*}
By \eqref{eqn:w^0-U0,w^0_z,w^0_zz}, we see that $f_1\in\mathcal X$ with $f_1^{\pm\infty}  = f_v(h^{\mp}(v^*), v^*)(A_1-\frac{1}{D}h^{\mp}(v^*))$.
Therefore, by Lemma \ref{lem:L^0-solvabilty} (i), to guarantee  
that \eqref{eqn:w^j_eqn} with $j=1$ is solvable in $\mathcal X$, we need to
verify that $f_1\in [w^0_z]^{\perp}$, i.e.,
\begin{align*}
	0=&\int_{\Omega}f_1(z)w^0_z(z)\,dz\\
	 =& \int_{\mathbb R} \Big(A_1-\frac{w^0(z)}{D}\Big) f_v(w^0(z), v^*)w^0_z(z)\,dz +\frac{N-1}{R_*} \int_{\mathbb R} (w^0_z)^2(z)\,dz\\
	=&-J'(v^*)A_1 -\frac{1}{D}\int_{\mathbb R}f_v(Q(z),v^*)Q(z)Q_z(z)dz+\frac{N-1}{R_*}\int_{\mathbb R} Q_z(z)^2dz, 
\end{align*}
where in the third equality, we have used the identity
\begin{align}\label{eqn:J'(v*)}
	J'(v^*) = \int_{h^{-}(v^*)}^{h^+(v^*)}f_v(s, v^*)\,ds = -\int_{\mathbb R} f_v(w^0(z), v^*)w^0_z(z)\,dz.
\end{align}
This implies that
\begin{align*}
J'(v^*)A_1=~&\displaystyle
\frac{N-1}{R_*}\int_{\mathbb R} Q_z(z)^2dz-\frac{1}{D}\int_{\mathbb R}f_v(Q(z),v^*)Q(z)Q_z(z)dz
=G_0(R_*, D).
\end{align*}
Since $J'(v^*)\neq 0$ by the condition (A3), this finishes the proof of \eqref{eqn:A_1} and determines $A_1$. Therefore, $U^{\pm,1}$ can now be determined by \eqref{eqn:outer-u-solutions}-(ii) from $A_0$ and $A_1$ as
\begin{align}\label{eqn:U^{pm,1}_solved}
	U^{\pm,1}(x) \equiv h_v^{\mp}(v^*)\Big(A_1-\frac{1}{D}h^{\mp}(v^*)\Big)   \quad  \text{ in } \Omega^\pm.
\end{align}

Denote $\widehat w^1:= -(L^0)^{-1}f_1$. Then any solution $w^1$ of \eqref{eqn:w^j_eqn} can be expressed as \eqref{eqn:w^1}. Applying Lemma \ref{lem:L^0-solvabilty} to the equation $L^0\widehat w^1+f_1=0$, we see from \eqref{eqn:w^0-U0,w^0_z,w^0_zz} and \eqref{eqn:w,w_z,w_zz<C|g|} that \eqref{eqn:w^j-Uj,w^j_z,w^j_zz} with $j=1$ holds. This implies that $w^1$ satisfies condition \eqref{eqn:bc_for_w^j} with $j=1$. We thus complete the proof of Part (i).

Next we prove Part (ii) for the case $j=2$. Then equation \eqref{eqn:w^j_eqn} with $j=2$ reads $L^0w^2=-f_2$ with
\begin{align*}
f_2& = f_v(W^0)\Big(A_2-\frac{w^1}{D}\Big)+\frac{1}{2}f_{uu}(W^0)(w^1)^2+f_{uv}(W^0)w^1\Big(A_1-\frac{w^0}{D}\Big)\\
&\quad +\frac{1}{2}f_{vv}(W^0)\Big(A_1-\frac{w^0}{D}\Big)^2+(N-1)\Big(\frac{1}{R_*}w^1_z+\frac{z}{R_*^2}w^0_z\Big),
\end{align*} 
where $W^0 := (w^0(z), v^*)$. By \eqref{eqn:w^0-U0,w^0_z,w^0_zz} and \eqref{eqn:w^j-Uj,w^j_z,w^j_zz} with $j=1$, we see that $f_2\in\mathcal X$.
Therefore, by Lemma \ref{lem:L^0-solvabilty} (i), to guarantee  
that \eqref{eqn:w^j_eqn} with $j=2$ is solvable, we need to
verify that $f_2\in [w^0_z]^{\perp}$.
Plugging $w^1 = a_1w^0_z + \widehat w^1$ from \eqref{eqn:w^1} and $A_1$ from \eqref{eqn:A_1} into the expression of $f_2$, we obtain from $f_2\in [w^0_z]^{\perp}$ that
\begin{align*}
&A_2\int_{\mathbb R}f_v(W^0)w^0_z(z)dz + \frac{(a_1)^2}{2}\int_{\mathbb R}f_{uu}(W^0)(w^0_z)^3dz \\
&+ a_1\int_{\mathbb R}\Big[\Big(f_{uu}(W^0)\widehat w^1+f_{uv}(W^0)\big(A_1-\tfrac{w^0}{D}\big)-\tfrac{1}{D}f_v(W^0)\Big)(w^0_z)^2 + \tfrac{N-1}{R_*}w^0_{zz}w^0_z\Big]dz\vspace{1ex}\\
&+ G_1(R_*, D, a_0) =0.
\end{align*}
Here all other terms not involving $A_2$ and $a_1$ have been clumped together into the last term $G_1(R_*, D, a_0)$ and we do not give an explicit formula for the function $G_1(R_*, D, a_0)$. We see from \eqref{eqn:J'(v*)} that the coefficient of $A_2$ is equal to $-J'(v^*)$. Through integrating by parts, we find that the coefficient of $(a_1)^2$ is $0$. We now compute the coefficient of $a_1$ to show that it is also $0$. Differentiating the equation $L^0\widehat w^1 + f_1=0$ with respect to $z$, we obtain that
\begin{align*}
0 & = \widehat w^1_{zzz} + f_{uu}(W^0)\widehat w^1w^0_z + f_{uv}(W^0)\Big(A_1-\frac{w^0}{D}\Big)w^0_z + f_u(W^0)\widehat w^1_z\\
& \quad - f_v(W^0)\frac{w^0_z}{D}+ \frac{N-1}{R_*} w^0_{zz}.
\end{align*} 
Therefore, we see from \eqref{eqn:w^0-U0,w^0_z,w^0_zz} and \eqref{eqn:w^j-Uj,w^j_z,w^j_zz} with $j=1$ and the fact $L^0 w^0_z =0$ that
\begin{align*}
\text{the coefficient of } a_1
& = -\int_{\mathbb R}\left[\widehat w^1_{zzz} + f_u(W^0)\widehat w^1_z \right]w^0_z\,dz\\
& = \Big({-w^0_z\widehat w^1_{zz}}+w^0_{zz}\widehat w^1_z\Big)\Big|_{-\infty}^{+\infty} -\int_{\mathbb R}\widehat w^1_z\Big(w^0_{zzz}+f_u(W^0)w^0_z\Big)\,dz\\
& = 0.
\end{align*}
Consequently, we obtain that 
$$-A_2J'(v^*) + G_1(R_*, D, a_0) =0.$$ 
This finishes the proof of \eqref{eqn:A_j} for $j=2$ and determines $A_2$. Therefore, $U^{\pm,2}$ can now be determined by \eqref{eqn:outer-u-solutions}-(iii) from $A_0, A_1$ and $A_2$. 

Denote $\widehat w^2:= -(L^0)^{-1}f_2$. Then any solution $w^2$ of \eqref{eqn:w^j_eqn} can be expressed as \eqref{eqn:w^j} with $j=2$. Applying Lemma \ref{lem:L^0-solvabilty} to the equation $L^0\widehat w^2+f_2=0$, we see from \eqref{eqn:w^0-U0,w^0_z,w^0_zz} and \eqref{eqn:w,w_z,w_zz<C|g|} that \eqref{eqn:w^j-Uj,w^j_z,w^j_zz} with $j=2$ holds. This implies that $w^2$ satisfies condition \eqref{eqn:bc_for_w^j} with $j=2$.
We thus finish the proof of Part (ii) for the case $j=2$.

The proof of Part (ii) for $j\geq 3$ is similar and follows in a inductive way. Since the computation is tedious, we omit the details.
\end{proof}

\begin{proof}[Proof of Proposition \ref{pro:ajbjR_*determine}]
It follows from \eqref{eqn:mathcal S^vep}, \eqref{eqn:A_0=v^*} and \eqref{eqn:u_k^{vep,D}-formula} that \eqref{eqn:mathcal S^vep=sum_A_j...} is equivalent to 
\begin{align}\label{eqn:M=..}
\begin{split}
M
& = v^*+ \sum\limits_{j=1}^{k+1}\vep^jA_j+\vep^{k+2} A_{k+2}^{\vep, D}+\frac{1}{|\Omega|}\left(1-\frac{\vep}{D}\right)\int_{\Omega}u^{\vep,D}_k\,dx\\
& = v^*+ \sum\limits_{j=1}^{k+1}\vep^jA_j+\vep^{k+2} A_{k+2}^{\vep, D}+\frac{1}{|\Omega|}\left(1-\frac{\vep}{D}\right)\sum\limits_{j=0}^{k}\vep^j\int_{\Omega} U^{j}(x)\,dx\\
&\quad +\frac{1}{|\Omega|}\left(1-\frac{\vep}{D}\right)\sum\limits_{j=0}^{k}\vep^j\int_{\Omega}\theta\Big(\frac{|x|-R_*}{r_0}\Big)\Big[ w^j\Big(\frac{|x|-R_*}{\vep}\Big) - U^j(x)\Big]dx\\
& = v^*+ \sum\limits_{j=1}^{k+1}\vep^jA_j+\vep^{k+2} A_{k+2}^{\vep, D}+\frac{1}{|\Omega|}\left(1-\frac{\vep}{D}\right)\sum\limits_{j=0}^{k}\vep^j\int_{\Omega} U^{j}(x)\,dx\\
& \quad +\sum_{j=0}^{k} \vep^j K^{\vep}_j-\frac{1}{D} \sum_{j=0}^{k} \vep^{j+1} K^{\vep}_j
\end{split}
\end{align}
where 
\begin{align*}
K^{\vep}_j:= \frac{1}{|\Omega|}\int_\Omega\theta\Big(\frac{|x|-R_*}{r_0}\Big)\Big[ w^j\Big(\frac{|x|-R_*}{\vep}\Big) - U^j(x)\Big] dx, \quad  0\leq j\leq  k.
\end{align*}
To calculate $K^{\vep}_j$, we first divide the domain $\big({-\tfrac{R_*}{\vep}}, \tfrac{1-R_*}{\vep}\big)$
of the stretched variable $z= \tfrac{r-R_*}{\vep}$ into the following three subintervals:
\begin{align}\label{eqn:I_1,2,3}
\begin{split}
I_1 &:= \big({-\tfrac{R_*}{\vep}},  -\tfrac{2r_0}{\vep}\big]\cup \big[\tfrac{2r_0}{\vep}, \tfrac{1-R_*}{\vep}\big);\\
I_2 &:= \big({-\tfrac{2r_0}{\vep}}, -\tfrac{r_0}{\vep}\big)\cup \big(\tfrac{r_0}{\vep},\tfrac{2r_0}{\vep}\big);\\
I_3 &:= \big[{-\tfrac{r_0}{\vep}}, \tfrac{r_0}{\vep}\big].
\end{split}
\end{align}
Moreover, we denote
\begin{align}\label{eqn:I_4, C_m^n}
		I_4:=\big({-\infty}, -\tfrac{2r_0}{\vep}\big]\cup\big[\tfrac{2r_0}{\vep},\infty\big), \quad 
		C_{N-1}^m := \tfrac{(N-1)!}{m!(N-1-m)!},
\end{align}
and  
\begin{align*}
	\tilde U^j(z) := 
		U^{+,j}\chi_{(0, \infty)} + U^{-,j}\chi_{({-\infty},0)}
\end{align*}
where $U^{\pm,j}$ are the constants defined in \eqref{eqn:outer-u-solutions} for all $j\geq 0$ (cf. the comment after \eqref{eqn:outer-u-solutions}). 
Then for $j=0, \cdots, k$, by direct calculation and \eqref{eqn:theta}, we see that
\begin{align}\label{eqn:K_j}
\begin{split}
K^{\vep}_j& = N\int_{-\tfrac{2r_0}{\vep}}^{\tfrac{2r_0}{\vep}}\theta\big(\tfrac{\vep }{r_0}z\big)\big(w^j(z) - \tilde U^j(z)\big)(R_*+\vep z)^{N-1}\vep\,dz\\
& = N\sum_{m=0}^{N-1}C_{N-1}^m\vep^{m+1} R_*^{N-1-m}\int_{\mathbb R}z^{m}\big(w^j(z) - \tilde U^j(z)\big)dz\\
& \quad -N\sum_{m=0}^{N-1}C_{N-1}^m\vep^{m+1} R_*^{N-1-m}\int_{I_4}z^{m}\big(w^j(z) - \tilde U^j(z)\big)dz\\
& \quad + N\int_{I_2}\left[\theta\left(\tfrac{\vep }{r_0}z\right)-1\right]\big(w^j(z) - \tilde U^j(z)\big)(R_*+\vep z)^{N-1}\vep\,dz\\
& = \sum_{m=0}^{N-1} \vep^{m+1}K_j^m+K^{\vep}_{j,1}+K^{\vep}_{j,2},
\end{split}
\end{align}	
where
\begin{align}
	K_j^m &:= N C_{N-1}^m R_*^{N-1-m}\int_{\mathbb R}z^{m}\big(w^j(z) - \tilde U^j(z)\big)dz   \quad (0\leq m\leq N-1),\label{eqn:K_j^m}\\
	K^{\vep}_{j,1} &:= -N\sum_{m=0}^{N-1}C_{N-1}^m\vep^{m+1} R_*^{N-1-m}\int_{I_4}z^{m}\big(w^j(z) - \tilde U^j(z)\big)dz, \label{eqn:K_j,1^vep}\\
	K^{\vep}_{j,2} &:= N\int_{I_2}\left[\theta\left(\tfrac{\vep }{r_0}z\right)-1\right]\big(w^j(z) - \tilde U^j(z)\big)(R_*+\vep z)^{N-1}\vep\,dz.\label{eqn:K_j,2^vep}
\end{align}
Note that $K_j^m$ do not depend on $\vep$.
Plugging \eqref{eqn:K_j} and \eqref{eqn:U^0=h^pm(v*)} into \eqref{eqn:M=..}, we obtain
\begin{align}\label{eqn:M =sum B_jvep^j}
M = \sum_{j=0}^{k+1}\vep^jB_j +\vep^{k+2}B_{k+2}^{\vep, D},
\end{align}
where 
\begin{align}\label{eqn:B_j}
\begin{split}
B_0 &:=v^* + \Big[h^+(v^*)R_*^N+h^-(v^*)(1-R_*^N)\Big],\\
B_1 &:=  A_1 + \Big(U^{-,1} -\tfrac{1}{D}h^+(v^*) \Big)R_*^N+\Big(U^{+,1}-\tfrac{1}{D}h^-(v^*)\Big)(1-R_*^N)+K_0^0,\\
B_j  &:= A_j + \Big(U^{-,j} -\tfrac{1}{D}U^{-,j-1}\Big)R_*^N+\Big(U^{+,j}-\tfrac{1}{D}U^{+,j-1}\Big)(1-R_*^N)\\
& \qquad +\sum\limits_{m=(j-1-k)_+}^{(N-1)\wedge (j-1)}K^m_{j-1-m}-\frac{1}{D}\sum\limits_{m= (j-2-k)_+}^{(N-1)\wedge(j-2)}K^m_{j-2-m} \quad\quad\quad (2\leq j\leq k),\\
B_{k+1} &:= A_k -\frac{1}{D}\left[U^{-,k}R_*^N+U^{+,k}(1-R_*^N)\right] +\sum\limits_{m=0}^{(N-1)\wedge k }K^m_{k-m}\\
& \qquad -\frac{1}{D}\sum\limits_{m=0}^{(N-1)\wedge(k-1) }K^m_{k-1-m},\\
B_{k+2}^{\vep, D} &:= A_{k+2}^{\vep, D}
+\sum\limits_{j=k-m+1}^{k}\sum\limits_{m=1}^{(N-1)\wedge(k+1)}\vep^{j+m-k-1}K_j^m
+\sum\limits_{j=0}^{k}\vep^{j-k-2}(K^{\vep}_{j,1}+K^{\vep}_{j,2})\\
& \qquad -\frac{1}{D}\sum\limits_{j=k-m}^{k}\sum\limits_{m=0}^{(N-1)\wedge k}\vep^{j+m-k}K_j^m-\frac{1}{D}\sum\limits_{j=0}^{k}\vep^{j-k-1}(K^{\vep}_{j,1}+K^{\vep}_{j,2}),
\end{split}
\end{align}
and we have used the notations
\begin{align*}
	k\wedge \ell:=\min\{k,\ell\},  \quad k\vee \ell:=\max\{k,\ell\}, \quad \ell_+ = \ell\vee 0.
\end{align*} 
By equating the coefficients of like powers of $\vep$ in \eqref{eqn:M =sum B_jvep^j}, we observe that the conditions
\begin{align}\label{eqn:B_j=0}
B_0 =M, \quad 
B_j = 0 \quad (1\leq j\leq k+1) \quad \text{and} \quad  B_{k+2}^{\vep, D}=0
\end{align}
together ensure that the expansion \eqref{eqn:mathcal S^vep=sum_A_j...} holds.

\smallskip

To complete the proof, we solve the system \eqref{eqn:B_j=0} inductively by appropriately choosing:
\begin{itemize}
	\item the interface location parameter $R_*$ to enforce $B_0 = M$,
	\item the free parameters $a_{j-1}$ to satisfy $B_j = 0$ for $1 \leq j \leq k+1$, and
	\item the coefficient $A_{k+2}^{\vep, D}$ to ensure $B_{k+2}^{\vep, D} = 0$.
\end{itemize}

We begin by enforcing $B_0 = M$ in \eqref{eqn:B_j=0}. To this end, we set
\begin{align*}
R_* = R_*(M) := \Big[\frac{M-v^*-h^-(v^*)}{h^+(v^*)-h^-(v^*)}\Big]^{1/N}
\end{align*}
as defined in \eqref{eqn:R_*}. To guarantee $R_* \in (0,1)$, we require that
$$M \in \big(v^* + h^-(v^*), v^* + h^+(v^*)\big) = I_{v^*}.$$ Since $R_* = R_*(M)$ is determined by the prescribed averaged total mass $M$, it follows from \eqref{eqn:A_1} and \eqref{eqn:outer-u-solutions}-(i)\&(ii) that the constants $A_1$ and the outer expansions $U^0$, $U^1$ are thereby fully determined by $M$ (via $R_*$) and $D$.

\smallskip

We now determine the parameter $a_0$ so that the condition $B_1=0$ in \eqref{eqn:B_j=0} is satisfied. Since $R_*$, $A_1$, $U^0$ and $U^1$ have already been determined above, all terms in the expression of $B_1$ in \eqref{eqn:B_j} are known except for $K_0^0$. To compute $K_0^0$, we use its  definition in \eqref{eqn:K_j^m} and the formula for $w^0$ given by \eqref{eqn:w^0(z) = Q(z+a_0; v^*)}, yielding
\begin{align*}
\left(NR_*^{N-1}\right)^{-1} K_0^0 & = \int_{\mathbb R}\big(w^0(z) - \widetilde U^0(z)\big)dz  \\
& = \int_{-\infty}^0\big(Q(z+a_0)-h^+(v^*)\big)dz + \int_0^{\infty}\big(Q(z+a_0)-h^-(v^*) \big)dz\\
& = \int_{-\infty}^0\big(Q(z)-h^+(v^*)\big)dz + \int_0^{\infty}\big(Q(z)-h^-(v^*)\big)dz\\
 & \quad -a_0\big[h^+(v^*)-h^-(v^*)\big]\\
& = -a_0 \big[h^+(v^*)-h^-(v^*)\big]+J_0,
\end{align*}
where
\begin{align*}
J_0 :=\int_{-\infty}^0\big(Q(z)-h^+(v^*)\big)dz + \int_0^{\infty}\big(Q(z)-h^-(v^*)\big)dz
\end{align*}
is a known constant independent of $a_0$. Substituting this expression for $K_0^0$ into the formula for $B_1$ in \eqref{eqn:B_j}, and noting that $h^+(v^*)-h^-(v^*)\neq 0$, we conclude that the condition $B_1 = 0$ holds if and only if $a_0$ is given by
\begin{align*}
a_0 & = \frac{
		A_1 + \big(U^{-,1} -\frac{1}{D}h^+(v^*) \big)R_*^N +\big(U^{+,1}-\frac{1}{D}h^-(v^*)\big)(1-R_*^N)+ NR_*^{N-1}J_0
}{\big(h^+(v^*)-h^-(v^*)\big)NR_*^{N-1}}\\
& =: a_0(M, D),
\end{align*}
where $R_*=R_*(M)$ is given by \eqref{eqn:R_*}, $A_1$ by \eqref{eqn:A_1}, and $U^{\pm,0}$, $U^{\pm,1}$ by \eqref{eqn:outer-u-solutions}-(i)\&(ii). Hence, the free parameter $a_0$ is now uniquely determined. Therefore $w^0(z) = Q(z + a_0)$ is completely determined. As a result, the constants $K_0^m$ for $0\leq m\leq N-1$ can now be explicitly computed from their expressions in \eqref{eqn:K_j^m}. Furthermore, it follows from \eqref{eqn:A_j} for $j=2$ that $A_2$ can now be determined by $G_1(R_*(M), D, a_0)/J'(v^*)$. Hence, $U^{\pm,2}$ can be determined by $A_0$, $A_1$ and $A_2$ via \eqref{eqn:outer-u-solutions}-(iii). 

In addition, from the above construction together with \eqref{eqn:outer-u-solutions} and \eqref{eqn:w^0-U0,w^0_z,w^0_zz},
it follows that for each $D_0 > 0$, there exists a constant $C'_0 = C'_0(M, D_0) > 0$ such that the uniform bounds \eqref{eqn:a_j+A_j unif_bdd} and \eqref{eqn:w^k, U^k unif-bdd} hold for $k = 0$.

We now determine the parameter $a_1$ in a similar manner such that the condition $B_2=0$ in \eqref{eqn:B_j=0} is satisfied. Observe that all terms in the expression of $B_2$ in \eqref{eqn:B_j} have already been determined in earlier steps, except for $K_1^0$. Using \eqref{eqn:K_j^m} and the representation of $w^1$ in \eqref{eqn:w^1}, we compute:
\begin{align*}
\left(NR_*^{N-1}\right)^{-1} K_1^0=& \int_{\mathbb R}\big(w^1(z) - \widetilde U^1(z)\big)dz  \\
= &\int_{-\infty}^0\big(a_1w^0_z(z)+\widehat w^1(z)-U^{-,1}\big)dz + \int_0^{\infty}\big(a_1w^0_z(z)+\widehat w^1(z)-U^{-,1}\big)dz\\ 	
= &-a_1\big[h^+(v^*)-h^-(v^*)\big]+J_1,
\end{align*}
where
\begin{align*}
J_1 :=\int_{-\infty}^0\big(\widehat w^1(z)-U^{-,1}\big)dz + \int_0^{\infty}\big(\widehat w^1(z)-U^{-,1}\big)dz
\end{align*}
is a known constant, since $\widehat w^1 = -(L^0)^{-1}f_1$ (cf. Proposition \ref{pro:sov of uj_condi on A_j_aj} (i)) and $U^{\pm,1}$ are already determined. Substituting the above expression for $K_1^0$ into the formula for $B_2$ in \eqref{eqn:B_j}, and noting that $h^+(v^*)-h^-(v^*)\neq 0$, we conclude that $B_2 =0$ holds if and only if
$$a_1:=a_1(M, D, a_0),$$ 
where $a_1(M, D, a_0)$ is a smooth function determined by solving $a_1$ from the equation $B_2 =0$. Thus, the parameter $a_1$ is now uniquely determined. Once $a_1$ is known, the function $w^1(z)$ is completely specified via \eqref{eqn:w^1}. Consequently, the constants $K_1^m$ for $0\leq m\leq N-1$ can be computed explicitly using \eqref{eqn:K_j^m}. Furthermore, it follows from \eqref{eqn:A_j} for $j=3$ that $A_3$ can now be determined by $G_2(R_*(M), D, a_0, a_1)/J'(v^*)$. Hence, $U^{\pm,3}$ can be determined by $A_0$, $A_1$, $A_2$ and $A_3$ via \eqref{eqn:outer-u-solutions}-(iii). 

Similarly, from the above analysis together with \eqref{eqn:outer-u-solutions}, \eqref{eqn:w^0-U0,w^0_z,w^0_zz}, and \eqref{eqn:w^j-Uj,w^j_z,w^j_zz}, it follows that for each
$D_0>0$, there exists a constant $C'_1(M, D_0)>0$ such that the uniform bounds in \eqref{eqn:a_j+A_j unif_bdd} and \eqref{eqn:w^k, U^k unif-bdd} hold for $k=1$.

We proceed inductively to determine the parameter $a_j$ such that the condition $B_{j+1}=0$ in \eqref{eqn:B_j=0} is satisfied for each $2\leq j\leq k$. Indeed, for any such $j$, all terms in the expression of $B_{j+1}$ in \eqref{eqn:B_j} have already been determined in previous steps, except for $K_j^0$. We compute $K_j^0$ using its definition in \eqref{eqn:K_j^m} and the expression for $w^j$ in \eqref{eqn:w^j} as follows:
\begin{align*}
\left(NR_*^{N-1}\right)^{-1} K_j^0 =& \int_{\mathbb R}\big(w^j(z) - \widetilde U^j(z)\big)dz \\
= &\int_{-\infty}^0 \big(a_jw^0_z(z)+\widehat w^j(z)-U^{-,j}\big)dz + \int_0^{\infty}\big(a_jw^0_z(z)+\widehat w^j(z)-U^{-,j}\big)dz\\ 	
=&-a_j\big[h^+(v^*)-h^-(v^*)\big]+J_{j}, 
\end{align*}
where 
\begin{align*}
J_{j} :=\int_{-\infty}^0\big(\widehat w^j(z)-U^{-,j}\big)dz + \int_0^{\infty}\big(\widehat w^j(z)-U^{-,j}\big)dz
\end{align*}
is explicitly computable, as $\widehat w^j = -(L^0)^{-1}f_j$ (cf. Proposition \ref{pro:sov of uj_condi on A_j_aj} (ii)) and $U^{\pm,j}$ are determined in earlier steps. Substituting this expression for $K_j^0$ into the formula for $B_{j+1}$ in \eqref{eqn:B_j}, and noting that $h^+(v^*)-h^-(v^*)\neq 0$, we conclude that the condition $B_{j+1}=0$ uniquely determines $a_j$. We denote this solution by 
$$a_j := a_j(M, D, a_0,\cdots, a_{j-1}),$$ 
where $a_j(M, D, a_0,\cdots, a_{j-1})$ is a smooth function of the variables indicated.
 
To summarize, for each $M\in  I_{v^*}$ and $D>0$, we can recursively determine the interface location and free parameters $a_j$ as
$$R_*=R_*(M), \quad a_0 = a_0(M, D)\quad \text{and} \quad a_j = a_j(M, D, a_0,\cdots, a_{j-1}) \quad \text{ for }j\geq 1,$$
such that $B_0 =M$ and $B_{j+1} =0$ for $j\geq 0$, respectively. As a result, the inner profiles $w^j$ are determined by \eqref{eqn:w^0(z) = Q(z+a_0; v^*)}, \eqref{eqn:w^1} and \eqref{eqn:w^j}, the constants $A_{j+1}$ are determined via \eqref{eqn:A_1} and \eqref{eqn:A_j}, and the outer profiles $U^{\pm,j}$ follow from \eqref{eqn:outer-u-solutions} for $j\geq 0$. Each of these quantities is determined recursively and depends smoothly on the quantities determined in previous steps. Furthermore, it follows from \eqref{eqn:w^0-U0,w^0_z,w^0_zz} and Proposition \ref{pro:sov of uj_condi on A_j_aj} that for each $D_0>0$ and $k\in\mathbb N$, there exists a constant $C'_k(M, D_0)>0$ such that the uniform bounds \eqref{eqn:a_j+A_j unif_bdd} and \eqref{eqn:w^k, U^k unif-bdd} hold.

Finally, for each $k\in\mathbb N$, to ensure that the condition $B_{k+2}^{\vep, D} = 0$ in \eqref{eqn:B_j=0} holds, we see from the expression of $B_{k+2}^{\vep, D}$ in \eqref{eqn:B_j} that it suffices to set
\begin{align}\label{eqn:A_k+2^vep-EXPRESSION}
\begin{split}
 A_{k+2}^{\vep, D}& :=
	-\sum_{j=k-m+1}^{k}\sum_{m=1}^{(N-1)\wedge(k+1)}\vep^{j+m-k-1}K_j^m-\sum_{j=0}^{k}\vep^{j-k-2}(K^{\vep}_{j,1}+K^{\vep}_{j,2})\\
	&\qquad +\frac{1}{D}\sum_{j=k-m}^{k}\sum_{m=0}^{(N-1)\wedge k}\vep^{j+m-k}K_j^m+\frac{1}{D}\sum_{j=0}^{k}\vep^{j-k-1}(K^{\vep}_{j,1}+K^{\vep}_{j,2}).
\end{split}
\end{align}
Then we conclude from \eqref{eqn:M=..}, \eqref{eqn:M =sum B_jvep^j} and \eqref{eqn:B_j=0} that the expansion \eqref{eqn:mathcal S^vep=sum_A_j...} holds for all $\vep > 0$, $D > 0$, and $k \in \mathbb{N}$.

To complete the proof of the proposition, it remains to establish the uniform bound \eqref{eqn:A^vep,D_k+2 unif_bdd}. Fix $\vep_0 >0$, $D_0>0$ and $k\in\mathbb N$. In the expression of $ A_{k+2}^{\vep, D}$ in \eqref{eqn:A_k+2^vep-EXPRESSION}, the powers of
$\vep$ in the first and third summations satisfy $j+m-k-1\geq 0$ and $j+m-k\geq 0$, respectively. This, combined with the expressions for $K_j^m$ in \eqref{eqn:K_j^m} and the preceding uniform estimates, ensures that both summations are uniformly bounded by a constant $C_k=C_k(M, D_0, \vep_0)>0$ for all $\vep\in(0,\vep_0]$ and $D\geq D_0$.
To bound the second and fourth summations, observe that
\begin{align}\label{eqn:forRemark}
\sup_{\vep\in(0, \vep_0]} \vep^{j-k-i+m}\int_{I_2\cup I_4}e^{-d_0|z|}|z|^m\,dz<\infty
\end{align}
for all integers $0\leq m\leq N-1$, $0\leq j\leq k$ and $i=1,2$, where the supremum depends only on $\vep_0$, $k$ and $M$, through the interface location $R_*$ 
in \eqref{eqn:R_*} and $r_0$ in \eqref{eqn:r_0} that appears in the integration domains $I_2$ and $I_4$ (cf. \eqref{eqn:I_1,2,3} and \eqref{eqn:I_4, C_m^n}). This bound, together with the expressions for $K^{\vep}_{j,1}$ and $K^{\vep}_{j,2}$ in \eqref{eqn:K_j,1^vep}-\eqref{eqn:K_j,2^vep} and the estimates of $w^j$ in \eqref{eqn:w^0-U0,w^0_z,w^0_zz} and \eqref{eqn:w^j-Uj,w^j_z,w^j_zz}, yields the desired uniform bound \eqref{eqn:A^vep,D_k+2 unif_bdd}. This completes the proof of the proposition.
\end{proof} 

\subsection{Properties of outer expansions and inner expansions}
We first summarize properties of the $k$-th order outer expansion $U^{\vep}_k$ defined in \eqref{eqn:U^vep_k,u^vep,k}. 

\begin{corollary}\label{cor:outer-estimate} 
Assume that conditions {\upshape (A1)-(A3)} hold. 
Let $M\in I_{v^*}$ and $D_0, \vep_0 >0$ be given. Then for each $k\in\mathbb N$, both $U^{\vep}_k|_{\Omega_+}$ and $U^{\vep}_k|_{\Omega_-}$ are constant and there exists a constant $C_k=C_k(M, D_0,\vep_0 )>0$ such that for all $\vep\in(0, \vep_0]$ and $D\geq D_0$,
\begin{align*}
\begin{vmatrix}	\vep^2\Delta U^{\vep}_k+ f\Big(U^{\vep}_k,\mathcal S^\vep[u_k^{\vep,D}]-\frac{\vep}{D}U^{\vep}_k\Big)\end{vmatrix}
=\begin{vmatrix}	
f\Big(U^{\vep}_k, \mathcal S^\vep[u_k^{\vep,D}]-\frac{\vep}{D}U^{\vep}_k\Big)\end{vmatrix}
\leq C_k \vep^{k+1}\quad \text{ in }\Omega\backslash \Gamma_*, 
\end{align*}
where $U^{\vep}_k$ is defined in \eqref{eqn:U^vep_k,u^vep,k} and $u_k^{\vep,D}$ is defined in \eqref{eqn:u_k^{vep,D}-formula}.
\end{corollary}
The proof of Corollary \ref{cor:outer-estimate} follows directly from the equations for $U^j$ in \eqref{eqn:outer-u-equs} and Proposition \ref{pro:ajbjR_*determine}. Thus we omit the details.

Next, we summarize properties of the $k$-th order inner expansion $w^{\vep}_k$ defined in \eqref{eqn:U^vep_k,u^vep,k}.
\begin{corollary}\label{cor:inner-sol_est}
Assume that conditions {\upshape (A1)-(A3)} hold. Let $M\in  I_{v^*}$ and $D_0, \vep_0 >0$ be given. Then for each $k\in\mathbb N$, there exists a constant $C_k=C_k(M, D_0, \vep_0)>0$ such that for all $\vep\in(0, \vep_0]$ and $D\geq D_0$,
\begin{align*}
\begin{vmatrix}	
\big(w^{\vep}_k\big)_{zz} +\vep \frac{N-1}{R_*+\vep z} \big(w^{\vep}_k\big)_{z}+ f\big(w^{\vep}_k, \mathcal S^\vep[u_k^{\vep,D}]-\frac{\vep}{D}w^{\vep}_k \big) 
\end{vmatrix}
	\leq C_k \vep^{k+1} \quad \text{ in } \begin{pmatrix}{-\frac{2r_0}{\vep}}, \frac{2r_0}{\vep}\end{pmatrix},
\end{align*}
where $w^{\vep}_k$ is defined in \eqref{eqn:U^vep_k,u^vep,k} and $u_k^{\vep,D}$ is defined in \eqref{eqn:u_k^{vep,D}-formula}.
\end{corollary}
Note that in the above corollary, we only require the estimate to hold for $z\in\big(-\tfrac{2r_0}{\vep}, \tfrac{2r_0}{\vep}\big)$
for each $\vep\in(0, \vep_0]$, thanks to the definition of $u_k^{\vep,D}$ in \eqref{eqn:u_k^{vep,D}-formula}.
\begin{proof}[Proof of Corollary \ref{cor:inner-sol_est}]
Note that for $i=0,\cdots,k$, $w^j$ are chosen via \eqref{eqn:w^0_eqn}-\eqref{eqn:w^j_eqn} so that 
\begin{align*}
\tfrac{\partial^i}{\partial\vep^i}
\begin{bmatrix}
\big(w^{\vep}_k\big)_{zz} +\vep \frac{N-1}{R_*+\vep z} \big(w^{\vep}_k\big)_{z}+ f\big(
w^{\vep}_k, \mathcal S^\vep[u_k^{\vep,D}]-\frac{\vep}{D}w^{\vep}_k\big)
\end{bmatrix}
\Big|_{\vep=0}\equiv 0.
\end{align*}
Therefore, for each $z\in\left(-\tfrac{2r_0}{\vep}, \tfrac{2r_0}{\vep}\right)$, there exists some $\tau\in(0, \vep)$ such that
\begin{align*}
&~\big(w^{\vep}_k\big)_{zz} +\vep \tfrac{N-1}{R_*+\vep z} \big(w^{\vep}_k\big)_{z}+ f\big(
w^{\vep}_k,\mathcal S^\vep[u_k^{\vep,D}]-\tfrac{\vep}{D}w^{\vep}_k\big)\\
=&~	\tfrac{\vep^{k+1}}{(k+1)!}\tfrac{\partial^{k+1}}{\partial\vep^{k+1}}
\begin{bmatrix}
\vep \frac{N-1}{R_*+\vep z} \big(w^{\vep}_k\big)_{z} +  f\big(
w^{\vep}_k,\sum_{j=0}^{k+1}\vep^j A_j + \vep^{k+2}A_{k+2}^{\vep, D}-\frac{\vep}{D}w^{\vep}_k\big)
\end{bmatrix}
\Big|_{\vep=\tau}.
\end{align*}
Since $R_*+\vep z\in (R_*-2r_0, R_*+2r_0 )$, the conclusion follows directly from \eqref{eqn:U^vep_k,u^vep,k} and Proposition \ref{pro:ajbjR_*determine}.
\end{proof}

In the end of this subsection, we show  that the approximate solutions $u^{\vep,D}_k(r)$ defined by \eqref{eqn:u_k^{vep,D}-formula} are uniformly bounded and exhibit a sharp transition layer of width $O(\vep)$ across the interface $\Gamma_*$.

\begin{corollary}\label{cor:width_interface_u^vep,D_}
Let $M\in I_{v^*}$ be given and $u^{\vep,D}_k(r)$ be the family of smooth radially symmetric functions defined by \eqref{eqn:u_k^{vep,D}-formula} for $\vep, D>0$ and $k\in\mathbb N$. Then the following statements hold.
\begin{enumerate}
\item For each $D_0, \vep_0>0$ and $k\in\mathbb N$, there there exists a constant $C_k=C_k(M, D_0, \vep_0)>0$ such that
\begin{align}
\|u^{\vep,D}_k\|_{L^{\infty}(\Omega)}\leq C_k,  \quad \forall \;\vep\in(0,\vep_0] \text{ and }D\geq D_0.
\end{align}
\item For each $D_0>0$, $k\in\mathbb N$ and $\eta>0$, there exist two positive constants $K=K(\eta)$ and $\bar\vep = \bar\vep(M, D_0, k, \eta)>0$ such that for all $\vep\in (0, \bar\vep]$ and $D\geq D_0$,
\begin{align}\label{eqn:u^vep, D_k_OUTSIDE_vep*K}
	\begin{cases}
	|u^{\vep, D}_k(r)-h^+(v^*)|<\eta    	& \text{on }  [0, R_*-\vep K] \vspace{1ex}\\
	|u^{\vep, D}_k(r)-h^-(v^*)|<\eta  	&\text{on }  [R_*+\vep K, 1].
	\end{cases}
\end{align}
\end{enumerate}
\end{corollary}
\begin{proof}
Part (i) follows directly from \eqref{eqn:u_k^{vep,D}-formula} and \eqref{eqn:w^k, U^k unif-bdd}.
	
We now prove Part (ii). By \eqref{eqn:U^0=h^pm(v*)}, \eqref{eqn:U^vep_k,u^vep,k} and \eqref{eqn:w^k, U^k unif-bdd}, we see that there exists a constant $\bar\vep_1 = \bar\vep_1(M, D_0, k, \eta)>0$ such that for any $\vep\in (0, \bar\vep_1]$ and $D\geq D_0$, 
\begin{align*}
	|U^{\vep}_k(r)-h^+(v^*)|<\eta \; \text{on } [0, R_*] \quad \text{ and }\quad  
	|U^{\vep}_k(r)-h^-(v^*)|<\eta \; \text{on }  [R_*, 1].
\end{align*}
On the other hand, we see from \eqref{eqn:U^vep_k,u^vep,k}, \eqref{eqn:w^0-U0,w^0_z,w^0_zz} and \eqref{eqn:w^k, U^k unif-bdd} that there exist two constants $K = K(\eta)>0$ and $\bar\vep_2 = \bar\vep_2(M, D_0, k, \eta)>0$ such that
\begin{align*}
	\begin{cases}
		|w^{\vep}_k(z) - h^{+}(v^*)|<\eta   	& \text{on }  (-\infty, -K] \vspace{1ex}\\
		|w^{\vep}_k(z) - h^{-}(v^*)|<\eta  	&\text{on }  [K, \infty)
	\end{cases}\quad\forall \;\vep\in (0, \bar\vep_2]\text{ and }D\geq D_0.
\end{align*}
Now, setting $\bar\vep := \min\{\bar\vep_1, \bar\vep_2\}$, we see from \eqref{eqn:u_k^{vep,D}-formula} and the above two estimates that \eqref{eqn:u^vep, D_k_OUTSIDE_vep*K} holds. This finished the proof of the corollary.
\end{proof}

\subsection{Proof of Proposition \ref{thm:approx_sol_scalarEQ}}\label{sec:Proof of Proposition 2.3}
In the remainder of this paper, we adopt the notation $\widetilde g(z)$ to denote the stretched version of a radial function $g(r)$ (with $r = |x|$) in $\Omega$, defined by $\widetilde g(z) := g(R_* + \vep z)$ for $z \in (-\tfrac{R_*}{\vep}, \tfrac{1 - R_*}{\vep})$.

We divide $\Omega$ into the following three subregions:
\begin{align}\label{eqn:Omega_1,2,3}
\begin{split}
\Omega_1 &:= \big\{x\in\Omega : |x|\leq R_*-2r_0 \text{ or }|x|\geq R_*+2r_0\big\},\\
\Omega_2 &:= \big\{x\in\Omega :|x|\in (R_*-2r_0, R_*-r_0)\cup(R_*+r_0, R_*+2r_0)\big\},\\
\Omega_3 & := \big\{x\in\Omega : |x|\in [R_*-r_0, R_*+r_0]\big\}.
\end{split}
\end{align}
Then by \eqref{eqn:u_k^{vep,D}-formula}, $u^{\vep,D}_k(r)$ satisfies that
\begin{align}\label{eqn:u^vep,D_k_in_Omega_i}
	u^{\vep,D}_k(r) =
	\begin{cases}
		U_k^\vep(r)= \sum\limits_{j=0}^k\vep^j U^j(r) \quad  &  \text{in } \Omega_1\vspace{1ex}\\
		U^{\vep}_k(r) +\theta(\frac{r-R_*}{r_0})\big[w^{\vep}_k\left(\frac{r-R_*}{\vep}\right) -U^{\vep}_k(r)\big] &  \text{in }\Omega_2 \vspace{1ex}\\
		w_k^\vep(\frac{r-R_*}{\vep}) =\sum\limits_{j=0}^k \vep^jw^j(\frac{r-R_*}{\vep}) \quad  &  \text{in } \Omega_3.
	\end{cases}
\end{align}
Recall that $I_i$ for $i=1,2,3$ are defined in \eqref{eqn:I_1,2,3}. The derivative of the stretched function $\widetilde u^{\vep,D}_k(z)$ satisfies that
\begin{align}\label{eqn:tilde u^vep,D_k_z_PIECEWISE_EXP}
	(\widetilde u^{\vep,D}_k)_z =
	\begin{cases}
		0     &  \text{ in }I_1\vspace{1ex}\\
		\theta(\frac{\vep}{r_0}z)\sum\limits_{j=0}^k \vep^jw^j_z(z) +\frac{\vep}{r_0}\theta'(\frac{\vep}{r_0}z)\sum\limits_{j=0}^k\vep^j[w^j(z)-\widetilde U^j(z)] &   \text{ in } I_2\vspace{1ex}\\
		\sum\limits_{j=0}^k \vep^jw^j_z(z) & \text{ in }I_3.
	\end{cases}
\end{align}

We now prove Proposition \ref{thm:approx_sol_scalarEQ}.
\begin{proof}
The smoothness of $u^{\vep,D}_k(r)$ follows  obviously from its definition in \eqref{eqn:u_k^{vep,D}-formula}. 

Part (i) of Proposition \ref{thm:approx_sol_scalarEQ} follows directly from Corollary \ref{cor:width_interface_u^vep,D_} (ii).  

We now proceed to prove Part (ii) of Proposition~\ref{thm:approx_sol_scalarEQ}. Let $\vep_0 > 0$ be a fixed constant throughout the remainder of this proof. The uniform bound in \eqref{eqn:u^vep,D_k unif.bdd} follows directly from Corollary~\ref{cor:width_interface_u^vep,D_} (i). To establish \eqref{eqn:R^vep}, we estimate the residual term $\mathcal R^\vep[u_k^{\vep,D}]$ in the $L^\infty$-norm over each subdomain $\Omega_i$ for $i = 1, 2, 3$, where the domains $\Omega_i$ are defined in \eqref{eqn:Omega_1,2,3}.

On $ \Omega_1$, it follows from \eqref{eqn:u^vep,D_k_in_Omega_i} and Corollary \ref{cor:outer-estimate} that there exists a constant $C_k = C_k(M, D_0, \vep_0)>0$ such that for all $\vep\in(0,\vep_0]$ and $D\geq D_0$,
\begin{align*}
\big\|\mathcal R^\vep[u^{\vep,D}_k]\big\|_{L^{\infty}(\Omega_1)} =\big\|
	f\big(U^{\vep}_k, \,\mathcal S^\vep[u_k^{\vep,D}]-\tfrac{\vep}{D}U^{\vep}_k\big)  \big\|_{L^{\infty}(\Omega_1)} 
\leq C_k \vep^{k+1}.
\end{align*}

On $\Omega_3$, it follows from \eqref{eqn:u^vep,D_k_in_Omega_i} and Corollary \ref{cor:inner-sol_est} that there exists a constant $C_k = C_k(M, D_0, \vep_0)>0$ such that for all $\vep\in(0,\vep_0]$ and $D\geq D_0$,
\begin{align*}
\big\|\mathcal R^\vep[u^{\vep,D}_k]\big\|_{L^{\infty}(\Omega_3)} \leq C_k\vep^{k+1}.
\end{align*}

For each $x\in\Omega_2$, we have
\begin{align*}
\big|
\mathcal R^\vep[u^{\vep,D}_k](x)\big|& = 
\big|
f\big(u^{\vep, D}_k(x), \, \mathcal S^{\vep}[u^{\vep, D}_k](x) -\tfrac{\vep}{D} u^{\vep, D}_k(x)\big)+\vep^2 \Delta u^{\vep, D}_k(x) 
\big|\\
& \leq
\big|
f\big(u^{\vep, D}_k(x),\,\mathcal S^{\vep}[u^{\vep, D}_k](x)-\tfrac{\vep}{D}u^{\vep, D}_k(x)\big) 
\big|\\
& \quad +
\big|
\{\theta(\frac{\vep }{r_0}z)\left[ w^{\vep}_k(z) - U^{\vep}_k(R_*+\vep z)\right]\}_{zz}
\big| \\
& \quad +\vep\big|\tfrac{N-1}{R_*+\vep z}\big|\cdot
\big|
\{\theta(\tfrac{\vep }{r_0}z)\left[ w^{\vep}_k(z) - U^{\vep}_k(R_*+\vep z)\right]\}_{z}
\big|\\
& =: K_1^\vep(x)+ K_2^\vep(z)+K_3^\vep(z),
\end{align*}
where we have used \eqref{eqn:u^vep,D_k_in_Omega_i} and  computed the derivatives in terms of the stretched variable $z= \tfrac{r-R_*}{\vep} \in I_2$ with $r = |x|$ for $K_2^\vep$ and $K_3^\vep$. We next estimate $K_i^\vep$ for $i = 1,2,3$.

For $K_1^\vep$, we see from Proposition \ref{pro:ajbjR_*determine}, Corollary \ref{cor:outer-estimate}, \eqref{eqn:u^vep,D_k_in_Omega_i}, \eqref{eqn:w^0-U0,w^0_z,w^0_zz} and \eqref{eqn:w^j-Uj,w^j_z,w^j_zz} that there exists a constant $C_k = C_k(M, D_0, \vep_0)>0$ such that for all $\vep\in(0,\vep_0]$ and $D\geq D_0$,
\begin{align*}
K_1^\vep(x)& \leq\big|
		f\big(U^{\vep}_k(x), \mathcal S^{\vep}[u^{\vep, D}_k](x)-\tfrac{\vep}{D}U^{\vep}_k(x)\big)
	\big|\\
	& \quad +\big|
		f\big(u^{\vep, D}_k(x),\mathcal S^{\vep}[u^{\vep, D}_k]  -\tfrac{\vep}{D} u^{\vep, D}_k(x)\big) - f\big(U^{\vep}_k(x), \mathcal S^{\vep}[u^{\vep, D}_k](x)-\tfrac{\vep}{D}U^{\vep}_k(x)\big)\big|\\
		& \leq C_k\vep^{k+1} + C\big(\|\nabla f\|_{L^{\infty}((-C_k,C_k)^2)}\big)\cdot\big| w^{\vep}_k\big(\tfrac{r-R_*}{\vep}\big)- U^{\vep}_k(r)\big|\\
		& \leq C_k\vep^{k+1} + C_ke^{-d_0|z|},
\end{align*}
where the estimation for the last term on the right hand side is computed in terms of the stretched variable $z= \tfrac{r-R_*}{\vep}\in I_2$ with $r = |x|$.

For $K_2^\vep$, By \eqref{eqn:U^vep_k,u^vep,k}, \eqref{eqn:w^0-U0,w^0_z,w^0_zz} and \eqref{eqn:w^j-Uj,w^j_z,w^j_zz}, we see that
\begin{align*}
K_2^\vep(z) &\leq \big|\theta''(\tfrac{\vep }{r_0}z)\tfrac{\vep^2 }{r_0^2}
\big| \cdot \big|w^{\vep}_k(z) - U^{\vep}_k(R_*+\vep z)\big| +2\big|\theta'(\tfrac{\vep }{r_0}z)\tfrac{\vep }{r_0}\big| \cdot {|(w^{\vep}_k)_z(z)|} + {|(w^{\vep}_k)_{zz}(z)|}\\
& \leq(\vep^2+\vep+1)C_ke^{-d_0|z|}.
\end{align*}
Similarly, for $K_3^\vep(z)$, by direct computation, we see that
\begin{align*}
K_3^\vep(z) \leq &~\tfrac{\vep(N-1)}{R_*-2r_0}\Big[
\big|\theta'\big(\tfrac{\vep }{r_0}z\big)\tfrac{\vep }{r_0}\big|\cdot|w^{\vep}_k(z) - U^{\vep}_k(R_*+\vep z)|+|\big(w^{\vep}_k\big)_{z}(z)|\Big] \leq (\vep^2+\vep)C_ke^{-d_0|z|}.
\end{align*}

Now for fixed $M\in I_{v^*}$, it is easy to see that
\begin{align*}
	\sup_{\vep\in(0,\vep_0], z\in I_2}\frac{ e^{-d_0|z|}}{\vep^{k+1}}  <\infty
\end{align*}
and the supremum depends only on $k$, $\vep_0$ and $M$ (via $R_*$ defined in \eqref{eqn:R_*} and $r_0$ defined in \eqref{eqn:r_0} that appears in the definition of the interval $I_2$ in \eqref{eqn:I_1,2,3}). Therefore, combining the above estimates for $K_i^\vep$ for $i=1,2,3$ together, we see that there exists a constant $C_k = C_k(M, D_0, \vep_0)>0$ such that for all $\vep\in(0,\vep_0]$ and $D\geq D_0$,
\begin{align*}
\big\|\mathcal R^\vep[u^{\vep,D}_k]\big\|_{L^{\infty}(\Omega_2)} \leq C_k\vep^{k+1}.
\end{align*}

Finally, combining the above estimates for $\big\|\mathcal R^\vep[u^{\vep,D}_k]\big\|_{L^{\infty}(\Omega_i)}$  for $i=1,2,3$, we complete the proof of Proposition \ref{thm:approx_sol_scalarEQ} (ii).

The first statement in Parts (iii) of Theorem \ref{thm:approx_sol_scalarEQ} follows directly 
from the definition of $\mathcal S^{\vep}[\,\cdot\,]$ in \eqref{eqn:mathcal S^vep}, while the second statement follows from \eqref{eqn:mathcal S^vep=sum_A_j...}, \eqref{eqn:a_j+A_j unif_bdd},
\eqref{eqn:A^vep,D_k+2 unif_bdd} and \eqref{eqn:u^vep,D_k unif.bdd}. This finishes the proof of the proposition.
\end{proof}

\section{Spectral analysis of the approximate solutions}
\subsection{Spectral analysis of the approximate solutions}\label{subsec:Spec_analysis}

To prove the existence of an exact solution to \eqref{eqn:MCRD_nonlocal_scalar} near the approximate solution $u_k^{\vep, D}$ constructed in Theorem \ref{thm:approx_sol_scalarEQ}, we perform a spectral analysis of the linearized operator obtained by linearizing \eqref{eqn:MCRD_nonlocal_scalar} around $u_k^{\vep, D}$. 
For this purpose, we first introduce some function spaces. For $1\leq s\leq \infty$, let
\begin{align*}
	L_{rad}^s & := \big\{u\in L^s(\Omega) \,|\, u(x) = u(|x|)\big\},\\
	W_{rad}^{2,s} & := W^{2,s}(\Omega) \cap L_{rad}^s
\end{align*}		
and 
\begin{align*}
	W_{\nu}^{2,s} &:= \big\{u\in W^{2,s}(\Omega) \,|\, \partial_{\nu} u = 0 \text{ on }\partial\Omega\big\}. 
\end{align*}
We consider the following \textit{nonlocal radial eigenvalue problem} with homogeneous Neumann boundary condition in $H_{rad}^2 := W_{rad}^{2,2}$:
\begin{align}\label{eqn:mathscr L_k^vep_EigenPro}
\mathscr L_k^{\vep}\Phi = \lambda\Phi \; \mbox{ in } \Omega,\quad \Phi_r(1)=0,
\end{align}
where the operator $\mathscr L_k^{\vep}$ is defined by
\begin{align}\label{eqn:mathscr L_k^vep_DEF}
\mathscr L_k^{\vep}\Phi
:= \vep^2\Delta\Phi + \left(f_u^{\vep,k}-\frac{\vep}{D}f_v^{\vep,k} \right)\Phi -\frac{1}{|\Omega|} \left(1-\frac{\vep}{D}\right)f_v^{\vep,k}\int_\Omega\Phi(x)\,dx 
\end{align}
with
\begin{align}\label{eqn:f_u,v^vep,k}
%\begin{split}
	f_u^{\vep,k}:= f_u\left(u_k^{\vep, D}, \, \mathcal S^{\vep}[u_k^{\vep, D}]-\frac{\vep}{D}u_k^{\vep, D}\right), \quad
	f_v^{\vep,k}:= f_v\left(u_k^{\vep, D}, \, \mathcal S^{\vep}[u_k^{\vep, D}]-\frac{\vep}{D}u_k^{\vep, D}\right)
%\end{split}
\end{align}
and $\mathcal S^{\vep}[\,\cdot\,]$ is defined in \eqref{eqn:mathcal S^vep}. 

Note that for notational simplicity, we have suppressed the explicit dependence of 
$\mathscr L_k^{\vep}$ on the parameters $D$ and $M$. The same convention will be adopted throughout the remainder of this paper for other mathematical notations that depend on these parameters.

To state our results for the eigenvalue problem \eqref{eqn:mathscr L_k^vep_EigenPro} precisely, we  introduce some shorthand notations. Recall that $R_* = R_*(M)\in(0, 1)$ is defined in \eqref{eqn:R_*}, $\Omega^- = B_{R_*}$ and $\Omega^+ = B_1\backslash\overline{B}_{R_*}$ are defined in \eqref{eqn:Omega^pm}. We then denote
\begin{align}\label{eqn:f_u^*,f_v^*}
f_u^*(x) :=
\begin{cases}
f_u(h^{+}(v^*), v^*)  \quad x\in\Omega^{-}\\
f_u(h^{-}(v^*), v^*)  \quad x\in\Omega^{+}
\end{cases}
\text{and}\quad 
f_v^*(x) := 
\begin{cases}
f_v(h^{+}(v^*), v^*)  \quad x\in\Omega^{-}\\
f_v(h^{-}(v^*), v^*)  \quad x\in\Omega^{+}.
\end{cases}
\end{align}
We further denote
\begin{align}\label{eqn:f_u^*pm,f_v^*pm}
\begin{array}{ll}
f_u^{*, -} = f^*_u|_{\Omega^-} \equiv f_u(h^{+}(v^*), v^*),&\quad  f_u^{*, +} = f^*_u|_{\Omega^+} \equiv f_u(h^{-}(v^*), v^*), \vspace{1ex}\\
f_v^{*, -} = f^*_v|_{\Omega^-} \equiv f_v(h^{+}(v^*), v^*),& \quad f_v^{*, +} = f^*_v|_{\Omega^+} \equiv f_v(h^{-}(v^*), v^*),
\end{array}
\end{align}
and
\begin{align}\label{eqn:E&G}
\begin{array}{l}
E:= f_u^{*, -} +f_u^{*, +} -R_*^Nf_v^{*, -} -(1-R_*^N) f_v^{*, +}<0,\vspace{1ex}\\
G:= f_u^{*, +}f_u^{*, -}-R_*^Nf_u^{*, +}f_v^{*, -}-(1-R_*^N)f_u^{*, -}f_v^{*, +}>0,
\end{array}
\end{align}
where the signs of $E$ and $G$ follow from the conditions (A1), (A2) and the fact that $R_*\in(0,1)$. 

We now state our main result Theorem \ref{thm:mathscr_L_k^vep} in this section, which will be used in the proof of Theorem \ref{thm:main-scalar} in Section 4. We also emphasize that all conclusions in this section hold for any integer
$k\geq1$, where $k\in\mathbb N$ is the order of the approximate solution $u^{\vep, D}_k$ defined in \eqref{eqn:u_k^{vep,D}-formula} .

\begin{theorem}\label{thm:mathscr_L_k^vep}
Assume that conditions {\upshape(A1)}-{\upshape(A3)} hold.
Let $M\in I_{v^*}$ be given.
Then, for each $D_0>0$ and integer $k\geq 1$,  there exist two constants $\bar\vep=\bar\vep(M, D_0, k)>0$ and $\lambda_*=\lambda_*(M, D_0, k)>0$ such that the following statements hold for the nonlocal radial eigenvalue problem \eqref{eqn:mathscr L_k^vep_EigenPro}:
\begin{enumerate}
\item The nonlocal operator $\mathscr L_k^{\vep}$ admits a unique eigenvalue $\lambda_0^{\vep,k}$ (called the {\upshape{critical eigenvalue}}) in the region $\mathbb C_{\lambda_*}:=\{\lambda\in\mathbb C\,|\,\re\lambda>-\lambda_*\}$ for all $\vep\in(0, \bar\vep]$ and $D\geq D_0$.
\item The {\upshape{critical eigenvalue}} $\lambda_0^{\vep,k}  \in \mathbb C_{\lambda_*}$ of \eqref{eqn:mathscr L_k^vep_EigenPro} is real, simple and
\begin{gather}\label{eqn:lambda_0^vep}
\lim_{\vep\to 0}\frac{\lambda_0^{\vep,k}}{\vep} =
\Lambda^* \quad \text{ uniformly in }D\geq D_0, 
\end{gather}
where
\begin{gather}\label{eqn:Lambda^*}
\Lambda^* := -\frac{NR_*^{N-1}(h^+(v^*) -h^-(v^*))}{m(v^*)}\cdot\frac{f_u^{*, +}f_u^{*, -}}{G}\cdot J'(v^*),
\end{gather}
$m(\,\cdot\,)$ is defined in Proposition \ref{pro:Q(z;v)} (ii) and $G$ is defined in \eqref{eqn:E&G}.   
\item Let $\Phi_0^{\vep, k}$ denote the $L^2$-normalized eigenfunction of $\mathscr L_k^{\vep}$ corresponding to the eigenvalue $\lambda_0^{\vep,k}$. Then there exists a constant $C_k=C_k(M, D_0, \bar\vep)>0$ such that for all $\vep\in(0,\bar\vep]$ and $D\geq D_0$,
\begin{align*}
	\|\Phi_0^{\vep, k}\|_{L^1(\Omega)}	 \leq C_k\sqrt{\vep},\quad \|\Phi_0^{\vep, k}\|_{L^\infty(\Omega)} \leq \frac{C_k}{\sqrt{\vep}}.
\end{align*}
\item Let $\widehat{\mathscr L}_k^\vep$ denote the $L^2$-adjoint of the nonlocal operator $\mathscr L_k^{\vep}$. Then, for all $\vep\in(0, \bar\vep]$ and $D\geq D_0$, the operator $\widehat{\mathscr L}_k^\vep$ admits a unique eigenvalue $\widehat\lambda_0^{\vep,k}$ in the region $\mathbb C_{\lambda_*}$, which is real, simple and satisfies
\begin{align*}
	\widehat\lambda_0^{\vep,k} = \lambda_0^{\vep,k}.
\end{align*}
Denote by $\widehat\Phi_0^{\vep, k}$ the eigenfunction of $\widehat{\mathscr L}_k^\vep$ corresponding to the eigenvalue $\widehat\lambda_0^{\vep,k}$. Then $\widehat\Phi_0^{\vep, k}$ can be normalized such that
$$\langle\widehat\Phi_0^{\vep, k}, \Phi_0^{\vep, k} \rangle =1,$$
and satisfies that
\begin{align*}
\|\widehat\Phi_0^{\vep, k}\|_{L^1(\Omega)} \leq C_k\sqrt{\vep},\quad \|\widehat\Phi_0^{\vep, k}\|_{L^\infty(\Omega)}\leq \frac{C_k}{\sqrt{\vep}}
\end{align*}  
for all $\vep\in(0,\bar\vep]$ and $D\geq D_0$.
\end{enumerate}
\end{theorem}
The proof of Theorem \ref{thm:mathscr_L_k^vep} will be given in Subsection \ref{sec:Proof of Proposition 3.1}.

\medskip

By elliptic regularity theory, the eigenfunction $\Phi_0^{\vep, k}$ of $\mathscr L_k^{\vep}$ belongs to $C^{2+\alpha}(\bar\Omega)$ for some $0<\alpha<1$. The same is true for $\widehat\Phi_0^{\vep, k}$.  Hence $\Phi_0^{\vep, k}$, $\widehat\Phi_0^{\vep, k} \in L_{rad}^s$ for all $1\leq s\leq\infty$. Note that it is easy to verify that $\mathscr L_k^{\vep}(W_{\nu}^{2,s}\cap L_{rad}^s)\subset L_{rad}^s$.
For each $1 < s \leq \infty$, define
\begin{align}\label{eqn:X^s}
X^s := \Big\{g\in L_{rad}^s \,|\, \langle \widehat \Phi_0^{\vep, k}, g\rangle_{s', s} = 0\Big\},
\end{align}
where $s'= \tfrac{s}{s-1}$ and $\langle \cdot,\cdot\rangle_{s,s'}$ denotes the pairing between $L^{s}(\Omega)$ and $L^{s'}(\Omega)$. Set
\begin{align*}
	X_{\nu}^{2,s} &:= W_{\nu}^{2,s}\cap X^s.
\end{align*}
Since the resolvent of $\mathscr L_k^{\vep}$ is compact and Theorem \ref{thm:mathscr_L_k^vep} ensures that
$$\Ker(\mathscr L_k^{\vep} - \lambda_0^{\vep,k})\cap X^s = \{0\} \quad \forall \;s\in(1,\infty),$$
the following lemma follows directly from the Riesz-Schauder theory.

\begin{lemma}\label{lem:mathscr L_k^vep invertible}
Assume that conditions {\upshape(A1)}-{\upshape(A3)} hold.
Let $M\in I_{v^*}$ be given.
Then for each $D_0>0$ and integer $k\geq1$, there exists a constant $\bar\vep =\bar\vep(M, D_0, k)>0$ such that the closed linear operator $\mathscr L_k^{\vep} -\lambda_0^{\vep,k}$
has a bounded right inverse $(\mathscr L_k^{\vep} -\lambda_0^{\vep,k})^{-1}$ from $X^s$ to $X_{\nu}^{2,s}$, for each $1<s<\infty$, $\vep\in(0,\bar\vep]$ and $D\geq D_0$.
\end{lemma}

Now, given $g\in X^\infty\subset X^s$, the equation 
\begin{align*}
(\mathscr  L^\vep_k -\lambda_0^{\vep,k})h =g
\end{align*}
has a unique solution $h\in X_{\nu}^{2,s}$ for each $s\in(1, \infty)$ by virtue of Lemma \ref{lem:mathscr L_k^vep invertible}. By similar arguments to the remark right after Lemma 2.2 in \cite{NT_1995}, we can show that
$h$ is independent of $s$. Therefore, by taking $s>(N/2)\vee 1$, we see that $h\in C^{\alpha}(\bar\Omega)$, $0<\alpha<1$, by the Sobolev embedding theorem. In particular, $h\in X^\infty$. Hence, $(\mathscr L_k^{\vep} -\lambda_0^{\vep,k})^{-1}$ may be considered as a linear mapping from $X^\infty$ into $X^\infty$. In fact, we have the following result.

\begin{lemma}\label{lem:mathscr L_k^vep inver_bdd}
Assume that conditions {\upshape(A1)}-{\upshape(A3)} hold.
Let $M\in I_{v^*}$ be given. Then for each $D_0>0$ and integer $k\geq 1$, there exists a constant $\bar\vep =\bar\vep(M, D_0, k)>0$ such that the inverse operator 
$(\mathscr L_k^{\vep} -\lambda_0^{\vep,k})^{-1}: X^\infty \to X^\infty$ exists and there exists a constant $C_k = C_k(M, D_0, \bar\vep) > 0$ such that
\begin{align}\label{eqn:mathscr L_inver_bdd}
\big\|(\mathscr L_k^{\vep} -\lambda_0^{\vep,k})^{-1}g \big\|_{L^{\infty}(\Omega)} \leq C_k\|g\|_{L^{\infty}(\Omega)} 
\end{align}
for all $g\in X^\infty$, $\vep\in(0, \bar\vep]$ and $D\geq D_0$.
\end{lemma}

Lemma \ref{lem:mathscr L_k^vep inver_bdd} will be proved in Subsection \ref{sec:Proof of Lemma 3.3}. 

\subsection{The principal eigenvalue of $\mathcal L^{\vep}_k$}
To prove Theorem \ref{thm:mathscr_L_k^vep} and Lemma \ref{lem:mathscr L_k^vep inver_bdd}, we need to first consider the following \textit{local symmetric radial eigenvalue problem} with the homogeneous Neumann boundary condition in $H_{rad}^2(\Omega)$:
\begin{align}\label{eqn:mathcal L^vep_k}
\begin{cases}
\mathcal L^{\vep}_k\phi := \vep^2\Delta\phi + \big(f_u^{\vep,k}-\frac{\vep}{D}f_v^{\vep,k}\big)\phi = \mu\phi  &\text{ in }\Omega,\\
\phi_r(1) = 0 &\text{ on }\partial\Omega,
\end{cases}
\end{align}
where $f_u^{\vep,k}$ and $f_v^{\vep,k}$ are defined in \eqref{eqn:f_u,v^vep,k}. Let 
$$\mu_0^{\vep,k}> \mu_1^{\vep,k} > \mu_2^{\vep,k}> \cdots$$
denote the sequence of (radial) eigenvalues of $\mathcal L^{\vep}_k$ with associated eigenfunctions 
$\phi_0^{\vep,k}$, $\phi_1^{\vep,k}$, $\phi_2^{\vep,k}$ and so on.

The principal eigenvalue $\mu_0^{\vep,k}$ of the operator $\mathcal L^{\vep}_k$ plays a key role in determining the spectrum of the eigenvalue problem \eqref{eqn:mathscr L_k^vep_EigenPro}. To analyze its asymptotic behavior as $\vep\to0$, we first establish two preparatory lemmas.

\begin{lemma}\label{lem:exp_decay_nearbc}
Let $W\subset\mathbb R^N$ be a bounded smooth domain. For $\delta>0$, define
$$W_{\delta} := \{x\in W\,|\,d(x)<\delta\},$$ 
where $d(x) = \dist(x, \partial W)$. Assume that $d\in C^2(\overline{W_{\delta}})$ and that $\rho\in C^2(\overline{W_{\delta}})$ satisfies
\begin{align*}
	\begin{cases}
		\vep^2\Delta\rho-c^{\vep}(x)\rho=0   &\text{ in }W_{\delta},\\
		\partial_{\nu}\rho =0          & \text{ on }\partial W,
	\end{cases}
\end{align*} 
where $c^{\vep}(x)\geq c_0>0$ for all $x\in \overline{W_{\delta}}$ and $\vep>0$, and $\nu$ is the unit outer normal on $\partial W$.
Then for each $\vep_0>0$, there exists a constant $\sigma_1 = \sigma_1(\vep_0, \|\Delta d\|_{L^{\infty}(W_{\delta})}, c_0)>0$ such that
\begin{align}\label{eqn:rho_nabla rho in W_delta}
\begin{split}
	&|\rho (x)|\leq 2 \|\rho \|_{L^{\infty}(W_{\delta})}\cdot e^{-\frac{\sigma_1(\delta-d(x))}{\vep}} \\
	&|\nabla\rho (x)|\leq \frac{C}{\vep}\|\rho \|_{L^{\infty}(W_{\delta})} \cdot e^{-\frac{\sigma_1(\delta-d(x))}{\vep}}
\end{split}\quad \forall \;x\in W_{\delta} \text{ and }\vep\in(0, \vep_0],
\end{align}
where $C>0$ is a constant depending only on $N$ and $\|c^{\vep}\|_{L^{\infty}(W_{\delta})}$.
\end{lemma}

In contrast to Lemma~\ref{lem:exp_decay_nearbc}, which addresses the decay behavior of $\rho(x)$ and $\nabla \rho(x)$ near the boundary of a domain, the following lemma concerns the decay as $x$ approaches an interior point. We formulate it on the ball $B_R\subset \mathbb R^N$ centered at the origin.

\begin{lemma}\label{lem:exp_decay_on_Ball}
Suppose that $\rho \in C^2(\overline{B}_R)$ satisfies
\begin{align*}
	\vep^2\Delta\rho -c^{\vep}(x)\rho =0   \quad \text{ in } B_R,
\end{align*} 
where $c^{\vep}(x)\geq c_0>0$ for all $x\in \overline{B}_R$ and $\vep>0$.
Then for each $\vep_0>0$, there exists a constant $\sigma_2 = \sigma_2(\vep_0, R^{-1}, c_0, N)>0$ such that
\begin{align}\label{eqn:rho_nabla rho in Ball}
	\begin{split}
		&|\rho (x)|\leq 2 \|\rho \|_{L^{\infty}(B_R)} \cdot e^{-\frac{\sigma_2(R-|x|)}{\vep}} \\
		&|\nabla\rho (x)|\leq \frac{C}{\vep}\|\rho \|_{L^{\infty}(B_R)} \cdot e^{-\frac{\sigma_2(R-|x|)}{\vep}}
	\end{split}\quad \forall \;x\in B_R\text{ and }\vep\in(0, \vep_0],
\end{align}
where $C>0$ is a constant depending only on $N$ and $\|c^{\vep}\|_{L^{\infty}(B_R)}$.
\end{lemma}

Lemmas \ref{lem:exp_decay_nearbc} and \ref{lem:exp_decay_on_Ball} will be proved in subsection \ref{subsec:proof of 2 Lems}. Since the spatial dimension $N$ is fixed, we shall henceforth not explicitly mention the dependence of subsequent quantities on $N$ when applying Lemmas~\ref{lem:exp_decay_nearbc} and~\ref{lem:exp_decay_on_Ball}.

\begin{proposition}\label{pro:L^vep_k:mu_0^vep,k=o(vep)}
Assume that conditions {\upshape(A1)}-{\upshape(A3)} hold. Let $M\in I_{v^*}$ be given. Then, for each integer $k\geq1$ and $D_0>0$, the principal eigenvalue $\mu_0^{\vep, k}$ of $\mathcal L^{\vep}_k$ satisfies
\begin{gather}\label{eqn:mu_0^vep=o(vep)}
\lim_{\vep\to0}\frac{\mu_0^{\vep, k}}{\vep} = 0 \quad  \text{ uniformly in }D\geq D_0
\end{gather}
and the $L^2$-normalized, positive, principal eigenfunction $\phi_0^{\vep, k}$ of $\mathcal L^{\vep}_k$ corresponding to $\mu_0^{\vep, k}$ satisfies
\begin{gather}\label{eqn:phi_0^vep}
\lim_{\vep\to0}\sqrt{\vep}\phi_0^{\vep, k}(R_* +\vep z) =-w^0_z(z) \sqrt{\frac{R_*^{1-N}}{N|\Omega|m(v^*)}} \quad \text{ in } C^2_{loc}(\mathbb R)
\end{gather}
uniformly in $D\geq D_0$. Furthermore, there exists a constant $\bar\vep =\bar\vep(M, D_0, k)>0$ such that for all $\vep\in(0, \bar\vep]$ and $D\geq D_0$,
\begin{align}\label{eqn:phi_0^vep_Norms_est}
\big\|\phi_0^{\vep, k}\big\|_{L^1(\Omega)} \leq C_k\sqrt{\vep}, \qquad  \big\|\phi_0^{\vep, k}\big\|_{L^\infty(\Omega)} \leq \frac{C_k}{\sqrt{\vep}},
\end{align} 
where $C_k=C_k(M, D_0,\bar\vep)>0$ is a constant depending only on the variables indicated.
\end{proposition}
\begin{proof}
Let	$M\in I_{v^*}$, $D_0>0$ and the integer $k\geq1$ be given. We first prove \eqref{eqn:mu_0^vep=o(vep)}. We claim that
\begin{align}\label{eqn:mu_0^vep>=0}
\liminf_{\vep\to0}\mu_0^{\vep, k}\geq 0 \text{ uniformly in } D\geq D_0.
\end{align}
To show this, we use the variational characterization of $\mu_0^{\vep, k}$ given by
\begin{align*}%\label{eqn:mu_0^vep var}
\displaystyle	\mu_0^{\vep, k} = -\inf\Bigg\{\frac{\displaystyle \vep^2\int_\Omega\left[|\nabla \phi|^2 - \left(f_u^{\vep,k}-\tfrac{\vep}{D}f_v^{\vep,k} \right) \phi^2\right]\,dx}{\displaystyle\|\phi\|^2_{L^2(\Omega)}}\,\Bigg|\, \phi\in H_{rad}^{1}, \phi\neq 0\Bigg\}.
\end{align*}  
Using $\phi(r) = w^0_z\big(\tfrac{r-R_*}{\vep}\big)$ as a test function, we easily see that
\begin{align*}
\mu_0^{\vep, k} \geq  -\frac{\displaystyle\int_{-\frac{R_*}{\vep}}^{\frac{1-R_*}{\vep}}\left[\big(w^0_{zz}(z)\big)^2-\left(\widetilde f_u^{\vep,k}-\tfrac{\vep}{D}\widetilde f_v^{\vep,k} \right)\left(w^0_z(z)\right)^2\right](R_*+\vep z)^{N-1}\,dz}{\displaystyle\int_{-\frac{R_*}{\vep}}^{\frac{1-R_*}{\vep}} \big(w^0_z(z)\big)^2(R_*+\vep z)^{N-1}\,dz},
\end{align*} 
where $\widetilde g(z)$ denotes the stretched version of a radial function $g(r)$ (with $r = |x|$) in $\Omega$, defined by $\widetilde g(z) = g(R_* + \vep z)$, as adopted at the beginning of subsection \ref{sec:Proof of Proposition 2.3}. Sending $\vep\to0$, we see from \eqref{eqn:w^0-U0,w^0_z,w^0_zz}, \eqref{eqn:U^vep_k,u^vep,k}-\eqref{eqn:mathcal S^vep=sum_A_j...} and the Dominated Convergence Theorem that
\begin{align*}
\displaystyle\liminf_{\vep\to0} \mu_0^{\vep, k} \geq  -\frac{\displaystyle\int_{-\infty}^{\infty}\left[\left(w^0_{zz}(z)\right)^2- f_u(w^0, v^*)\big(w^0_z(z)\big)^2\right]\,dz}{\displaystyle\int_{-\infty}^{\infty} \big(w^0_z(z)\big)^2\,dz} \text{ uniformly in } D\geq D_0. 
\end{align*}
Then we see from $L^0w^0_z = 0$ and  \eqref{eqn:w^0-U0,w^0_z,w^0_zz} that the integral in the numerator equals zero. Thus \eqref{eqn:mu_0^vep>=0} holds.

Recall that $d_0$ is a fixed constant satisfying \eqref{eqn:d_0}. Then, by Theorem \ref{thm:approx_sol_scalarEQ}, \eqref{eqn:d_0} and \eqref{eqn:f_u,v^vep,k}, there exist two constants $K>0$ and $\bar\vep = \bar\vep(M, D_0, k)>0$ such that for all $\vep\in(0, \bar\vep]$ and $D\geq D_0$,
\begin{align}\label{eqn:tilde f_u^vep<=-d_0}
	f_u^{\vep,k}(r) \leq -d_0^2 \quad \text{ on }[0, R_*-\vep K]\cup [R_*+\vep K, 1],
\end{align}
where we have omitted the dependence of $K$ and $\bar\vep$ on $d_0$ as $d_0$ is fixed. Denote by $\phi_0^{\vep, k}$ the positive eigenfunction corresponding to the eigenvalue $\mu_0^{\vep, k}$. We first normalize it such that
$$\|\phi_0^{\vep, k}\|_{L^{\infty}(\Omega)} =1.$$ 
Then $\phi_0^{\vep, k}$ satisfies
\begin{align}\label{eqn:phi_0^{vep,k}}
	\begin{cases}
   \vep^2\Delta\phi_0^{\vep, k}-\big[\mu_0^{\vep, k}- \big(f_u^{\vep,k}-\frac{\vep}{D}f_v^{\vep,k} \big)\big]\phi_0^{\vep, k} = 0  \quad &\text{ in }\Omega=B_1 ,\\
	\partial_{\nu}\phi_0^{\vep, k} = 0 \quad &\text{ on }\partial \Omega = \partial B_1.
	\end{cases}
\end{align}
By \eqref{eqn:mu_0^vep>=0}, \eqref{eqn:tilde f_u^vep<=-d_0} and Theorem \ref{thm:approx_sol_scalarEQ}, we can choose $\bar\vep$ even smaller if necessary so that for all $\vep\in(0,\bar\vep]$ and $D\geq D_0$, 
\begin{align}\label{eqn:geq d_0^2/2}
\mu_0^{\vep, k}- \big(f_u^{\vep,k}(r)-\tfrac{\vep}{D}f_v^{\vep,k}(r) \big)\geq \tfrac{d_0^2}{2}  \quad \text{ on }[0, R_*-\vep K]\cup [R_*+\vep K, 1],
\end{align}
We now apply Lemma \ref{lem:exp_decay_nearbc} to the equation for $\phi_0^{\vep, k}$ with $(\rho , c_0, W, \delta) = (\phi_0^{\vep, k}, d_0^2/2, \Omega, 1-R_*-\vep K)$ and obtain from \eqref{eqn:mu_0^vep>=0}, \eqref{eqn:tilde f_u^vep<=-d_0} and Theorem \ref{thm:approx_sol_scalarEQ} that there exist two constants $\sigma_1 = \sigma_1 (\bar\vep, 1/R_*, D_0)>0$ and $C_{k,1} = C_{k,1}(M, D_0, \bar\vep)>0$ such that for all $\vep\in(0,\bar\vep]$ and $D\geq D_0$,
\begin{align}\label{eqn:phi_0^vep,k-EST1}
\begin{cases}
\displaystyle\big|
	\phi_0^{\vep, k}(r)\big| \leq 2 e^{\sigma_1 K}e^{-\frac{\sigma_1(r-R_*)}{\vep}}
\vspace{1ex}\\
\displaystyle\big|(\phi_0^{\vep, k})_r(r)\big| \leq \tfrac{C_{k,1}}{\vep}e^{\sigma_1 K}e^{-\frac{\sigma_1(r-R_*)}{\vep}}
\end{cases}
	\quad  \text{ on } [R_*+\vep K, 1].
\end{align}
Similarly, applying Lemma \ref{lem:exp_decay_on_Ball} to equation for $\phi_0^{\vep, k}$ with $(\rho , c_0, R) = (\phi_0^{\vep, k}, d_0^2/2, R_*-\vep K)$, we obtain that there exist $\sigma_2 = \sigma_2(\bar\vep, 1/R_*, D_0)$ and $C_{k,2} = C_{k,2}(M, D_0, \bar\vep)>0$ such that for all $\vep\in(0,\bar\vep]$ and $D\geq D_0$,
\begin{align}
\label{eqn:phi_0^vep,k-EST2}
\begin{cases}
	\displaystyle \big|\phi_0^{\vep, k}(r)
		\big| \leq 2 e^{\sigma_2 K} e^{-\frac{\sigma_2(R_*-r)}{\vep}}\vspace{1ex}\\
	\displaystyle \big|(\phi_0^{\vep, k})_r(r)\big| \leq \tfrac{C_{k,2}}{\vep}e^{\sigma_2 K}e^\frac{-\sigma_2(R_*-r)}{\vep}
\end{cases}
	\quad \text{ on } [0, R_*-\vep K].
\end{align}
Therefore, setting $\sigma:=\min\{\sigma_1, \sigma_2\}$ and recalling that $\bar\vep = \bar\vep(M, D_0, k)$ and $R_*=R_*(M)$ defined in \eqref{eqn:R_*}, we see from \eqref{eqn:phi_0^vep,k-EST1} and \eqref{eqn:phi_0^vep,k-EST2} that for all $\vep\in(0, \bar\vep]$ and $D\geq D_0$,
\begin{align}\label{eqn:tilde_phi_exp_decay}
	\big|\widetilde\phi_0^{\vep, k}(z)\big| + \big|(\widetilde\phi_0^{\vep, k})_{z}(z)\big|\leq C_k e^{-\sigma|z|} \qquad \text{ in } (-\tfrac{R_*}{\vep}, -K]\cup [K, \tfrac{1-R_*}{\vep}),
\end{align}
where $C_k = C_k(M, D_0, \bar\vep)>0$ is a constant depending only on the variables indicated.
 
Clearly, $\widetilde \phi_0^{\vep, k}(z)$ satisfies
\begin{align}\label{eqn:tilde_varphi_0^vep}
\begin{cases}
(\widetilde\phi_0^{\vep, k})_{zz} +\vep\tfrac{N-1}{R_*+\vep z}(\widetilde\phi_0^{\vep, k})_{z} + \big(\widetilde f_u^{\vep,k}-\tfrac{\vep}{D}\widetilde f_v^{\vep,k}\big)	\widetilde\phi_0^{\vep, k}
= \mu_0^{\vep, k} \widetilde\phi_0^{\vep, k} \quad \text{ in }\big(-\tfrac{R_*}{\vep}, \tfrac{1-R_*}{\vep}\big),\vspace{1ex}\\
(\widetilde\phi_0^{\vep, k})_z(-\tfrac{R_*}{\vep}) = (\widetilde\phi_0^{\vep, k})_z(\tfrac{1-R_*}{\vep})=0.
\end{cases} 
\end{align}
By the Maximum Principle and \eqref{eqn:geq d_0^2/2}, we see that $\widetilde\phi_0^{\vep, k}$ can only attain its maximal value on the interval $[-K, K]$ for all $\vep\in(0,\bar\vep]$ and $D\geq D_0$. It is easy to see, by applying the Maximum Principle to the equation for $\phi_0^{\vep, k}$ in \eqref{eqn:phi_0^{vep,k}}, that $\mu_0^{\vep, k}$ must be bounded above uniformly for all $\vep\in(0,\bar\vep]$ and $D\geq D_0$. Therefore, for any sequences $\vep_n$ with $\lim_{n\to\infty}\vep_n = 0$ and $D_n$ satisfying $D_n\geq D_0$, there exists a subsequence (which we still denote by $\vep_n$ and $D_n$) such that
\begin{align*}
\lim_{n\to\infty}\mu_0^{\vep_n, k} =\mu_0^*\geq 0.
\end{align*}
Since $\big|\widetilde \phi_0^{\vep_n, k}(z)\big|\leq 1$, by using \eqref{eqn:tilde_varphi_0^vep} and \eqref{eqn:tilde_phi_exp_decay}, interpolation inequalities and a standard compactness argument, we find that, passing to another suitable subsequence of $\vep_n$ and $D_n$ if necessary, $\widetilde\phi_0^{\vep_n, k}$ converges to a nonnegative function $\widetilde\phi_0^*$ in $C^2_{loc}(\mathbb R)$ as $n\to\infty$, where $\widetilde\phi_0^*$ satisfies $|\widetilde\phi_0^*|\leq 1$ on $\mathbb R$, $ \widetilde\phi_0^*(z_*)=1$ at some point $z_*\in [-K, K]$, and
\begin{align}\label{eqn:L^0 tilde phi_0^*=...}
L^0\widetilde\phi_0^*= (\widetilde\phi_0^*)_{zz} +  f_u(w^0, v^*)\widetilde\phi_0^*= \mu_0^* \widetilde\phi_0^*\quad  
\text{ in }\mathbb R.
\end{align}
Moreover, the above conclusion combined with \eqref{eqn:tilde_phi_exp_decay} implies that $\widetilde\phi_0^*$ satisfies the following estimate:
\begin{align}\label{eqn:tilde_phi^*_exp_decay}
	\big|\widetilde\phi_0^*(z)\big| + \big|(\widetilde\phi_0^*)_{z}(z)\big|\leq C_k e^{-\sigma|z|} \quad\forall\; |z|\geq K.
\end{align}
Therefore, multiplying both sides of \eqref{eqn:L^0 tilde phi_0^*=...} by $w^0_z$, integrating over $\mathbb R$, using integration by parts and then sending $n\to\infty$, 
we see from $L^0w^0_z =0$,  \eqref{eqn:w^0-U0,w^0_z,w^0_zz} and \eqref{eqn:tilde_phi^*_exp_decay} that
\begin{align*}
\mu_0^* \int_{\mathbb R}\widetilde\phi_0^* w^0_z\,dz = 0.
\end{align*}
However, since $w^0_z<0$ and $\widetilde\phi_0^*$ is a nontrial nonnegative continuous function, this implies $\mu_0^* = 0$. Therefore,  $\widetilde\phi_0^*\in C^2(\mathbb R)$ is a bounded solution which satisfies
\begin{align*}
L^0\widetilde\phi_0^*=(\widetilde\phi_0^*)_{zz} +  f_u\big(w^0(z), v^*\big)\widetilde\phi_0^*= 0\quad  
\text{ in }\mathbb R.
\end{align*}
This equation has two linearly independent solutions, that is, 
\begin{align*}
w^0_z(z) \quad\text{and}\quad  w^0_z(z)\int_0^z \left(w^0_z(\eta)\right)^{-2}\,d\eta.
\end{align*}
The latter one goes to infinity as $|z|\to\infty$, which contradicts \eqref{eqn:tilde_phi^*_exp_decay}. Therefore $\widetilde\phi_0^* = -k_*^{-1}w^0_z$ with $k_* := \|w^0_z\|_{L^{\infty}(\mathbb R)}$. Since the original sequences $\vep_n$ and $D_n$ could be chosen arbitrarily, we conclude that
\begin{align}\label{eqn:tilde_phi_0^vep_C_loc^2-limit}
\lim_{\vep\to 0} \mu_0^{\vep, k} =0 \text{ and } \lim_{\vep\to 0} \widetilde\phi_0^{\vep, k} = - k_*^{-1}w^0_z ~~\text{ in }~~ C^2_{loc}(\mathbb R) \text{ uniformly in } D\geq D_0.
\end{align}

Denote
\begin{align}\label{eqn:f^{vep,k}}
f^{\vep,k}(x) := f\big( u^{\vep,D}_k(x),  \mathcal S^{\vep}[u_k^{\vep, D}]-\tfrac{\vep}{D}u_k^{\vep, D}(x) \big)  \quad\text{ for } x\in \Omega.
\end{align}
It follows from Theorem \ref{thm:approx_sol_scalarEQ} that $\widetilde u^{\vep,D}_k$ satisfies
\begin{align*}
(\widetilde u^{\vep,D}_k)_{zz} +\vep\tfrac{N-1}{R_*+\vep z}(\widetilde u^{\vep,D}_k)_{z} + \widetilde f^{\vep,k}(z)= \widetilde{\mathcal R}^{\vep}\quad \text{ in }\left(-\tfrac{R_*}{\vep}, \tfrac{1-R_*}{\vep}\right),
\end{align*}
where $\widetilde{\mathcal R}^{\vep}(z) := \mathcal R^{\vep}[u^{\vep,D}_k](R_*+\vep z)$.
Multiplying the above equation by $(\widetilde\phi_0^{\vep, k})_z$, integrating by parts and dividing the result by $\vep$, we obtain that
\begin{align}\label{eqn:mu_0^vep, k-identity}
\frac{\mu_0^{\vep, k}}{\vep}\int_{-\frac{R_*}{\vep}}^{\frac{1-R_*}{\vep}}\widetilde\phi_0^{\vep, k}(\widetilde u^{\vep,D}_k)_{z}\,dz
%&~ \frac{1}{\vep}(\widetilde\phi_0^{\vep, k})_{z}(\widetilde u^{\vep,D}_k)_{z}\Big|_{-\frac{R_*}{\vep}}^{\frac{1-R_*}{\vep}}+\frac{1}{\vep}\widetilde\phi_0^{\vep, k}\, \widetilde f^{\vep,k} \Big|_{-\frac{R_*}{\vep}}^{\frac{1-R_*}{\vep}} \notag\\
%\displaystyle&~-\frac{1}{\vep} \int_{-\frac{R_*}{\vep}}^{\frac{1-R_*}{\vep}}(\widetilde\phi_0^{\vep, k})_z\widetilde{\mathcal R}^{\vep}\,dz 
% +2(N-1)\int_{-\frac{R_*}{\vep}}^{\frac{1-R_*}{\vep}}\frac{(\widetilde\phi_0^{\vep, k})_{z}(\widetilde u^{\vep,D}_k)_{z}}{R_*+\vep z}\,dz \notag\\
 =&~ T_1^{\vep}+T_2^{\vep}+T_3^{\vep}+2(N-1)T_4^{\vep}.
\end{align}
where
\begin{gather*}
	T_1^{\vep} :=\frac{1}{\vep}(\widetilde\phi_0^{\vep, k})_{z}(\widetilde u_k^{\vep,D})_{z}\Big|_{-\frac{R_*}{\vep}}^{\frac{1-R_*}{\vep}}, \quad T_2^{\vep} := \frac{1}{\vep}\widetilde\phi_0^{\vep, k}\, \widetilde f^{\vep,k} \Big|_{-\frac{R_*}{\vep}}^{\frac{1-R_*}{\vep}},\\
	T_3^{\vep} := -\frac{1}{\vep} \int_{-\frac{R_*}{\vep}}^{\frac{1-R_*}{\vep}}(\widetilde\phi_0^{\vep, k})_z\widetilde{\mathcal R}^{\vep}\,dz, \quad T_4^{\vep} := \int_{-\frac{R_*}{\vep}}^{\frac{1-R_*}{\vep}}\frac{(\widetilde\phi_0^{\vep, k})_{z}(\widetilde u^{\vep,D}_k)_{z}}{R_*+\vep z}\,dz.
\end{gather*}
Note that the integral $T_4^{\vep}$ is well defined since $(\widetilde u^{\vep,D}_k)_{z}$ vanishes near the boundary of the integration interval (cf. \eqref{eqn:tilde u^vep,D_k_z_PIECEWISE_EXP}). 
\smallskip

We now estimate $T_i^{\vep}$ for $1\leq i\leq 4$ in \eqref{eqn:mu_0^vep, k-identity}. 
\begin{claim}\label{clm:T_i}
The following statements hold for $T_i^{\vep}$: 
\begin{enumerate}
	\item $T_1^{\vep}=0$ for all $\vep\in(0,\bar\vep]$ and $D\geq D_0$.
	\item  There exists a constant $C_k=C_k(M, D_0, \bar\vep)$ such that for all $\vep\in(0,\bar\vep]$ and $D\geq D_0$, 
	\begin{align}\label{eqn:T2,T3}
	|T_2^{\vep}|\leq C_k\vep^{k} \quad \text{ and }\quad |T_3^{\vep}|\leq C_k\vep^{k}.
	\end{align}
	\item $\lim\limits_{\vep\to 0} T_4^{\vep} =0$ uniformly in $D\geq D_0$.
\end{enumerate}
\end{claim}
We now prove Claim \ref{clm:T_i}. Indeed, Part (i) of Claim \ref{clm:T_i} follows from the fact that both $(\widetilde\phi_0^{\vep, k})_{z}$ and $(\widetilde u^{\vep,D}_k)_{z}$ vanish at $-\frac{R_*}{\vep}$ and $\frac{1-R_*}{\vep}$. For Part (ii), the first inequality in \eqref{eqn:T2,T3} follows from the uniform boundedness of $\widetilde\phi_0^{\vep, k}$, \eqref{eqn:f^{vep,k}}, \eqref{eqn:u^vep,D_k_in_Omega_i} and Corollary \ref{cor:outer-estimate}. The second inequality in \eqref{eqn:T2,T3} follows from
Theorem \ref{thm:approx_sol_scalarEQ} (ii), \eqref{eqn:tilde_phi_exp_decay}, \eqref{eqn:tilde_phi_0^vep_C_loc^2-limit} and the Dominated Convergence Theorem. It only remains to prove Part (iii) of Claim \ref{clm:T_i}. Using  \eqref{eqn:I_1,2,3} and \eqref{eqn:tilde u^vep,D_k_z_PIECEWISE_EXP}, we obtain
\begin{align}\label{eqn:T4}
	T_4^{\vep} =\int_{I_2\cup I_3}\frac{(\widetilde\phi_0^{\vep, k})_{z}(\widetilde u^{\vep,D}_k)_{z}}{R_*+\vep z}\,dz,
\end{align}
where $I_2 =(-\tfrac{2r_0}{\vep}, -\tfrac{r_0}{\vep})\cup(\tfrac{r_0}{\vep},\tfrac{2r_0}{\vep})$ and $I_3= [-\tfrac{r_0}{\vep}, \tfrac{r_0}{\vep}]$ are defined in \eqref{eqn:I_1,2,3}. 
By the Dominated Convergence Theorem, we see from \eqref{eqn:tilde_phi_0^vep_C_loc^2-limit}, \eqref{eqn:tilde u^vep,D_k_z_PIECEWISE_EXP} and Proposition \ref{pro:Q(z;v)} that for each $\widetilde K>K>0$ fixed, it holds that
\begin{align*}
\lim_{\vep\to 0}\int_{-\widetilde K}^{\widetilde K}\frac{(\widetilde\phi_0^{\vep, k})_{z}(\widetilde u^{\vep,D}_k)_{z}}{R_*+\vep z}\,dz= -\frac{1}{R_*k_*}\int_{-\widetilde K}^{\widetilde K} w^0_{zz} w^0_z \,dz 
= -\frac{1}{2R_*k_*}\left(w^0_z\right)^2\Big|_{-\widetilde K}^{\widetilde K}
\end{align*}
uniformly in $D\geq D_0$. On the other hand, by \eqref{eqn:w^k, U^k unif-bdd} and \eqref{eqn:tilde u^vep,D_k_z_PIECEWISE_EXP},
$(\widetilde u^{\vep,D}_k)_{z}$ is uniformly bounded on $I_2\cup I_3$ for all $\vep\in(0,\bar\vep]$ and $D\geq D_0$. Hence, it follows from \eqref{eqn:tilde_phi_exp_decay} that there exists a constant $C_k = C_k(M, D_0, \bar\vep)>0$ such that
\begin{align*}
\Big|\int_{(I_2\cup I_3)\backslash(-\widetilde K,\widetilde K)}\frac{(\widetilde\phi_0^{\vep, k})_{z}(\widetilde u^{\vep,D}_k)_{z}}{R_*+\vep z}\,dz\Big|  
\leq C_k \int_{(I_2\cup I_3)\backslash(-\widetilde K,\widetilde K)}\big|(\widetilde\phi_0^{\vep, k})_{z}\big| \,dz
\leq  C_k e^{-\sigma \widetilde K}
\end{align*}
for all $\vep\in(0,\bar\vep]$ and $D\geq D_0$. Consequently, by choosing $\widetilde K$ sufficiently large first and then sending $\vep\to0$ in \eqref{eqn:T4}, we see from Proposition \ref{pro:Q(z;v)} and the above two estimates that Part (iii) of Claim \ref{clm:T_i} holds. This completes the proof of Claim \ref{clm:T_i}.

We now estimate the right hand side of \eqref{eqn:mu_0^vep, k-identity}. Since, by \eqref{eqn:tilde u^vep,D_k_z_PIECEWISE_EXP},
\begin{align*}
\int_{-\frac{R_*}{\vep}}^{\frac{1-R_*}{\vep}}\widetilde\phi_0^{\vep, k}(\widetilde u^{\vep,D}_k)_{z}\,dz =\int_{I_2\cup I_3}\widetilde\phi_0^{\vep, k}(\widetilde u^{\vep,D}_k)_{z}\,dz,
\end{align*}
we see from \eqref{eqn:w^0-U0,w^0_z,w^0_zz}, \eqref{eqn:tilde u^vep,D_k_z_PIECEWISE_EXP}, \eqref{eqn:tilde_phi_0^vep_C_loc^2-limit}, \eqref{eqn:tilde_phi_exp_decay} and the Dominated Convergence Theorem that
\begin{align}\label{eqn:k_*^-1w_z^0L2}
\lim_{\vep\to0}	\int_{-\frac{R_*}{\vep}}^{\frac{1-R_*}{\vep}}\widetilde\phi_0^{\vep, k}(\widetilde u^{\vep,D}_k)_{z}\,dz= -k_*^{-1}\int_{\mathbb R} \left(w^0_z\right)^2\,dz\text{ uniformly in }D\geq D_0.
\end{align}
Now, sending $\vep\to0$ in \eqref{eqn:mu_0^vep, k-identity}, we conclude from Claim \ref{clm:T_i} and \eqref{eqn:k_*^-1w_z^0L2} that
\begin{align*}
\lim_{\vep\to0}\frac{\mu_0^{\vep, k}}{\vep} = 0\text{ uniformly in }D\geq D_0. 
\end{align*}
This finished the proof of \eqref{eqn:mu_0^vep=o(vep)}. 

Next, we prove \eqref{eqn:phi_0^vep}. We now rescale $\phi_0^{\vep, k}$ by a normalization constant $c_{\vep}>0$ such that $c_{\vep}\phi_0^{\vep, k}$ is $L^2$-normalized. Then by \eqref{eqn:tilde_phi_exp_decay},  \eqref{eqn:tilde_phi_0^vep_C_loc^2-limit} and similar arguments as above, we see that
\begin{align*}
	\lim_{\vep\to0}\frac{1}{\vep  N|\Omega|c^2_{\vep}} & = \lim_{\vep\to0}\int_{-\frac{R_*}{\vep}}^{\frac{1-R_*}{\vep}}\big[\widetilde\phi_0^{\vep, k}(z)\big]^2(R_*+\vep z)^{N-1}dz\\
	& = k_*^{-2}R_*^{N-1}\int_{\mathbb R}(w^0_z)^2\,dz  =k_*^{-2}R_*^{N-1}m(v^*)
\end{align*}
uniformly in $D\geq D_0$, which implies that
\begin{align*}
\lim_{\vep\to0}	\sqrt{\vep}c_{\vep} = \sqrt{\frac{k_*^2R_*^{1-N}}{N|\Omega|m(v^*)}} \quad \text{uniformly in  }D\geq D_0.
\end{align*}
This combined with \eqref{eqn:tilde_phi_0^vep_C_loc^2-limit} implies \eqref{eqn:phi_0^vep}. 

Finally, assume that now $\phi_0^{\vep, k}$ is $L^2$-normalized, then \eqref{eqn:phi_0^vep_Norms_est} follows from \eqref{eqn:phi_0^vep}, \eqref{eqn:phi_0^vep,k-EST1} and \eqref{eqn:phi_0^vep,k-EST2}. This finishes the proof of the proposition.
\end{proof}

\begin{corollary}\label{cor:phi_0^vep_integral}
Assume that conditions {\upshape(A1)}-{\upshape(A3)} hold.
Let $M\in I_{v^*}$ be given and $\phi_0^{\vep, k}$ be the $L^2$-normalized, positive, principal eigenfunction of $\mathcal L^{\vep}_k$ corresponding to $\mu_0^{\vep, k}$.
Then for each $D_0>0$ and $g\in C_{rad}(\bar\Omega)$, we have
\begin{gather}
\lim_{\vep\to0}\Big\langle g,\, \Big(\frac{\phi_0^{\vep, k}}{\sqrt{\vep}}\Big)\Big\rangle=g(R_*)\sqrt{\frac{N|\Omega|R_*^{N-1}}{m(v^*)}}\big(h^+(v^*) -h^-(v^*)\big),\vspace{1ex}	\label{eqn:COR_phi_0^vep_integral_I}\\
\displaystyle \lim_{\vep\to0}\Big\langle g,\, \Big(\frac{\phi_0^{\vep, k}}{\sqrt{\vep}}\Big)f_v^{\vep,k}\Big\rangle = g(R_*)\sqrt{\frac{N|\Omega|R_*^{N-1}}{m(v^*)}}J'(v^*)\label{eqn:COR_phi_0^vep_integral_II}
\end{gather}
uniformly in $D\geq D_0$,
where $\langle \cdot, \cdot\rangle$ denotes the standard Hilbert inner product in $L^2(\Omega)$ and 
$C_{rad}(\bar\Omega)$ denotes for the subspace of all radially symmetric functions in $C(\bar\Omega)$.
\end{corollary}
\begin{proof}
By using \eqref{eqn:phi_0^vep} and \eqref{eqn:J'(v*)}, we have
	\begin{align*}
		& \quad \Big\langle g,\, \Big(\frac{\phi_0^{\vep, k}}{\sqrt{\vep}}\Big)f_v^{\vep,k}\Big\rangle\\
		& = N|\Omega|\int_0^1 g(r)\phi_0^{\vep, k}(r)f_v^{\vep,k}(r)r^{N-1}\frac{dr}{\sqrt{\vep}}\\
		& = N|\Omega|\int_{-\frac{R_*}{\vep}}^{\frac{1-R_*}{\vep}} \sqrt{\vep} g(R_*+\vep z)\phi_0^{\vep, k}(R_*+\vep z)\widetilde f_v^{\vep,k}(R_*+\vep z)(R_*+\vep z)^{N-1}dz\\
		& \longrightarrow  g(R_*)\sqrt{\frac{N|\Omega|R_*^{N-1}}{m(v^*)}}J'(v^*)\quad \text{ as } \vep\to0 \text{ uniformly in }D\geq D_0,
	\end{align*}
where we used \eqref{eqn:J'(v*)}, \eqref{eqn:tilde u^vep,D_k_z_PIECEWISE_EXP}, \eqref{eqn:phi_0^vep}, \eqref{eqn:tilde_phi_exp_decay}, Theorem \ref{thm:approx_sol_scalarEQ} (iii) and the Dominated Convergence Theorem.
This proves \eqref{eqn:COR_phi_0^vep_integral_II}. The proof of \eqref{eqn:COR_phi_0^vep_integral_I} is similar and thus omitted.
\end{proof}

By elliptic regularity theory, 
the principal eigenfunction $\phi_0^{\vep, k}$ of $\mathcal L^{\vep}_k$ belongs to $C^{2+\alpha}(\bar\Omega)$ for some $0<\alpha<1$. Hence $\phi_0^{\vep, k} \in L_{rad}^s$ for $1\leq s\leq\infty$. Now for each $1< s\leq \infty$, define
\begin{align}\label{eqn:Y^s}
	Y^s := \big\{g\in L_{rad}^s \,|\, \langle g,\phi_0^{\vep, k}\rangle_{s,s'}  =0\big\}.
\end{align}
Set
\begin{align*}
	Y_{\nu}^{2,s} &:= 	W_{\nu}^{2,s}\cap Y^s.
\end{align*}
Since the resolvent of $\mathcal L^{\vep}_k$ is compact and Proposition \ref{pro:L^vep_k:mu_0^vep,k=o(vep)} ensures that $$\Ker(\mathcal L^{\vep}_k-\mu_0^{\vep, k})\cap Y^s = \{0\}$$ for each $s\in(1,\infty)$, the following lemma follows directly from the Riesz-Schauder theory.

\begin{lemma}\label{lem:L^vep_k -mu_0^vep Inverse}
Assume that conditions {\upshape(A1)}-{\upshape(A3)} hold.
Let $M\in I_{v^*}$ be given. 
Then the closed linear operator $\mathcal L^{\vep}_k -\mu_0^{\vep, k}$ has a bounded right inverse $(\mathcal L^{\vep}_k -\mu_0^{\vep, k})^{-1}$ from $Y^s$ to $Y_{\nu}^{2,s}$ for each $\vep>0$, $D>0$, $k\in\mathbb N$ and $1<s<\infty$.
\end{lemma}

Now, given $g\in Y^\infty$, by 
Lemma \ref{lem:L^vep_k -mu_0^vep Inverse} and similar arguments just before Lemma \ref{lem:mathscr L_k^vep inver_bdd}, we can see that the equation 
\begin{align*}
(\mathcal L^{\vep}_k -\mu_0^{\vep, k})h=g
\end{align*}
has a unique solution $h\in Y^\infty\cap C^{\alpha}(\bar\Omega)$ for some $0<\alpha<1$.  
Hence, $(\mathcal L^{\vep}_k -\mu_0^{\vep, k})^{-1}$ may be considered as a linear mapping from $Y^\infty$ into $Y^\infty$. Now we are ready to show the following result.

\begin{lemma}\label{lem:L^vep_k -mu_0^vep Inv bdd}
Assume that conditions {\upshape(A1)}-{\upshape(A3)} hold.
Let $M\in I_{v^*}$ be given.
Then for each $D_0>0$ and integer $k\geq1$, there exists a constant $\bar\vep=\bar\vep(M, D_0, k)>0$ such that
the inverse operator $(\mathcal L^{\vep}_k -\mu_0^{\vep, k})^{-1}: Y^\infty\to Y^\infty$ exists and there exists a constant $C_k = C_k(M, D_0, \bar\vep) > 0$ satisfying
\begin{align}\label{eqn:L^vep_k -mu_0^vep Inv bdd}
\big\|(\mathcal L^{\vep}_k -\mu_0^{\vep, k})^{-1}g \big\|_{L^{\infty}(\Omega)} \leq C_k\|g\|_{L^{\infty}(\Omega)} 
\end{align}
for all $g\in Y^\infty$, $\vep\in(0, \bar\vep]$ and $D\geq D_0$.
\end{lemma}

\begin{proof}
We prove \eqref{eqn:L^vep_k -mu_0^vep Inv bdd} by contradiction.  Assume that there exist two sequences of parameters $\vep_n\to 0$, $D_n\geq D_0$ and a sequence $\{g_n\}\subset Y^\infty$ such that
\begin{align*}
\|g_n\|_{L^{\infty}}  =1 \quad \text{and} \quad \big\|(\mathcal L^{\vep_n}_k -\mu_0^{\vep_n, k})^{-1}g_n \big\|_{L^{\infty}} \geq n.
\end{align*}
Denote 
\begin{align*}
\varrho_n := \frac{(\mathcal L^{\vep_n}_k -\mu_0^{\vep_n, k})^{-1}g_n}{\|(\mathcal L^{\vep_n}_k -\mu_0^{\vep_n, k})^{-1}g_n \|_{L^{\infty}} }\quad  \text{and}\quad 
q_n:= \frac{g_n}{\|(\mathcal L^{\vep_n}_k -\mu_0^{\vep_n, k})^{-1}g_n \|_{L^{\infty}} }. 
\end{align*}
Then $\varrho_n$ satisfies 
\begin{align*}
(\mathcal L^{\vep_n}_k -\mu_0^{\vep_n, k})\varrho_n = q_n
\end{align*}
with $\|\varrho_n\|_{L^{\infty}} =1$ and $\|q_n\|_{L^{\infty}} \leq \frac{1}{n}$. Since $\varrho_n\in Y^\infty\cap C^{\alpha}(\bar\Omega)$, replacing $(\varrho_n, q_n)$ with $(-\varrho_n, -q_n)$ if necessary, we see that for each $n$, there exists some $r_n\in[0,1]$ such that $\|\varrho_n\|_{L^{\infty}} = \varrho_n(r_n) = 1$.

Since $\|q_n\|_{L^{\infty}} \to 0$ and $\mu_n:=\mu_0^{\vep_n, k} = o(\vep_n)$ as $\vep_n\to 0$ by \eqref{eqn:mu_0^vep=o(vep)}, by similar arguments as in the proof of Proposition \ref{pro:L^vep_k:mu_0^vep,k=o(vep)}, we see from the equation for $\varrho_n$ that $r_n\in[R_*-\vep_n K, R_*+\vep_n K]$ for all $n$ large. We now express $r_n = R_*+\vep_n z_n$ and the equation for $\varrho_n$ in terms of the stretched variable $z = \tfrac{x-R_*-\vep_n z_n}{\vep_n}$ by $\widehat \varrho_n(z) = \varrho_n(R_*+\vep_n(z_n+z))$. Then $|z_n|\leq K$ and the equation for $\widehat \varrho_n(z)$ is
\begin{align}\label{eqn:hat_w}
&(\widehat \varrho_n)_{zz}(z) + \tfrac{\vep_n(N-1)}{R_*+\vep_n(z+z_n)}(\widehat \varrho_n)_z(z)- \mu_n \widehat \varrho_n (z)\\
&\quad +\big[f_u^{\vep_n, k}\big(R_*+\vep_n(z_n+z)\big)-\tfrac{\vep_n}{D_n} f_v^{\vep_n, k}\big(R_*+\vep_n(z_n+z)\big)\big]\widehat \varrho_n (z)= \widehat q_n(z), \notag 
\end{align}
where $\widehat q_n(z) = q_n(R_*+\vep_n(z_n+z))$. Since $\|\widehat \varrho_n\|_{L^{\infty}}=1$ and  $|z_n|\leq K$ for all $n$, the compactness argument implies that there exist subsequences, which are still denoted by $\{\vep_n\}$, $\{z_n\}$ and $\{\widehat \varrho_n\}$, such that
\begin{gather}\label{eqn:hat varrho_n->hat varrho_0}
\vep_n\to0, \quad  z_n\to z_*\in\mathbb R, 
\quad  \widehat \varrho_n\to \widehat \varrho_0
\text{ in }C^1_{loc}(\mathbb R) \text{ as }n\to\infty.
\end{gather}
Passing to the limit as $n\to\infty$ in the weak version ($H^1$-formation) of \eqref{eqn:hat_w}, and using regularity arguments, we find that $\widehat \varrho_0\in C^2(\mathbb R)$ satisfies $\widehat \varrho_0(0) =1$, $(\widehat \varrho_0)_z(0) =0$ and 
\begin{align*}
L^0_{z_*}\widehat \varrho_0:=(\widehat \varrho_0)_{zz} + f_u(w^0(z+z_*),v^*)\widehat \varrho_0 =0 \quad \text{ in } \mathbb R.
\end{align*}	
Since $\widehat \varrho_0$ is a bounded solution with $\widehat \varrho_0(0) =1$, $\widehat \varrho_0$ must be a multiple of the principal eigenfunction $w^0_z(z+z_*)$ of $L^0_{z_*}$, namely, $\widehat \varrho_0= cw^0_z(z+z_*)$ for some constant $c\neq 0$. However, since $\varrho_n\in Y^\infty$, we see from \eqref{eqn:Y^s} that
\begin{align*}
0 =&~ \frac{1}{\sqrt{\vep_n}N|\Omega|}\int_{\Omega}\varrho_n\phi_0^{\vep_n,k}dx \\
=&~ \int_{-\frac{R_*+\vep_nz_n}{\vep_n}}^{\frac{1-R_*-\vep_nz_n}{\vep_n}}\widehat \varrho_n(z)\widehat\phi_0^{\vep_n,k}(z)(R_*+\vep_nz_n+\vep_nz)^{N-1}\sqrt{\vep_n}dz,
\end{align*}
where $\widehat\phi_0^{\vep_n,k}(z) = \phi_0^{\vep_n,k}(R_*+\vep_nz_n+\vep_nz)$. Sending $n\to\infty$, it follows from \eqref{eqn:phi_0^vep}, \eqref{eqn:tilde_phi_exp_decay}, \eqref{eqn:hat varrho_n->hat varrho_0}, $\|\widehat \varrho_n\|_{L^{\infty}}=1$ and the Dominated Convergence Theorem that
\begin{align*}
0 = \int_{\mathbb R}\widehat \varrho_0w^0_z(z+z_*)\,dz = c\int_{\mathbb R}\left(w^0_z(z+z_*)\right)^2\,dz. 
\end{align*}
Therefore, $c=0$ and hence $\widehat \varrho_0\equiv0$, which contradicts to $\widehat \varrho_0(0) =1$. So we are done.
\end{proof}

Since eigenfunctions associated with $\mu_1^{\vep}$ of $\mathcal L^{\vep}_k$ belong to $Y^{\infty}$, the follow result follows directly from Lemma \ref{lem:L^vep_k -mu_0^vep Inv bdd}.
\begin{corollary}\label{cor:mu_1^vep<-mu_*}	
Assume that conditions {\upshape(A1)}-{\upshape(A3)} hold.
Let $M\in I_{v^*}$ be given and $\mu_1^{\vep,k}$ be the second eigenvalue of the operator $\mathcal L^{\vep}_k$. Then for each $D_0>0$ and integer $k\geq1$, there exists a constant $\bar\vep=\bar\vep(M, D_0, k)>0$ such that 
\begin{align}\label{eqn:mu_1^vep_leq}
\mu_1^{\vep,k} < -\mu_* \quad \forall \;\vep\in(0, \bar\vep] \text{ and } D\geq D_0,
\end{align}
where $\mu_* =\mu_*(M, D_0, k)>0$ is a constant depending only on the indicated parameters.
\end{corollary}	

\smallskip

We denote $[\phi_0^{\vep, k}]^{\perp}$ as the orthogonal complement of $\phi_0^{\vep, k}$ in $L^2_{rad}$ and $\mathcal P^\vep: L^2_{rad}\to L^2_{rad}$ as the orthogonal projection operator onto $[\phi_0^{\vep, k}]^{\perp}$. Now we are ready to show the following result.

\begin{proposition}\label{pro:L^vep_k properties}	
Assume that conditions {\upshape(A1)}-{\upshape(A3)} hold. Let $M\in I_{v^*}$, $D_0>0$ and integer $k\geq 1$ be given. Then there exists a constant $\bar\vep = \bar\vep(M, D_0, k)>0$ such that the following statements hold.
\begin{enumerate}
\item The inverse operator $(\mathcal L^{\vep}_k - \mu)^{-1}: [\phi_0^{\vep, k}]^{\perp}\to [\phi_0^{\vep, k}]^{\perp}$ exists and is  $L^2$-bounded uniformly in  $(\vep, D)\in (0,\bar\vep]\times[D_0,\infty)$ and $\mu\in\mathbb C$ with $\re\mu> -\mu_*$, where $\mu_*$ is defined as in Corollary \ref{cor:mu_1^vep<-mu_*}. Moreover, for $g\in L_{rad}^2$ and $\mu\in\mathbb C$ with $\re\mu> -\mu_*$,
\begin{align}
\|(\mathcal L^{\vep}_k - \mu)^{-1}\mathcal P^\vep g\|_{L^2(\Omega)} \leq \frac{1}{|\mu + \mu_*|}\|g\|_{L^2(\Omega)}
\end{align}
for all $\vep\in(0, \bar\vep]$ and $D \geq D_0$.
\item 
Let $F$ be a smooth function in $\mathbb R^2$. Define
\begin{align*}
	F^{\vep}_k:= F\big(u_k^{\vep, D}, \mathcal S^{\vep}[u_k^{\vep, D}]-\tfrac{\vep}{D}u_k^{\vep, D}\big),\quad   F^{*}|_{\Omega^{\pm}} := F\big(h^{\mp}(v^*), v^*\big).
\end{align*}
Denote
\begin{align}\label{eqn:hat_mu}
	\widehat\mu :=\frac{1}{2}\min\{\mu_*, -f_u^{*,-}, -f_u^{*,+}\}>0,
\end{align}
where $f_u^{*,\pm}$ are defined in \eqref{eqn:f_u^*pm,f_v^*pm}. Then for any $g\in L_{rad}^{\infty}$ and $\mu\in\mathbb C$ with $\re\mu\geq -\widehat\mu$, 
\begin{align}\label{eqn:L^vep_k-mu_inverse}
\lim_{\vep\to 0}\left[(\mathcal L^{\vep}_k - \mu)^{-1}\mathcal P^\vep (F^{\vep}_kg)\right]= \frac{F^{*}g}{f_u^*(x)-\mu}~~ \mbox{ in } L^2_{rad}
\end{align}
uniform in $D\geq D_0$ and $\mu\in \{\mu\in\mathbb C\,|\, \re \mu \geq -\widehat\mu\}$,
where $f_u^*(x)$ is defined in \eqref{eqn:f_u^*,f_v^*}.
\end{enumerate}
\end{proposition}
\begin{proof}
Part (i) follows directly from Corollary  \ref{cor:mu_1^vep<-mu_*} and the eigenfunction expansion of the resolvent operator $(\mathcal L_k^\vep-\mu)^{-1}$.    

\smallskip

For Part (ii), we refer the reader to the proof of \cite[Lemma 2.2]{NF}.
\end{proof}

\subsection{Proofs of Lemmas \ref{lem:exp_decay_nearbc} and \ref{lem:exp_decay_on_Ball}}\label{subsec:proof of 2 Lems}

\begin{proof}[Proof of Lemma \ref{lem:exp_decay_nearbc}]
We only need to find suitable barrier functions. Let 
$$\eta_1(x) = \cosh(\sigma_1 d(x)/\vep),$$ 
where $\sigma_1>0$ is a constant to be determined later. Then by direct computation, we see that
\begin{align*}
	\vep^2\Delta\eta_1 -c^{\vep}(x)\eta_1 \leq \Big[\sigma_1^2 + \vep\sigma_1 \tanh\Big(\frac{\sigma_1 d(x)}{\vep}\Big)\Delta d - c_0\Big]\eta_1 \quad  \text{ in } W_{\delta}.
\end{align*}
Since $|\tanh(\sigma_1 d(x)/\vep)|< 1$ and $\Delta d$ is uniformly bounded on $\overline{W_{\delta}}$, there exists a constant $\sigma_1 = \sigma_1\big(\vep_0, \|\Delta d\|_{L^{\infty}(W_{\delta})}, c_0\big)>0$ such that 
$$\vep^2\Delta\eta_1-c^{\vep}(x)\eta_1 \leq 0 \; \text{ in } W_{\delta}, \quad\forall \;\vep\in(0,\vep_0].$$

Let 
$$\vartheta_1(x) = \rho (x) -\frac{\eta_1(x)}{\cosh(\sigma_1 \delta/\vep)} \|\rho \|_{L^{\infty}(W_{\delta})} \;\text{ on }\overline{W_{\delta}}.$$
Then
\begin{align*}
	\begin{cases}
	\vep^2\Delta \vartheta_1 -c^{\vep}(x) \vartheta_1\geq 0  &\quad  \text{ in } W_{\delta},\\
	\vartheta_1\leq 0 &\quad  \text{ on }\partial W_{\delta}\backslash\partial W,\\
	\partial_{\nu}\vartheta_1 =0 &\quad  \text{ on }\partial W.
	\end{cases}
\end{align*}
Therefore, by the maximum principle, we see that $\vartheta_1(x)\leq 0$ in $ W_{\delta}$. Using the fact that $\frac{e^y}{2}\leq \cosh(y)\leq e^y$ for all $y\geq 0$, we obtain that
\begin{align*}
\rho(x) \leq 2\|\rho \|_{L^{\infty}(W_{\delta})} e^{-\frac{\sigma_1(\delta-d(x))}{\vep}}  \text{ in } W_{\delta}, \quad \forall \;\vep\in(0,\vep_0].
\end{align*} 
Similarly, by working with $\rho (x) + \frac{\eta_1(x)}{\cosh(\sigma_1 \delta/\vep)} \|\rho \|_{L^{\infty}(W_{\delta})}$, we can show that
\begin{align*}
	\rho (x) \geq -\frac{1}{2}\|\rho \|_{L^{\infty}(W_{\delta})} e^{-\frac{\sigma_1(\delta-d(x))}{\vep}}  \text{ in } W_{\delta}, \quad \forall \;\vep\in(0,\vep_0].
\end{align*} 
Combining the above two inequalities, we obtain the first estimate in \eqref{eqn:rho_nabla rho in W_delta}.
	
The second estimate in \eqref{eqn:rho_nabla rho in W_delta} follows from the standard pointwise interpolation inequality for derivatives. See. e.g., \cite[Chapter 3.4]{GT}.
\end{proof}

\begin{proof}[Proof of Lemma \ref{lem:exp_decay_on_Ball}]
Let $\eta_2(x) = \cosh(\sigma_2 |x|/\vep)$, where $\sigma_2>0$ is a constant to be determined later. Then by direct computation, we see that
\begin{align*}
	\vep^2\Delta\eta_2 -c^{\vep}(x)\eta_2 \leq \Big[\sigma_2^2 + \sigma_2 \tanh\Big(\frac{\sigma_2 |x|}{\vep}\Big)\cdot\frac{\vep(N-1)}{|x|} - c_0\Big]\eta_2 \quad  \text{ in } B_R.
\end{align*} 
Note that 
\begin{align*}
	\sup_{\vep\in (0, \vep_0], \sigma_2\in (0, \bar\sigma_2], x\in B_R}\tanh\Big(\frac{\sigma_2 |x|}{\vep}\Big)\frac{\vep(N-1)}{|x|}<\infty, 
\end{align*}
and the supremum depends only on $\bar\sigma_2>0$, $N$, $\vep_0$ and $R$. Therefore, there exists a constant $\sigma_2 = \sigma_2(\vep_0, R, c_0, N)>0$ such that 
$$\vep^2\Delta\eta_2-c^{\vep}(x)\eta_2 \leq 0  \text{ in }B_R , \forall \;\vep\in(0,\vep_0].$$
Let 
$$\vartheta_2(x) = \rho (x) -\frac{\eta_2(x)}{\cosh(\sigma_2 R/\vep)} \|\rho \|_{L^{\infty}(B_R)} \text{ in } B_R.$$ 
Then the rest of the proof follows from similar arguments as in the proof of Lemma \ref{lem:exp_decay_nearbc}. Hence, we omit the details.
\end{proof}

\subsection{Proof of Theorem \ref{thm:mathscr_L_k^vep}}\label{sec:Proof of Proposition 3.1}
\begin{proof} 
Let $M\in I_{v^*}$, $D_0>0$ and integer $k\geq1$ be given.

We begin by analyzing the eigenvalue of $\mathscr L_k^{\vep}$ with the largest real part.  
To this end, define
\begin{align}\label{eqn:lambda_*}
\lambda_* := \min\Big\{\widehat\mu,~ -\frac{\re(E+\sqrt{E^2-4G})}{4}\Big\}>0,
\end{align}
where $\widehat\mu>0$ is given by \eqref{eqn:hat_mu}, and $E<0$, $G>0$ are defined in \eqref{eqn:E&G}. Next, we introduce the domain
\begin{align*}
\mathbb C_{\lambda_*} :=\{\lambda\in\mathbb C\,|\, \re\lambda\geq -\lambda_* \}.
\end{align*}
Our first objective is to characterize the spectrum of 
$\mathscr L_k^{\vep}$ in the region $\mathbb C_{\lambda_*}$, in order to establish Theorem \ref{thm:mathscr_L_k^vep} (i) and (ii). 

In the sequel, we assume that $\lambda^\vep\in\mathbb C_{\lambda_*}$ is an eigenvalue of $\mathscr L_k^{\vep}$,  and let $\Phi^\vep$ be a corresponding eigenfunction.

Let $(\mu_0^{\vep, k}, \phi_0^{\vep, k})$ be the principal eigenpair of the operator $\mathcal L_k^\vep$ defined in \eqref{eqn:mathcal L^vep_k} with $\phi_0^{\vep, k}$ being $L^2$-normalized. We decompose $\Phi^\vep$ as
\begin{align}\label{eqn:decom of Phi^vep}
\Phi^\vep = a^\vep\phi_0^{\vep, k} + \eta^{\vep}\in{\rm span}\{\phi_0^{\vep, k}\}\oplus[\phi_0^{\vep, k}]^{\perp}.
\end{align}
%where $a^\vep\in\mathbb C$ and $\eta^{\vep}\in[\phi_0^{\vep, k}]^{\perp}$.
Then the equation for $\Phi^\vep$ can be reformulated as follows:
\begin{align}\label{eqn:a_and_w^vep}
\begin{array}{l}
\displaystyle	(\mu_0^{\vep, k}-\lambda^\vep)a^\vep =\frac{1}{|\Omega|} \left(1-\frac{\vep}{D}\right)\Big(a^\vep\int_{\Omega}\phi_0^{\vep, k}\,dx +\int_{\Omega}\eta^{\vep}\,dx\Big)\int_{\Omega}f_v^{\vep, k}\phi_0^{\vep, k}\,dx, \vspace{1ex}\\
\displaystyle (\mathcal L_k^\vep-\lambda^\vep)\eta^{\vep} =\frac{1}{|\Omega|}\left(1-\frac{\vep}{D}\right)\Big(a^\vep\int_{\Omega}\phi_0^{\vep, k}\,dx +\int_{\Omega}\eta^{\vep}\,dx\Big) \mathcal P^\vep f_v^{\vep, k}.
\end{array}
\end{align}
Thanks to Proposition \ref{pro:L^vep_k properties} (i) and the assumption that $\lambda^\vep\in\mathbb C_{\lambda_*}$, $\eta^{\vep}$ can be solved from the second equation in \eqref{eqn:a_and_w^vep} as
\begin{align}\label{eqn:w^vep}
\eta^{\vep} = \frac{1}{|\Omega|}\left(1-\frac{\vep}{D}\right)\Big(a^\vep\int_{\Omega}\phi_0^{\vep, k}\,dx +\int_{\Omega}\eta^{\vep}\,dx\Big)(\mathcal L_k^\vep-\lambda^\vep)^{-1}\mathcal P^\vep f_v^{\vep, k}.
\end{align}
For notational simplicity, denote
\begin{align}\label{eqn:b=int_w}
b^\vep := \int_{\Omega}\eta^{\vep}\,dx
\end{align}
and
\begin{align}\label{eqn:psi_lambda^vep}
\psi_{\lambda^\vep}^{\vep} := (\mathcal L_k^\vep-\lambda^\vep)^{-1}\mathcal P^\vep f_v^{\vep, k}\in [\phi_0^{\vep, k}]^{\perp}.
\end{align}
Integrating equation \eqref{eqn:w^vep} over $\Omega$, we obtain that
\begin{align}\label{eqn:b}
b^\vep = \frac{b^\vep}{|\Omega|}\left(1-\frac{\vep}{D}\right)\int_{\Omega}\psi_{\lambda^\vep}^{\vep}\,dx +\frac{a^\vep}{|\Omega|} \left(1-\frac{\vep}{D}\right)\int_{\Omega}\phi_0^{\vep, k}\,dx\int_{\Omega}\psi_{\lambda^\vep}^{\vep}\,dx.
\end{align}

We will first prove Parts (i)-(ii) of Theorem \ref{thm:mathscr_L_k^vep}.
Since the proof is rather lengthy, we structure it by isolating several key intermediate results as Claims \ref{clm:b-solvable}-\ref{clm:lambda^vep_Simple} below. Each claim establishes a crucial step in the argument, and together they lead to the proof of Theorem \ref{thm:mathscr_L_k^vep} (i)-(ii).

\begin{claim}\label{clm:b-solvable}
There exists a constant $\bar\vep =\bar\vep(M, D_0, k)>0$ such that for each $(\vep,\lambda^{\vep})\in(0, \bar\vep] \times\mathbb C_{\lambda_*}$ and $D\geq D_0$, equation \eqref{eqn:b} admits a unique solution given explicitly by
\begin{align}\label{eqn:b-2}
	b^\vep = a^\vep\cdot\frac{\left(1-\frac{\vep}{D}\right) \frac{1}{|\Omega|}\int_{\Omega}\psi_{\lambda^\vep}^{\vep}\,dx}{1 -\left(1-\frac{\vep}{D}\right) \frac{1}{|\Omega|} \int_{\Omega}\psi_{\lambda^\vep}^{\vep}\,dx}\int_{\Omega}\phi_0^{\vep, k}\,dx.
\end{align}
\end{claim}

Indeed, it follows from \eqref{eqn:psi_lambda^vep} and Proposition \ref{pro:L^vep_k properties} (ii) that 
\begin{align}\label{eqn:psi_lambda^vep_limit}
\lim_{\vep\to0}\psi_{\xi}^{\vep}(x) = 	\frac{f^*_v(x)}{f^*_u(x)-\xi}\quad \text{ in }L^2_{rad}
\end{align}
uniformly in $\xi\in\mathbb C_{\lambda_*}$ and $D\geq D_0$, where $f^*_v$ and $f^*_u$ are defined in \eqref{eqn:f_u^*,f_v^*}. Consequently,
\begin{align}\label{eqn:1-psi_lambda^vep_limit}
\lim_{\vep\to0}\Big[1 -\left(1-\frac{\vep}{D}\right) \frac{1}{|\Omega|}\int_{\Omega}\psi_{\xi}^{\vep}\,dx\Big]= \frac{\xi^2 -E\xi +G}{(f_u^{*, +}-\xi)(f_u^{*, -}-\xi)}
\end{align}
uniformly in $\xi\in\mathbb C_{\lambda_*}$ and $D\geq D_0$. By the definition of $\lambda_*$ in \eqref{eqn:lambda_*}, the right-hand side of \eqref{eqn:1-psi_lambda^vep_limit} remains uniformly bounded away from zero for all $\xi\in\mathbb C_{\lambda_*}$.
Therefore, there exists a constant $\bar\vep=\bar\vep(M, D_0, k)>0$ small such that
\begin{align*}
1 -\left(1-\frac{\vep}{D}\right) \frac{1}{|\Omega|}\int_{\Omega}\psi_{\xi}^{\vep}\,dx\neq 0 \quad \forall\; (\vep,\xi)\in(0, \bar\vep] \times\mathbb C_{\lambda_*}\text{ and } D\geq D_0.
\end{align*}
This ensures that $b^\vep$ is uniquely solvable from equation \eqref{eqn:b} and is given by \eqref{eqn:b-2}, which concludes the proof of Claim \ref{clm:b-solvable}. 

\begin{claim}\label{clm:a neq 0}
Let $\vep\in(0, \bar\vep]$ and $D\geq D_0$. Then for $\lambda^{\vep}\in \mathbb C_{\lambda_*}$ to be an eigenvalue of $\mathscr L_k^{\vep}$, it is necessary that $a^\vep\neq 0$ in \eqref{eqn:decom of Phi^vep}.
\end{claim}
Suppose for contradiction that Claim \ref{clm:a neq 0} is false, i.e., $a^\vep=0$. Then we see from \eqref{eqn:b-2} that $b^\vep = 0$. This combined with \eqref{eqn:w^vep} and \eqref{eqn:b=int_w} implies that $\eta^{\vep} =0$ and thus $\Phi^\vep =0$, which is a contradiction since $\Phi^\vep$ is an eigenfunction of $\mathscr L_k^{\vep}$. This finishes the proof of Claim \ref{clm:a neq 0}.

\medskip

By Claim \ref{clm:a neq 0}, we may, without loss of generality, set
$$a^\vep=1$$
in the remainder of this proof. Substituting $a^\vep=1$, \eqref{eqn:b=int_w} and the expression \eqref{eqn:b-2} into the first equation of \eqref{eqn:a_and_w^vep}, we obtain:
\begin{align}\label{eqn:lambda^vep=mu^vep-...}
	\displaystyle
\lambda^\vep = \mu_0^{\vep, k} - \frac{1}{|\Omega|} \left(1-\frac{\vep}{D}\right)\frac{\int_{\Omega}\phi_0^{\vep, k}\,dx}{1 -\left(1-\frac{\vep}{D}\right) \frac{1}{|\Omega|}\int_{\Omega}\psi_{\lambda^\vep}^{\vep}\,dx}\int_{\Omega}f_v^{\vep, k}\phi_0^{\vep, k}\,dx.
\end{align}

\begin{claim}\label{clm:lambda/vep->Lambda*}
\begin{align}\label{eqn:lambda/vep=Lambda*}
	\lim_{\vep\to0}\frac{\lambda^\vep}{\vep} =\Lambda^*\in\mathbb R\backslash \{0\} \quad \text{ uniformly in }D\geq D_0,
\end{align}
where $\Lambda^*$ is defined in \eqref{eqn:Lambda^*}. 
\end{claim}

We now prove Claim
\ref{clm:lambda/vep->Lambda*}. Let $\bar\vep$ be given as in Claim \ref{clm:b-solvable}. We first show that, by choosing $\bar\vep$ even smaller if necessary, there exists a constant $C_k = C_k(M, D_0, \bar\vep)>0$ such that
\begin{align}\label{eqn:lambda^vep<C_kvep}
	|\lambda^\vep| \leq C_k\,\vep \quad\text{ for all } \vep\in(0,\bar\vep] \text{ and } D\geq D_0.
\end{align}
We now prove \eqref{eqn:lambda^vep<C_kvep}.
Since the RHS of \eqref{eqn:1-psi_lambda^vep_limit} is a rational function of $\xi$ whose numerator and denominator have the same degree, its modulus is uniformly bounded from above and below by two positive constants in the region $\mathbb C_{\lambda_*}$, due to our choice of $\lambda_*$ in \eqref{eqn:lambda_*}. Therefore, choosing $\bar\vep$ even smaller if necessary, we see from \eqref{eqn:1-psi_lambda^vep_limit} that there exists two constants $\bar c_1, \bar c_2>0$, depending only on $M$, $D_0$, $k$ and $\bar\vep$, such that 
\begin{align}\label{eqn:1-int_psi_lamba bdd}
	\bar c_1\leq \Big| 1 -\left(1-\frac{\vep}{D}\right) \frac{1}{|\Omega|}\int_{\Omega}\psi_{\lambda^\vep}^{\vep}\,dx\Big| \leq \bar c_2\quad \forall \;(\vep,\lambda^\vep)\in (0,\bar\vep]\times\mathbb C_{\lambda_*}\text{ and }D\geq D_0.
\end{align}
On the other hand, we see from Corollary \ref{cor:phi_0^vep_integral} that
\begin{gather}
	\lim_{\vep\to0}\int_{\Omega}\frac{1}{\sqrt{\vep}}\phi_0^{\vep, k}\,dx = \sqrt{\frac{N|\Omega|R_*^{N-1}}{m(v^*)}} \big(h^+(v^*) -h^-(v^*)\big),\vspace{1ex} \label{eqn:phi_0^vep_integral}\\
	\lim_{\vep\to0}\int_{\Omega}\frac{1}{\sqrt{\vep}}f_v^{\vep,k}\phi_0^{\vep, k}\,dx =\sqrt{\frac{N|\Omega|R_*^{N-1}}{m(v^*)}}J'(v^*),\label{eqn:f_vphi_0^vep_integral}
\end{gather}
both of which hold uniformly in $D\geq D_0$. Therefore, \eqref{eqn:lambda^vep<C_kvep} follows directly from \eqref{eqn:mu_0^vep=o(vep)}, \eqref{eqn:lambda^vep=mu^vep-...} and \eqref{eqn:1-int_psi_lamba bdd}-\eqref{eqn:f_vphi_0^vep_integral}.

We now prove \eqref{eqn:lambda/vep=Lambda*}. By \eqref{eqn:1-psi_lambda^vep_limit} and \eqref{eqn:lambda^vep<C_kvep}, we see that
\begin{align}\label{eqn:denominator converge}
	\lim_{\vep\to0}\Big|
1 -\left(1-\frac{\vep}{D}\right) \frac{1}{|\Omega|}\int_{\Omega}\psi_{\lambda^\vep}^{\vep}\,dx\Big| = \frac{G}{f_u^{*, +}f_u^{*, -}}>0 \text{ uniformly in }D\geq D_0.
\end{align}
Dividing both sides of \eqref{eqn:lambda^vep=mu^vep-...} by $\vep$ and sending $\vep\to0$, we see from \eqref{eqn:mu_0^vep=o(vep)} and \eqref{eqn:phi_0^vep_integral}-\eqref{eqn:denominator converge} that
\begin{align*}
\lim_{\vep\to0}\frac{\lambda^\vep}{\vep} = -\frac{R_*^{N-1}(h^+(v^*) -h^-(v^*))}{Nm(v^*)}\cdot\frac{f_u^{*, +}f_u^{*, -}}{G}\cdot J'(v^*)=:\Lambda^*
\end{align*}
uniformly with $D\geq D_0$.
This finishes the proof of Claim
\ref{clm:lambda/vep->Lambda*} and hence \eqref{eqn:lambda_0^vep}.

\begin{claim}\label{clm:lambda^vepIsReal}
Let $\vep\in(0, \bar\vep]$ and $D\geq D_0$. Then for $\lambda^{\vep}\in \mathbb C_{\lambda_*}$ to be an eigenvalue of $\mathscr L_k^{\vep}$, it is necessary that 
$\lambda^\vep$ is real, i.e., $\im\lambda^\vep = 0$.
\end{claim}
We now prove Claim \ref{clm:lambda^vepIsReal}.
Since $\Lambda^*\in\mathbb R$, it follows from \eqref{eqn:lambda/vep=Lambda*} that
\begin{align}\label{eqn:imlambda=o(vep)}
\lim_{\vep\to0}\frac{\im\lambda^\vep}{\vep} = 0 \quad \text{ uniformly with }D\geq D_0.
\end{align} 
On the other hand, we see from \eqref{eqn:lambda^vep=mu^vep-...} that 
\begin{align}\label{eqn:im_lambda^vep}
\im\lambda^\vep = \frac{1}{|\Omega|} \left(1-\frac{\vep}{D}\right)\int_{\Omega}\phi_0^{\vep, k}\,dx\int_{\Omega}f_v^{\vep, k}\phi_0^{\vep, k}\,dx\cdot\im\Big(
\frac{1}{1 -\left(1-\frac{\vep}{D}\right) \frac{1}{|\Omega|}\int_{\Omega}\psi_{\lambda^\vep}^{\vep}\,dx}
\Big).
\end{align}
Recall that $\psi_{\lambda^\vep}^{\vep}$ is defined in \eqref{eqn:psi_lambda^vep}, which can be rewritten in the form
\begin{align}\label{eqn:psi_lambda^vep-2}
\psi_{\lambda^\vep}^{\vep} = (\lambda^\vep-\mu_0^{\vep, k})(\mathcal L_k^\vep-\mu_0^{\vep, k} )^{-1}\psi_{\lambda^\vep}^{\vep} + (\mathcal L_k^\vep-\mu_0^{\vep, k} )^{-1}\mathcal P^\vep f_v^{\vep, k},
\end{align}
where $\mathcal P^\vep$ denotes the orthogonal projection onto $[\phi_0^{\vep, k}]^{\perp}$. Then, by choosing $\bar\vep$ even smaller if necessary, Lemma \ref{lem:L^vep_k -mu_0^vep Inv bdd} implies that there exists a constant $C_k = C_k(M, D_0, \bar\vep) > 0$ such that for all $(\vep, \lambda^\vep) \in (0, \bar\vep] \times \mathbb C_{\lambda_*}$ and $D \geq D_0$, 
\begin{align*}
	\|\psi_{\lambda^\vep}^{\vep}\|_{L^\infty} & \leq C_k\big(|\lambda^\vep-\mu_0^{\vep, k}|\cdot 	\|\psi_{\lambda^\vep}^{\vep}\|_{L^\infty} +\|\mathcal P^\vep f_v^{\vep, k}\|_{L^\infty}\big)\\
	& \leq C_k\Big(|\lambda^\vep-\mu_0^{\vep, k}|\cdot \|\psi_{\lambda^\vep}^{\vep}\|_{L^\infty}+ \|f_v^{\vep, k}\|_{L^\infty}+\Big|\Big\langle f_v^{\vep,k},\frac{\phi_0^{\vep, k}}{\sqrt{\vep}} \Big\rangle\Big| \cdot\|\sqrt{\vep}\phi_0^{\vep, k}\|_{L^\infty}\Big)\\
	& \leq \frac{1}{2}\|\psi_{\lambda^\vep}^{\vep}\|_{L^\infty}+C_k.
\end{align*}
Here, to obtain the last inequality, we used 
\eqref{eqn:mu_0^vep=o(vep)} and \eqref{eqn:lambda/vep=Lambda*} to ensure that $C_k|\lambda^\vep-\mu_0^{\vep, k}|\leq \frac{1}{2}$, Theorem \ref{thm:approx_sol_scalarEQ} to ensure $\|f_v^{\vep, k}\|_{L^\infty}\leq C_k$, and \eqref{eqn:phi_0^vep_integral} and \eqref{eqn:f_vphi_0^vep_integral} to control the last term. Consequently, we arrive at 
\begin{align}\label{eqn:psi_sup_norm}
	\|\psi_{\lambda^\vep}^{\vep}\|_{L^\infty}\leq C_k\quad \forall \;(\vep,\lambda^\vep)\in (0,\bar\vep]\times\mathbb C_{\lambda_*}\text{ and }D\geq D_0.
\end{align}	
Rewrite \eqref{eqn:psi_lambda^vep} into its real and imaginary parts respectively, we see that 
\begin{align*}
	\begin{cases}
	(\mathcal L_k^\vep- \re\lambda^\vep)\re\psi_{\lambda^\vep}^{\vep} + \im \lambda^\vep\cdot\im\psi_{\lambda^\vep}^{\vep}  = \mathcal P^{\vep}f_v^{\vep, k},\\
(\mathcal L_k^\vep- \re\lambda^\vep)\im\psi_{\lambda^\vep}^{\vep} = \im \lambda^\vep\cdot\re\psi_{\lambda^\vep}^{\vep}.
\end{cases}
\end{align*}
Then, arguing as in the proof of \eqref{eqn:psi_sup_norm}, we obtain that, for all  $(\vep,\lambda^\vep)\in (0,\bar\vep]\times\mathbb C_{\lambda_*}$ and $D\geq D_0$,
\begin{align}\label{eqn:im_psi^vep=O(im_lambda^vep)}
\begin{split}
\begin{Vmatrix}
\im\psi_{\lambda^\vep}^{\vep}
\end{Vmatrix}_{L^\infty} = |\im \lambda^\vep|\cdot \begin{Vmatrix}\displaystyle
(\mathcal L_k^\vep- \re\lambda^\vep)^{-1}\re\psi_{\lambda^\vep}^{\vep}
\end{Vmatrix}_{L^\infty} & \leq C_k |\im \lambda^\vep|\cdot\|\re\psi_{\lambda^\vep}^{\vep}\|_{L^\infty}\\
&\leq C_k |\im \lambda^\vep|,
\end{split}
\end{align}
where, in the last inequality, we used \eqref{eqn:psi_sup_norm}.
By direct computation, we see from \eqref{eqn:denominator converge}, \eqref{eqn:imlambda=o(vep)} and \eqref{eqn:im_psi^vep=O(im_lambda^vep)} that   
for all $(\vep,\lambda^\vep) \in (0,\bar\vep]\times\mathbb C_{\lambda_*}$ and $D\geq D_0$,
\begin{align}\label{eqn:factor=O(im_lambda^vep)}
	\Big|\im\Big(
		\frac{1}{1 -\left(1-\frac{\vep}{D}\right) \frac{1}{|\Omega|}\int_{\Omega}\psi_{\lambda^\vep}^{\vep}\,dx}
	\Big)\Big| \leq~ C_k |\im \lambda^\vep|.
\end{align}
Using the estimates \eqref{eqn:phi_0^vep_integral}, \eqref{eqn:f_vphi_0^vep_integral} and \eqref{eqn:factor=O(im_lambda^vep)}, we see from  \eqref{eqn:im_lambda^vep} that
\begin{align*}
	|\im \lambda^\vep|\leq C_k\,\vep|\im \lambda^\vep|, \quad 	\forall \;(\vep,\lambda^\vep)\in (0,\bar\vep]\times\mathbb C_{\lambda_*}\text{ and }D\geq D_0.
\end{align*}
By choosing $\bar\vep$ eventually smaller such that $C_k\bar\vep<1$, the above inequality yields $\im\lambda^\vep =0$ for all $\vep\in(0,\bar\vep]$ and $D\geq D_0.$ This finishes the proof of Claim \ref{clm:lambda^vepIsReal}.

\begin{claim}\label{clm:lambda^vep_Unique}
Let $\vep\in(0, \bar\vep]$ and $D\geq D_0$. Then equation \eqref{eqn:lambda^vep=mu^vep-...} 
has a unique solution $\lambda^{\vep}$ in $\mathbb C_{\lambda_*}$.
\end{claim}
We now prove Claim \ref{clm:lambda^vep_Unique}.
Define a map $\Gamma_k^\vep: [-\frac{\lambda_*}{\vep},\infty)\to\mathbb R$ as 
\begin{align*}
	\Gamma_k^\vep(\beta) := \frac{\mu_0^{\vep, k}}{\vep} -\frac{1}{\vep}\cdot \frac{1}{|\Omega|} \left(1-\frac{\vep}{D}\right)\frac{\int_{\Omega}\phi_0^{\vep, k}\,dx}{1 -\left(1-\frac{\vep}{D}\right) \frac{1}{|\Omega|}\int_{\Omega}\psi_{\vep\beta}^{\vep}\,dx}\int_{\Omega}f_v^{\vep, k}\phi_0^{\vep, k}\,dx.
\end{align*}
Then equation \eqref{eqn:lambda^vep=mu^vep-...} is equivalent to $\Gamma_k^\vep(\frac{\lambda^{\vep}}{\vep}) = \frac{\lambda^{\vep}}{\vep}$. Therefore, by choosing $\bar\vep$ even smaller if necessary, we see from Claims \ref{clm:lambda/vep->Lambda*} and \ref{clm:lambda^vepIsReal} that, to prove Claim \ref{clm:lambda^vep_Unique}, it suffices to prove uniqueness of solution to the equation
$\Gamma_k^\vep(\beta) =\beta$ on the interval $I_{\Lambda^*}$ for all $\vep\in(0, \bar\vep]$ and $D\geq D_0$, where
\begin{align*}
I_{\Lambda^*} :=\Big\{\beta \in {\mathbb R}\,\Big|\, \tfrac{|\Lambda^*|}{2} \le \sgn(\Lambda^*)\beta \le 2|\Lambda^*| \Big\}.
\end{align*}
Indeed, it follows from \eqref{eqn:psi_lambda^vep-2} that for any $\beta_1, \beta_2\in I_{\Lambda^*}$,
\begin{align*}
\psi_{\vep\beta_1}^{\vep} -\psi_{\vep\beta_2}^{\vep} = \left[\vep\beta_2-\mu_0^{\vep, k}(\mathcal L_k^\vep-\mu_0^{\vep, k} )^{-1}\right](\psi_{\vep\beta_1}^{\vep} -\psi_{\vep\beta_2}^{\vep})+\vep(\beta_1-\beta_2)\psi_{\vep\beta_1}^{\vep}.
\end{align*}
Then, by \eqref{eqn:psi_sup_norm} and similar arguments to the proof of \eqref{eqn:psi_sup_norm}, we see that 
\begin{align*}
	\big\|\psi_{\vep\beta_1}^{\vep} -\psi_{\vep\beta_2}^{\vep}\big\|_{L^{\infty}}\leq C_k\vep|\beta_1-\beta_2|, \quad \forall \; \beta_1, \beta_2\in I_{\Lambda^*}, \vep\in(0, \bar\vep] \text{ and }D\geq D_0.
\end{align*}
Then, by using estimates \eqref{eqn:phi_0^vep_integral}, \eqref{eqn:f_vphi_0^vep_integral}, \eqref{eqn:denominator converge}, the above inequality and similar arguments to the proof of Claim \ref{clm:lambda/vep->Lambda*}, one can easily show that
\begin{align*}
\Gamma_k^\vep(\beta)\in I_{\Lambda^*} \quad \text{ and }\quad	|\Gamma_k^\vep(\beta_1) - \Gamma_k^\vep(\beta_2)|\leq \frac{|\beta_1-\beta_2|}{2}
\end{align*}
for all $\beta, \beta_1, \beta_2\in I_{\Lambda^*}$, $\vep\in(0, \bar\vep]$ and $D\geq D_0$.
By the contraction mapping theorem,  $\beta = \Gamma_k^\vep(\beta)$ has a unique solution in $I_{\Lambda^*}$ for all $\vep\in(0, \bar\vep]$ and $D\geq D_0$. So we are done.
%get the Claim \ref{clm:lambda^vep_Unique}.

\begin{claim}\label{clm:lambda^vep_Simple}
The nonlocal eigenvalue problem \eqref{eqn:mathscr L_k^vep_EigenPro}
has a unique eigenvalue $\lambda^{\vep}$ in $\mathbb C_{\lambda_*}$, and it is a simple. 	
\end{claim}

Indeed, by Claim \ref{clm:lambda^vep_Unique}, $\lambda^{\vep}$ is uniquely determined from the equation \eqref{eqn:lambda^vep=mu^vep-...}. Substituting \eqref{eqn:b=int_w}, \eqref{eqn:psi_lambda^vep}, \eqref{eqn:b-2} and $a^\vep=1$ into \eqref{eqn:w^vep}, we obtain that
\begin{align}\label{eqn:w^vep_2}
	\eta^{\vep} = \frac{1}{|\Omega|}\left(1-\frac{\vep}{D}\right)\frac{\int_{\Omega}\phi_0^{\vep, k}\,dx}{1 -\left(1-\frac{\vep}{D}\right) \frac{1}{|\Omega|}\int_{\Omega}\psi_{\lambda^{\vep}}^{\vep}\,dx}\cdot \psi_{\lambda^{\vep}}^{\vep}.
\end{align} 
This implies that the second equation in \eqref{eqn:a_and_w^vep} also admits a unique solution for given $\lambda^{\vep}$. Therefore, $\lambda^{\vep}$ is the unique eigenvalue of $\mathscr L_k^{\vep}$ in $\mathbb C_{\lambda_*}$, and there exists an eigenfunction $\Phi^{\vep} =\phi_0^{\vep, k} + \eta^{\vep}$ of $\mathscr L_k^{\vep}$ corresponding to $\lambda^{\vep}$, which is unique up to  a constant multiple. This finishes the proof of Claim \ref{clm:lambda^vep_Simple}. 

\medskip

Now Theorem \ref{thm:mathscr_L_k^vep} (i) and (ii) follows directly from Claims \ref{clm:lambda/vep->Lambda*}-\ref{clm:lambda^vep_Simple}.

\medskip

For definiteness, we henceforth denote the critical eigenvalue $\lambda^{\vep}$ of $\mathscr L_k^{\vep}$ in $\mathbb C_{\lambda_*}$ by $\lambda_0^{\vep, k}$ and the corresponding eigenfunction $\Phi^{\vep}$ by $\Phi_0^{\vep, k}$ in the remainder of this proof.

\medskip

Next, we prove Theorem \ref{thm:mathscr_L_k^vep} (iii).
Plugging the estimates \eqref{eqn:phi_0^vep_Norms_est}, \eqref{eqn:denominator converge} and \eqref{eqn:psi_sup_norm} into \eqref{eqn:w^vep_2}, we immediately obtain that there exists a constant $C_k=C_k(M, D_0, \bar\vep)$ such that for all $\vep\in(0,\bar\vep]$ and $D\geq D_0$,
\begin{align}\label{eqn:w^vep_Norms_est}
\|\eta^{\vep}\|_{L^\infty(\Omega)}	+ \|\eta^{\vep}\|_{L^2(\Omega)} + \|\eta^{\vep}\|_{L^1(\Omega)}	 \leq C_k\sqrt{\vep}.
\end{align}
Therefore, dividing $\Phi_0^{\vep, k} = \phi_0^{\vep, k} + \eta^{\vep}$ by $\|\Phi_0^{\vep, k}\|_{L^2}$ so that the resulting function is $L^2$-normalized and still denoting the resulting function by $\Phi_0^{\vep, k}$, we see from \eqref{eqn:phi_0^vep_Norms_est} and \eqref{eqn:w^vep_Norms_est} that 
\begin{align*}
\|\Phi_0^{\vep, k}\|_{L^1(\Omega)}	 \leq C_k\sqrt{\vep},\quad \|\Phi_0^{\vep, k}\|_{L^\infty(\Omega)} \leq \frac{C_k}{\sqrt{\vep}}
\end{align*}
for all $\vep\in(0,\bar\vep]$ and $D\geq D_0$.
This finishes the proof of Theorem \ref{thm:mathscr_L_k^vep} (iii). 

\medskip

Finally, we prove Theorem \ref{thm:mathscr_L_k^vep} (iv). Note that the adjoint operator $\widehat{\mathscr L}_k^\vep$ on $L^s(\Omega)$ with domain $W_{\nu}^{2,s}$ is defined as:
\begin{align}\label{eqn:def_hat_mathcal L_k^vep}
\begin{split}
\displaystyle	\widehat{\mathscr L}_k^\vep\psi 
:= &~\vep^2\Delta\psi + \left(f_u^{\vep,k}-\frac{\vep}{D}f_v^{\vep,k} \right)\phi -\frac{1}{|\Omega|} \left(1-\frac{\vep}{D}\right)\int_\Omega f_v^{\vep,k}\psi\,dx \\
= &~ \mathcal L_k^\vep\psi -\frac{1}{|\Omega|} \left(1-\frac{\vep}{D}\right)\int_\Omega f_v^{\vep,k}\psi\,dx.
\end{split}
\end{align}

In the sequel, we assume that $\widehat\lambda^\vep\in\mathbb C_{\lambda_*}$ is an eigenvalue of $\widehat{\mathscr L}_k^\vep$,  and let $\widehat\Phi^\vep$ be a corresponding eigenfunction.
We then proceed by applying similar arguments as those used in the proofs of Theorem \ref{thm:mathscr_L_k^vep} (i)-(iii).

Decompose $\widehat\Phi^\vep$ as
\begin{align}\label{eqn:decom of hat Phi_0^vep}
	\widehat\Phi^\vep = \widehat a^{\vep} \phi_0^{\vep, k} + \widehat \eta^{\vep}\in{\rm span}\{\phi_0^{\vep, k}\}\oplus[\phi_0^{\vep, k}]^{\perp},
\end{align}
%where $\widehat a^{\vep}\in\mathbb C$ and $\widehat \eta^{\vep}\in[\phi_0^{\vep, k}]^{\perp}$. 
then the equation for $\widehat\Phi^\vep$ can be reformulated as follows:
\begin{align}\label{eqn:hat_a_and_w^vep}
\begin{array}{l}
	\displaystyle	(\mu_0^{\vep, k}-\widehat\lambda^\vep)\widehat a^{\vep} =\frac{1}{|\Omega|} \Big(1-\frac{\vep}{D}\Big)\Big(\widehat a^{\vep}\int_{\Omega}f_v^{\vep,k}\phi_0^{\vep, k}\,dx +\int_{\Omega}f_v^{\vep,k}\widehat \eta^{\vep}\,dx\Big)\int_{\Omega}\phi_0^{\vep, k}\,dx, \vspace{1ex}\\
	\displaystyle (\mathcal L_k^\vep-\widehat\lambda^\vep)\widehat \eta^{\vep} =\frac{1}{|\Omega|}\Big(1-\frac{\vep}{D}\Big)\Big(\widehat a^{\vep}\int_{\Omega}f_v^{\vep,k}\phi_0^{\vep, k}\,dx +\int_{\Omega}f_v^{\vep,k}\widehat \eta^{\vep}\,dx\Big) \mathcal P^\vep \mathds{1},
\end{array}
\end{align}
where $\mathds{1}$ denotes the constant function taking value $1$ on $\Omega$. Thanks to Proposition \ref{pro:L^vep_k properties} (i) and the assumption that $\widehat\lambda^\vep \in \mathbb C_{\lambda_*}$, we can solve for $\widehat \eta^{\vep}$ from the second equation in \eqref{eqn:hat_a_and_w^vep}. Multiplying the resulting expression of $\widehat \eta^{\vep}$ by $f_v^{\vep,k}$ yields
\begin{align}\label{eqn:hat_w^vep}
	f_v^{\vep,k}\widehat \eta^{\vep} = \frac{1}{|\Omega|}\Big(1-\frac{\vep}{D}\Big)\Big(\widehat a^{\vep}\int_{\Omega}f_v^{\vep, k}\phi_0^{\vep, k}\,dx +\int_{\Omega}f_v^{\vep, k}\widehat \eta^{\vep}\,dx\Big)f_v^{\vep,k}(\mathcal L_k^\vep-\widehat \lambda^\vep)^{-1}\mathcal P^\vep \mathds{1}.
\end{align}
For notational simplicity, denote
\begin{align}\label{eqn:hat b=int_hat w}
	\widehat b^\vep := \int_{\Omega}f_v^{\vep,k}\widehat \eta^{\vep}\,dx
\quad \text{ and }\quad
	\widehat \psi_{\widehat\lambda^\vep}^{\vep} :=f_v^{\vep,k}(\mathcal L_k^\vep-\widehat\lambda^\vep)^{-1}\mathcal P^\vep \mathds{1}.
\end{align}
Integrating equation \eqref{eqn:hat_w^vep} over $\Omega$, we obtain that
\begin{align}\label{eqn:hat_b}
\widehat b^\vep = \frac{\widehat b^\vep}{|\Omega|}\left(1-\frac{\vep}{D}\right)\int_{\Omega}\widehat \psi_{\widehat\lambda^\vep}^{\vep}\,dx +\frac{
\widehat a^{\vep}}{|\Omega|} \left(1-\frac{\vep}{D}\right)\int_{\Omega}f_v^{\vep,k}\phi_0^{\vep, k}\,dx\int_{\Omega}\widehat \psi_{\widehat\lambda^\vep}^{\vep}\,dx.
\end{align}

By similar arguments as in the proof of Claim \ref{clm:b-solvable}, we see that there exists a constant $\bar\vep =\bar\vep(M, D_0, k)>0$ such that for each $(\vep, \widehat\lambda^{\vep})\in(0, \bar\vep] \times\mathbb C_{\lambda_*}$ and $D\geq D_0$, equation \eqref{eqn:hat_b} admits a unique solution given explicitly by
\begin{align}\label{eqn:hat_b-2}
\widehat b^\vep = \widehat a^{\vep}\cdot\frac{\left(1-\frac{\vep}{D}\right) \frac{1}{|\Omega|}\int_{\Omega}\widehat \psi_{\widehat\lambda^\vep}^{\vep}\,dx}{1 -\left(1-\frac{\vep}{D}\right) \frac{1}{|\Omega|} \int_{\Omega}\widehat \psi_{\widehat\lambda^\vep}^{\vep}\,dx}\int_{\Omega}f_v^{\vep,k}\phi_0^{\vep, k}\,dx.
\end{align}

Then, by reasoning analogous to that used in the proof of Claim~\ref{clm:a neq 0}, we deduce from \eqref{eqn:hat_b-2} that for $\widehat\lambda^{\vep}\in \mathbb C_{\lambda_*}$ to be an eigenvalue of $\widehat{\mathscr L}_k^{\vep}$, it is necessary that 
\begin{align}\label{eqn:hat a^vep neq 0}
	\widehat a^{\vep}\neq 0 \quad \forall \;\vep\in(0, \bar\vep]\text{ and }D\geq D_0.
\end{align}

Substituting \eqref{eqn:hat b=int_hat w} and \eqref{eqn:hat_b-2} into the first equation of \eqref{eqn:hat_a_and_w^vep} and canceling the factor $\widehat a^{\vep}\neq 0$ on both sides, we obtain the following explicit expression for $\widehat\lambda^\vep$:
\begin{align}\label{eqn:hat_lambda^vep=mu^vep-...}
\displaystyle
\widehat\lambda^\vep = \mu_0^{\vep, k} - \frac{1}{|\Omega|} \left(1-\frac{\vep}{D}\right)\frac{\int_{\Omega}\phi_0^{\vep, k}\,dx}{1 -\left(1-\frac{\vep}{D}\right) \frac{1}{|\Omega|}\int_{\Omega}\widehat\psi_{\widehat\lambda^\vep}^{\vep}\,dx}\int_{\Omega}f_v^{\vep, k}\phi_0^{\vep, k}\,dx.
\end{align} 

Then by similar arguments to the proofs used in Claims \ref{clm:lambda/vep->Lambda*}-\ref{clm:lambda^vep_Simple}, we obtain: 

\begin{claim}\label{clm:hat_lambda^vep_Simple}
Let $\vep\in(0, \bar\vep]$ and $D\geq D_0$. Then the $L^2$-adjoint operator $\widehat{\mathscr L}_k^\vep$ of $\mathscr L_k^{\vep}$ has a unique eigenvalue $\widehat\lambda^{\vep}$ in $\mathbb C_{\lambda_*}$, and it is a simple. Moreover, $\widehat\lambda^\vep$ is real and satisfies
\begin{align}\label{eqn:hat_lambda/vep=Lambda*}	\lim_{\vep\to0}\frac{\widehat\lambda^\vep}{\vep} =\Lambda^*\in\mathbb R\backslash \{0\} \quad \text{ uniformly in }D\geq D_0,
\end{align}
where $\Lambda^*$ is defined in \eqref{eqn:Lambda^*}.  
\end{claim}

We henceforth denote the critical eigenvalue $\widehat\lambda^{\vep}$ of $\widehat{\mathscr L}_k^{\vep}$ in $\mathbb C_{\lambda_*}$ by $\widehat\lambda_0^{\vep, k}$ and the corresponding eigenfunction $\widehat\Phi^{\vep}$ by $\widehat\Phi_0^{\vep, k}$ in the remainder of this paper.

By similar arguments to the proof of \eqref{eqn:w^vep_Norms_est}, we can show that
there exists a constant $C_k = C_k(M, D_0, \bar\vep)>0$ such that for all $\vep\in(0,\bar\vep]$ and $D\geq D_0$,
\begin{align}\label{eqn:hat_w^vep_Norms_est}
	\|\widehat \eta^{\vep}\|_{L^\infty(\Omega)} + \|\widehat \eta^{\vep}\|_{L^2(\Omega)} + \|\widehat \eta^{\vep}\|_{L^1(\Omega)}\leq C_k\widehat a^{\vep}\sqrt{\vep}.
\end{align}

We now claim that
\begin{claim}\label{eqn:lambda_0=hat_lambda_0}
$\lambda_0^{\vep, k} =\widehat\lambda_0^{\vep, k}$ and $\langle \Phi_0^{\vep, k}, \widehat\Phi_0^{\vep, k} \rangle\neq 0$
for all $\vep\in(0, \bar\vep]$ and $D\geq D_0$.
\end{claim}

We now prove Claim \ref{eqn:lambda_0=hat_lambda_0}. Since $\Phi_0^{\vep, k} = \phi_0^{\vep, k} + \eta^{\vep}$ and 
$\widehat\Phi_0^{\vep, k} = \widehat a^{\vep} \phi_0^{\vep, k} + \widehat \eta^{\vep}$,
where $0\neq\widehat a^{\vep}\in\mathbb R$, $\eta^{\vep}, \widehat \eta^{\vep}\in[\phi_0^{\vep, k}]^{\perp}$ are real functions, it follows from \eqref{eqn:w^vep_Norms_est} and \eqref{eqn:hat_w^vep_Norms_est} that
\begin{align}\label{eqn:for_hat a}
\Big|\displaystyle \frac{\langle \Phi_0^{\vep, k}, \widehat\Phi_0^{\vep, k} \rangle}{\widehat a^{\vep}}
-1 \Big| = \Big|\frac{1}{\widehat a^{\vep}}\int_{\Omega}\eta^{\vep}\widehat \eta^{\vep}\,dx \Big| \leq C_k\vep,\quad \forall \;\vep\in(0, \bar\vep] \text{ and }D\geq D_0.
\end{align}
Therefore, by choosing $\bar\vep$ even smaller if necessary, we see that $\langle \Phi_0^{\vep, k}, \widehat\Phi_0^{\vep, k} \rangle\neq 0$ for all $\vep\in(0, \bar\vep]$ and $D\geq D_0$. Now, multiplying the equation $\mathscr L_k^{\vep}\Phi_0^{\vep, k} =  \lambda_0^{\vep, k}\Phi_0^{\vep, k}$ by $\widehat\Phi_0^{\vep, k}$ and the equation $\widehat{\mathscr L}_k^\vep\widehat\Phi_0^{\vep, k} = \widehat\lambda_0^{\vep, k} \widehat\Phi_0^{\vep, k}$ by $\Phi_0^{\vep, k}$, integrating over $\Omega$ and subtracting the resulting equation, we obtain that
\begin{align*}
	(\lambda_0^{\vep, k}-\widehat\lambda_0^{\vep, k})\cdot\langle \Phi_0^{\vep, k}, \widehat\Phi_0^{\vep, k} \rangle =0.
\end{align*}
It then follows from $\langle \Phi_0^{\vep, k}, \widehat\Phi_0^{\vep, k} \rangle\neq 0$ that $\widehat\lambda_0^{\vep, k}=\lambda_0^{\vep, k}$. This finishes the proof of Claim \ref{eqn:lambda_0=hat_lambda_0}. 

Finally, normalize the eigenfunction $\widehat\Phi_0^{\vep, k}$ of the adjoint operator $\widehat{\mathscr L}_k^\vep$ corresponding to $\widehat\lambda_0^{\vep, k}$ so that $\langle\widehat\Phi_0^{\vep, k}, \Phi_0^{\vep, k} \rangle =1$, we obtain from \eqref{eqn:for_hat a} that 
\begin{align}\label{eqn:hat_a^vep_est}
|\widehat a^{\vep}- 1|\leq C_k\vep.
\end{align}
This combined with \eqref{eqn:hat_w^vep_Norms_est} completes the proof of Theorem \ref{thm:mathscr_L_k^vep} (iv).
\end{proof}

\subsection{Proof of Lemma \ref{lem:mathscr L_k^vep inver_bdd}}\label{sec:Proof of Lemma 3.3}
\begin{proof}
We prove \eqref{eqn:mathscr L_inver_bdd} by contradiction.  Assume that there exist two sequences $\vep_j\to 0$, $D_j\geq D_0$ and a sequence of functions $\{g_j\}\subset X^\infty$ such that
\begin{align*}
\|g_j\|_{L^{\infty}}  =1 \quad \text{and} \quad \|(\mathscr L^{\vep_j}_k -\lambda_0^{\vep_j, k})^{-1}g_j \|_{L^{\infty}} \geq j.
\end{align*}
Put 
\begin{align*}
h_j := \frac{(\mathscr L^{\vep_j}_k -\lambda_0^{\vep_j, k})^{-1}g_j}{\|(\mathscr L^{\vep_j}_k -\lambda_0^{\vep_j, k})^{-1}g_j \|_{L^{\infty}} }\quad  \text{and}\quad 
q_j:= \frac{g_j}{\|(\mathscr L^{\vep_j}_k -\lambda_0^{\vep_j, k})^{-1}g_j \|_{L^{\infty}} }. 
\end{align*}
Then $h_j$ satisfies 
\begin{align}\label{eqn:mathscr_L_k^vep_chi_j=q_j}
(\mathscr L^{\vep_j}_k -\lambda_0^{\vep_j, k})h_j = q_j, \quad \mbox{with } \|h_j\|_{L^{\infty}} = 1\; \text{ and }\;  \|q_j\|_{L^{\infty}} \leq j^{-1}.
\end{align}

We will apply to \eqref{eqn:mathscr_L_k^vep_chi_j=q_j} the same procedures as in the proof of Theorem \ref{thm:mathscr_L_k^vep}. Note that since $\lambda_0^{\vep_j, k}\in\mathbb R$, we can decompose $h_j$ as 
$$h_j = c_j\phi_0^{\vep_j, k} + \theta_j\in{\rm span}\{\phi_0^{\vep_j, k}\}\oplus[\phi_0^{\vep_j, k}]^{\perp}.$$
%where $c_j\in\mathbb R$ and $\theta_j\in[\phi_0^{\vep, k}]^{\perp}$ is a real function. 
Then equation \eqref{eqn:mathscr_L_k^vep_chi_j=q_j} can be reformulated as follows:
\begin{align}
\label{eqn:c^j}
\begin{split}
%\footnotesize
\big(\mu_0^{\vep_j, k}-\lambda_0^{\vep_j, k}\big) c_j & =\frac{1}{|\Omega|} \Big(1-\frac{\vep_j}{D_j}\Big)\Big(c_j\int_{\Omega}\phi_0^{\vep_j, k}\,dx +\int_{\Omega}\theta_j\,dx\Big)\int_{\Omega}f_v^{\vep_j, k}\phi_0^{\vep_j, k}\,dx\\
& \quad + \int_{\Omega}q_j\phi_0^{\vep_j, k}dx,
\end{split}
\end{align}
and
\begin{align}
\label{eqn:theta^j}
\begin{split}
\theta_j & = \frac{1}{|\Omega|}\Big(1-\frac{\vep_j}{D_j}\Big)\Big(c_j\int_{\Omega}\phi_0^{\vep_j, k}dx +\int_{\Omega}\theta_j\,dx\Big) \big(\mathcal L_k^{\vep_j}-\lambda_0^{\vep_j, k}\big)^{-1}\mathcal P^{\vep_j} f_v^{\vep_j, k}\\
& \quad + \big(\mathcal L_k^{\vep_j}-\lambda_0^{\vep_j, k}\big)^{-1}\mathcal P^{\vep_j}q_j.
\end{split}
\end{align}
Denote
\begin{align*}%\label{eqn:bj&psi^j&chi^j}
b_j:=\int_{\Omega}\theta_j\,dx, \quad \psi_j:= (\mathcal L_k^{\vep_j}-\lambda_0^{\vep_j, k})^{-1}\mathcal P^{\vep_j} f_v^{\vep_j, k} \quad \text{ and } \quad \rho_j := (\mathcal L_k^{\vep_j}-\lambda_0^{\vep_j, k})^{-1}\mathcal P^{\vep_j}q_j.
\end{align*}
Then the equation in \eqref{eqn:theta^j} can be rewritten as
\begin{align}\label{eqn:theta_j-rewrite}
	\theta_j = \frac{1}{|\Omega|}\Big(1-\frac{\vep_j}{D_j}\Big)\Big(c_j\int_{\Omega}\phi_0^{\vep_j, k}dx +b_j\Big)\psi_j +\rho_j.
\end{align}
By similar argument as in the proof of  \eqref{eqn:psi_sup_norm}, we can show that there exists a constant $C_k  = C_k(M, D_0)>0$ such that for all $j$ large,
\begin{align}\label{eqn:psi_j,rho_j-EST}
\|\psi_j\|_{L^\infty}\leq C_k\quad \text{ and }\quad \|\rho_j\|_{L^\infty}\leq C_k\|q_j\|_{L^\infty}.
\end{align}
Integrating equation \eqref{eqn:theta_j-rewrite} over $\Omega$, we see that
\begin{align}\label{eqn:b_j}
b_j = \frac{1}{|\Omega|}\Big(1-\frac{\vep_j}{D_j}\Big)\Big(\displaystyle c_j\int_{\Omega}\phi_0^{\vep_j, k}dx +b_j\Big)\int_{\Omega} \psi_j dx+\int_{\Omega}\rho_j dx.
\end{align}
By similar arguments to the proof of \eqref{eqn:denominator converge}, we see that
\begin{align}\label{eqn:1-int_psi^j}
\lim_{j\to0}\Big[1 -\frac{1}{|\Omega|}\Big(1-\frac{\vep_j}{D_j}\Big) \int_{\Omega} \psi_j\,dx\Big] = \frac{G}{f_u^{*, +}f_u^{*, -}}>0.
\end{align} 
Therefore, we can solve $b_j$ from \eqref{eqn:b_j} and for all $j$ sufficiently large and obtain that
\begin{align}\label{eqn:b_j_solved}
b_j = \frac{c_j\frac{1}{|\Omega|}\left(1-\frac{\vep_j}{D_j}\right)\int_{\Omega}\phi_0^{\vep_j, k}dx\int_{\Omega} \psi_j dx +\int_{\Omega}\rho_j dx}{1 -\frac{1}{|\Omega|}\left(1-\frac{\vep_j}{D_j}\right) \int_{\Omega} \psi_j\,dx}.
\end{align}
Plugging \eqref{eqn:b_j_solved} into \eqref{eqn:theta_j-rewrite}, we obtain that
\begin{align*}
\theta_j = \frac{1}{|\Omega|}\Big(1-\frac{\vep_j}{D_j}\Big) \frac{c_j\int_{\Omega}\phi_0^{\vep_j, k}dx +\int_{\Omega}\rho_j dx}{1 -\frac{1}{|\Omega|}\left(1-\frac{\vep_j}{D_j}\right) \int_{\Omega} \psi_j\,dx}\psi_j +\rho_j,	
\end{align*}
which combined with \eqref{eqn:psi_j,rho_j-EST}, \eqref{eqn:1-int_psi^j} and \eqref{eqn:phi_0^vep_integral} implies that there exists a constant $C_k = C_k(M, D_0)>0$ such that for all $j$ sufficiently large,
\begin{align}\label{eqn:theta_j-L^inf}
	\|\theta_j\|_{L^{\infty}}\leq C_k\left(|c_j|\sqrt{\vep_j} + \|q_j\|_{L^\infty}\right).
\end{align}

On the other hand, since $h_j\in X^{\infty}$, see \eqref{eqn:X^s}, we have $\langle h_j,  \widehat\Phi_0^{\vep_j, k}\rangle =0$,
where $\widehat\Phi_0^{\vep_j, k}$ is the eigenfunction of the adjoint operator $\widehat{\mathscr L}_k^{\vep_j}$ corresponding to $\lambda_0^{\vep_j, k}$. It follows from the proof of Theorem \ref{thm:mathscr_L_k^vep} (iv) that
$\widehat\Phi_0^{\vep_j, k} = \widehat a^{\vep_j}\phi_0^{\vep_j, k}+\widehat \eta^{\vep_j}$, where $\widehat \eta^{\vep_j}\in [\phi_0^{\vep_j, k}]^{\perp}$ and $\widehat a^{\vep_j}\in\mathbb R\backslash \{0\}$ satisfy \eqref{eqn:hat_w^vep_Norms_est} and \eqref{eqn:hat_a^vep_est} respectively.
Consequently, 
\begin{align*}
	0 = \big\langle h_j, \widehat\Phi_0^{\vep_j, k}\big\rangle = \big\langle c_j\phi_0^{\vep_j, k} + \theta_j,\, \widehat a^{\vep_j}\phi_0^{\vep_j, k}+\widehat \eta^{\vep_j}\big\rangle 
	= c_j\widehat a^{\vep_j} + \int_{\Omega} \theta_j\widehat \eta^{\vep_j}\,dx,
\end{align*}
which combined with \eqref{eqn:hat_w^vep_Norms_est} and \eqref{eqn:hat_a^vep_est} implies that
\begin{align*}
	|c_j| = \Big|
-\frac{1}{\widehat a^{\vep_j}}\int_{\Omega} \theta_j\widehat \eta^{\vep_j}\,dx
\Big| \leq C_k\sqrt{\vep_j} \|\theta_j\|_{L^{\infty}}.
\end{align*}
Plugging this back into \eqref{eqn:theta_j-L^inf}, we obtain that
\begin{align*}
\|\theta_j\|_{L^{\infty}}\leq C_k\left(\vep_j\|\theta_j\|_{L^{\infty}} +\|q_j\|_{L^\infty}\right).
\end{align*}
Since $\lim_{j\to\infty}\vep_j = 0$ and $\|q_j\|_{L^{\infty}} \leq j^{-1}$, this implies that 
\begin{align*}
\|\theta_j\|_{L^{\infty}}\leq	C_k\|q_j\|_{L^\infty}\to 0  \quad \text{ and hence } \quad \frac{|c_j|}{\sqrt{\vep_j}}\to 0, \quad \text{as }j\to\infty.
\end{align*}
Therefore, we see from \eqref{eqn:phi_0^vep_Norms_est}  that 
\begin{align*}
	\|h_j\|_{L^{\infty}} = 	\|c_j\phi_0^{\vep_j, k} + \theta_j\|_{L^{\infty}}\to 0, \quad \text{as }j\to\infty,
\end{align*}
which is a contradiction to \eqref{eqn:mathscr_L_k^vep_chi_j=q_j}. This completes the proof of Lemma \ref{lem:mathscr L_k^vep inver_bdd}.  
\end{proof}

\section{Proof of the main result} 
In this section, we prove our main result Theorem \ref{thm:main-scalar}. 

\begin{proof}[Proof of Theorem \ref{thm:main-scalar}]
We fix an order of approximation $k \geq 2$ in Theorem \ref{thm:approx_sol_scalarEQ},
and look for solutions of equation \eqref{eqn:MCRD_nonlocal_scalar} in the following form:
\begin{align*}
u = u_k^{\vep, D} + p, \quad \text{where} p\in  L_{rad}^{\infty}.
\end{align*}
In terms of $p$, equation \eqref{eqn:MCRD_nonlocal_scalar} is equivalent to 
\begin{align}\label{eqn:mathscr L^vep_k p}
\mathscr L_k^{\vep} p +  \mathcal F^\vep[p] + \mathcal R^\vep[u^{\vep,D}_k] = 0,
\end{align}
where $\mathcal R^\vep[u^{\vep,D}_k]$ is defined in Theorem \ref{thm:approx_sol_scalarEQ} (ii) and 
\begin{align}\label{eqn:mathcal F^vep}
\begin{split}
\mathcal F^\vep[p] & := f\left(u_k^{\vep, D} + p,\, \mathcal S^{\vep}[u_k^{\vep, D} + p]-\frac{\vep}{D}(u_k^{\vep, D} + p)\right) 
-  f\left(u^{\vep,D}_k,\,\mathcal S^{\vep}[u^{\vep,D}_k]-\frac{\vep}{D}u^{\vep,D}_k \right)\\
&\quad - f_u^{\vep,k}p+\Big[\frac{\vep}{D} p +\frac{1}{|\Omega|} \Big(1-\frac{\vep}{D}\Big)\int_\Omega p(x)\,dx\Big]f_v^{\vep,k}. \end{split}
\end{align}
Note that $f_u^{\vep,k}$ and $f_v^{\vep,k}$ are defined in \eqref{eqn:f_u,v^vep,k}.

Denote $\mathcal Q^{\vep}$ the $L^2$-projection operator onto $[\Phi_0^{\vep, k}]$ along $\widehat\Phi_0^{\vep, k}$ defined by 
\begin{align*}
\mathcal Q^{\vep} p = \langle  \widehat\Phi_0^{\vep, k},\, p\rangle \Phi_0^{\vep, k}.
\end{align*}
Rewrite $p$ in the form
\begin{align*}
p = \alpha \vep^{k+\frac{1}{2}}\Phi_0^{\vep, k} + \vep^k q,
\end{align*}
where $\alpha  =\langle  \widehat\Phi_0^{\vep, k},\, p\rangle/\vep^{k+\frac{1}{2}} \in\mathbb R$ and $q\in X^{\infty}$. Note that $X^{\infty}$ is defined in \eqref{eqn:X^s}.
Then \eqref{eqn:mathscr L^vep_k p} is reformulated as
\begin{align}\label{eqn:alpha&w}
\begin{split}
\vep^{k+\frac{1}{2}}\alpha \lambda_0^{\vep,k} +  \langle  \widehat\Phi_0^{\vep, k},\, \mathcal F^\vep[\alpha \vep^{k+\frac{1}{2}}\Phi_0^{\vep, k} + \vep^k q]\rangle & = -\langle  \widehat\Phi_0^{\vep, k},\, \mathcal R^\vep[u^{\vep,D}_k]\rangle, \\
\vep^k (\mathscr L_k^{\vep} -\lambda_0^{\vep,k}) q + (I - \mathcal Q^{\vep}) \mathcal F^\vep[\alpha \vep^{k+\frac{1}{2}}\Phi_0^{\vep, k} + \vep^k q] & = -\vep^k\lambda_0^{\vep,k} q -(I - \mathcal Q^{\vep})\mathcal R^\vep[u^{\vep,D}_k].	
\end{split}
\end{align}
According to Lemma \ref{lem:mathscr L_k^vep inver_bdd}, the inverse operator $(\mathscr L_k^{\vep} -\lambda_0^{\vep,k})^{-1}$ is a bounded operator from $X^{\infty}$ to itself. Moreover, by the comment just above Lemma \ref{lem:mathscr L_k^vep inver_bdd}, we can actually work in the space
\begin{align*}
X_0 := C_{rad}(\overline\Omega)\cap X^{\infty}
\end{align*}
and consider the system: 
\begin{align}\label{eqn:eqn:alpha&w rewrite}
\alpha  =  \mathcal T_1[\alpha, q], \; q = \mathcal T_2[\alpha, q], \quad \text{ in } \mathbb R\times X_0,
\end{align}
where
\begin{align*}
\mathcal T_1[\alpha, q] & :=	-\frac{1}{\vep^{k+\frac{1}{2}} \lambda_0^{\vep,k}}\Big\langle  \widehat\Phi_0^{\vep, k},\, \mathcal F^\vep[\alpha \vep^{k+\frac{1}{2}}\Phi_0^{\vep, k} + \vep^k q]\Big\rangle+ \Big\langle  \widehat\Phi_0^{\vep, k},\, \mathcal R^\vep[u^{\vep,D}_k]\Big\rangle,\\
\mathcal T_2[\alpha, q] & := -(\mathscr L_k^{\vep} -\lambda_0^{\vep,k})^{-1}\Big(\frac{1}{\vep^k}(I - \mathcal Q^{\vep}) \mathcal F^\vep[\alpha \vep^{k+\frac{1}{2}}\Phi_0^{\vep, k} + \vep^k q]+\lambda_0^{\vep,k} q\Big)\\
& \qquad -\frac{1}{\vep^k}(\mathscr L_k^{\vep} -\lambda_0^{\vep,k})^{-1}\Big((I - \mathcal Q^{\vep})\mathcal R^\vep[u^{\vep,D}_k]\Big).
\end{align*}
Observe that if $(\alpha, q) \in \mathbb R\times X_0$ sovles \eqref{eqn:eqn:alpha&w rewrite}, then actually $q$ is a classical solution to the second equation in \eqref{eqn:alpha&w}. In fact, by the definition of $\mathcal Q^{\vep}$, we see that
$$\frac{1}{\vep^k}(I - \mathcal Q^{\vep}) \mathcal F^\vep[\alpha  \vep^{k+\frac{1}{2}}\Phi_0^{\vep, k} + \vep^k q]+\lambda_0^{\vep,k} q+\frac{1}{\vep^k}(I - \mathcal Q^{\vep})\mathcal R^\vep[u^{\vep,D}_k] \in X^{\infty}.$$
Therefore, $q\in X_{\nu}^{2,s}\subset W_{\nu}^{2,s}$ for any $s\in(1,\infty)$ by Lemma \ref{lem:mathscr L_k^vep invertible}. Choosing $s>N$, we may infer that $q\in C^{1+\alpha}(\overline\Omega)$ for some $\alpha\in(0,1)$ by the Sobolev embedding theorem. Then by the Schauder theory, $w$ turns out to be of class $C^{2+\alpha}(\overline\Omega)$ and satisfies the desired zero-Newmann boundary condition.

We now solve \eqref{eqn:eqn:alpha&w rewrite} by applying Banach's contraction mapping principle. To this end, we define a mapping $\mathcal T: \mathbb R\times X_0\to\mathbb R\times X_0$ by $\mathcal T = (\mathcal T_1, \mathcal T_2)$. We equip $\mathbb R\times X_0$ with the norm
\begin{align*}
	\triplenorm{(\alpha, q)}:=|\alpha|+ \|q\|_{L^{\infty}(\Omega)},
\end{align*}
under which it becomes a Banach space.
We first establish the following preparatory lemma.

\begin{lemma}\label{lem:mathcal F^vep}
Assume that conditions {\upshape(A1)}-{\upshape(A3)} hold. Let $M\in  I_{v^*}$, $D_0>0$, $\vep_0>0$ and integer $k\geq 2$ be given. 
Then, for each constant $R>0$, there exists a constant $\widehat C_k = \widehat C_k(M, D_0, \vep_0, R) >0$ such that for any $(\alpha, q), (\alpha_i, q_i)\in \mathbb R\times X_0$ satisfying
\begin{align*}
	\triplenorm{(\alpha, q)}\leq R, \quad 	\triplenorm{(\alpha_i, q_i)}\leq R \quad (i=1, 2),
\end{align*} 
the following estimates hold for all $\vep\in(0,\vep_0]$ and $D\geq D_0$:
\begin{gather}
\big\|\mathcal F^\vep[\alpha  \vep^{k+\frac{1}{2}}\Phi_0^{\vep, k} + \vep^k q]\big\|_{L^{\infty}(\Omega)} \leq \widehat C_k\vep^{2k}\triplenorm{(\alpha, q)}^2,\label{eqn:mathcal F^vep bdd}\\
\big\|\mathcal F^\vep[\alpha_1\vep^{k+\frac{1}{2}}\Phi_0^{\vep, k} + \vep^k q_1] -\mathcal F^\vep[\alpha_2\vep^{k+\frac{1}{2}}\Phi_0^{\vep, k} + \vep^k q_2]\big\|_{L^{\infty}(\Omega)}\label{eqn:mathcal F^vep_Lip}\\
\hspace{10ex}\leq \widehat C_kR \vep^{2k}\triplenorm{(\alpha_1-\alpha_1, q_1-q_2)}. \notag
\end{gather}
\end{lemma}
\begin{proof}
We apply Taylor expansion with remainder to the function $F^\vep[p]$ defined in \eqref{eqn:mathcal F^vep}, yielding that 
\begin{align*}
\big|\mathcal F^\vep[p](x) \big|\leq C \Big[|p(x)|^2+\Big(\frac{\vep}{D}p(x)+\frac{1}{|\Omega|} \left(1-\frac{\vep}{D}\right)\int_\Omega p(x)\,dx\Big)^2\Big],
\end{align*}
where $C>0$ is a constant depending only on the sup-norm of $\nabla^2f$ on a compact region in $\mathbb R^2$ determined by $\|u_k^{\vep, D}\|_{L^{\infty}(\Omega)}$ and $\|p\|_{L^{\infty}(\Omega)}$. Setting $p = \alpha \vep^{k+\frac{1}{2}}\Phi_0^{\vep, k} + \vep^k q$ in the above inequality, we obtain \eqref{eqn:mathcal F^vep bdd} from  Theorem \ref{thm:approx_sol_scalarEQ} and Theorem \ref{thm:mathscr_L_k^vep} (iii). 

By similar calculation, we can show that
\begin{align*}
\big\|\mathcal F^\vep[p_1]-\mathcal F^\vep[p_2]\big\|_{L^{\infty}(\Omega)}
\leq C\max\big\{\|p_1\|_{L^{\infty}(\Omega)}, \|p_2\|_{L^{\infty}(\Omega)}\big\}\times\|p_1 -p_2\|_{L^{\infty}(\Omega)}.
\end{align*}
Setting $p_i = \alpha_i\vep^{k+\frac{1}{2}}\Phi_0^{\vep, k} + \vep^k q_i$ for $i=1, 2$ in the above inequality, we obtain
\eqref{eqn:mathcal F^vep_Lip}. \end{proof}

By Theorem \ref{thm:approx_sol_scalarEQ} (ii) and Theorem \ref{thm:mathscr_L_k^vep} (iii)-(iv), we have
\begin{align*}
	\big|\langle  \widehat\Phi_0^{\vep, k},\, \mathcal R^\vep[u^{\vep,D}_k]\rangle\big| \leq \|\mathcal R^\vep[u^{\vep,D}_k]\|_{L^{\infty}(\Omega)}|\langle\widehat\Phi_0^{\vep, k},\, 1\rangle|\leq C_k\vep^{k+\frac{3}{2}},
\end{align*}
which implies that
\begin{gather*}
\|\mathcal Q^{\vep}\mathcal R^\vep[u^{\vep,D}_k]\|_{L^{\infty}(\Omega)} = |\langle\widehat\Phi_0^{\vep, k},\, \mathcal R^\vep[u^{\vep,D}_k] \rangle| \cdot \|\Phi_0^{\vep, k}\|_{L^{\infty}(\Omega)} \leq  C_k\vep^{k+1},\\
\|(I-\mathcal Q^{\vep})\mathcal R^\vep[u^{\vep,D}_k]\|_{L^{\infty}(\Omega)}  \leq C_k\vep^{k+1}.
\end{gather*}
Therefore, using the above estimates,
Theorem \ref{thm:mathscr_L_k^vep}, Lemmas \ref{lem:mathscr L_k^vep inver_bdd} and \ref{lem:mathcal F^vep}, we obtain that
\begin{align*}
|\mathcal T_1[\alpha, q]|& \leq  C_k\widehat C_k\vep^{k-1}\triplenorm{(\alpha, q)}^2+C_k,\\
\|\mathcal T_2[\alpha, q]\|_{L^{\infty}(\Omega)} & \leq C_k\widehat C_k\vep^{k}\triplenorm{(\alpha, q)}^2+C_k\vep\|q\|_{L^{\infty}(\Omega)}+C_k\vep. 	
\end{align*}
Thus, if we define 
$$\mathcal B := \big\{(\alpha, q) \in \mathbb R\times X_0 \,\big|\, \triplenorm{(\alpha, q)} \leq 4C_k\big\}$$
and let $R = 4C_k$ in Lemma \ref{lem:mathcal F^vep}, then $\mathcal T(\mathcal B)\subset \mathcal B$ provided that
\begin{align*}
\displaystyle 0<\vep\leq \min\Big\{\Big(16C_k^2\widehat C_k\Big)^{-\frac{1}{k-1}},\,  \frac{1}{2},\, \frac{1}{4C_k}\Big\} =:\overline\vep(M, D_0, k)
\text{ and }D\geq D_0.
\end{align*} 
Moreover, for any $(\alpha_1, q_1)$, $(\alpha_2, q_2)\in \mathcal B$,
\begin{align*}
|\mathcal T_1[\alpha_1, q_1] - \mathcal T_1[\alpha_2, q_2]| & \leq 
4C_k^2 \widehat C_k \vep^{k-1}\triplenorm{(\alpha_1-\alpha_1, q_1-q_2)}, \\
|\mathcal T_2[\alpha_1, q_1] - \mathcal T_2[\alpha_2, q_2] | & \leq 
4C_k^2\widehat C_k \vep^{k}\triplenorm{(\alpha_1-\alpha_1, q_1-q_2)} + C_k\vep\|q_1-q_2\|_{L^{\infty}(\Omega)}.
\end{align*}
Thus, with the above choice of $\overline\vep$, we obtain that
\begin{align*}
\triplenorm{\mathcal T[\alpha_1, q_1] -\mathcal T[\alpha_2, q_2]}\leq\frac{5}{8}
\triplenorm{(\alpha_1 -\alpha_2, q_1-q_2)}, \quad \forall \;\vep\in(0, \overline\vep] \text{ and }D\geq D_0.
\end{align*}

Consequently, by the contraction mapping principle, the mapping $\mathcal T$ admits a unique fixed point $(\alpha^{\vep}, q^{\vep})$ in the ball $\mathcal B$ for all $\vep\in(0, \overline\vep]$ and $D \geq D_0$. Since $q^{\vep}$ actually belongs to class $C^{2+\alpha}(\overline\Omega)$ and satisfies the zero-Newmann boundary condition, $p= \alpha^{\vep} \vep^{k+\frac{1}{2}}\Phi_0^{\vep, k} + \vep^k q^{\vep}$ is a classical solution to \eqref{eqn:mathscr L^vep_k p}. Therefore, 
$$u^{\vep, D}:= u_k^{\vep, D} + \alpha^{\vep} \vep^{k+\frac{1}{2}}\Phi_0^{\vep, k} + \vep^k q^{\vep}$$
yields a family of classical solutions to the scalar nonlocal problem \eqref{eqn:MCRD_nonlocal_scalar} for all $\vep\in(0, \overline\vep]$ and $D\geq D_0$. This finishes the proof of existence of $u^{\vep, D}$ in  Theorem \ref{thm:main-scalar}. 

Now, Part (i) of Theorem \ref{thm:main-scalar} follows directly from Theorem \ref{thm:approx_sol_scalarEQ} (i) and the above construction. 
\end{proof}

Note that in order to prove Theorem~\ref{thm:main-scalar}, it suffices to apply Theorem~\ref{thm:approx_sol_scalarEQ} with $k=2$. Nevertheless, one advantage of establishing Theorem~\ref{thm:approx_sol_scalarEQ} for general integers $k\geq2$ is that, in combination with the estimate$\|\Phi_0^{\vep, k}\|_{L^\infty(\Omega)} = O(\vep^{-\frac{1}{2}})$ from Theorem~\ref{thm:mathscr_L_k^vep} (iii), the proof of Theorem~\ref{thm:main-scalar} yields the following higher-order approximation estimates for the solutions.

\begin{corollary}\label{cor:vep^k-accuracy}
Assume that conditions {\upshape (A1)-(A3)} hold. Then, for each given $M\in I_{v^*}$, $D_0> 0$ and integer $k\geq 2$, there exists a constant $\overline\vep = \overline\vep(M, D_0, k)>0$ such that the solution $(u^{\vep, D}, v^{\vep, D})$ to system \eqref{eqn:MCRD} obtained in Theorem \ref{thm:main} is approximated by the pair of functions $(u_k^{\vep, D}, v_k^{\vep, D})$ given in Theorem \ref{thm:approx_sol_scalarEQ} with accuracy
\begin{align*}
\big\|
(u^{\vep, D}, v^{\vep, D})-(u_k^{\vep, D} v_k^{\vep, D})\big\|_{L^{\infty}(\Omega)}\leq C_k\vep^{k}
\end{align*}
for all $\vep\in(0, \vep_0]$ and $D\geq D_0$, where $C_k = C_k(M, D_0, \overline\vep)>0$ is a constant depending only on the variables indicated. 
\end{corollary}

On the other hand, in \eqref{eqn:U^{pm,0} = h^mp(A_0)}, if we choose $U^{\pm,0}(x)\equiv h^\pm(A_0)$ in $\Omega^\pm$
instead, then the resulting approximate solution leads to a different but structurally symmetric family of functions to the one obtained in Theorem \ref{thm:approx_sol_scalarEQ}. This yields the following alternative existence result of Theorem~\ref{thm:main}, where we still denote the family of subsequent radial solutions by $(u^{\vep, D}, v^{\vep, D})$.
\begin{theorem}\label{thm:main2}
Assume that conditions {\upshape (A1)-(A3)} hold. Then, for each given
\begin{align*}
M\in  I_{v^*} \text{ and } D_0> 0,
\end{align*}
there exists a constant $\overline\vep = \overline\vep(M, D_0)>0$ such that system \eqref{eqn:MCRD-ss} has a family of spherically symmetric classical solutions $(u^{\vep, D}(r), v^{\vep, D}(r))$ defined for $\vep\in(0, \overline\vep]$ and $D\geq D_0$ and satisfying
\begin{align*}
\frac{1}{|\Omega|}\int_{\Omega} [u^{\vep,D}(x) +  v^{\vep,D}(x)] dx = M.
\end{align*}  
Moreover, these solutions satisfy the following properties:
\begin{enumerate}
\item  There exists a constant $C = C(M, D_0, \overline\vep)>0$ such that
\begin{align*}
\| u^{\vep,D}\|_{L^{\infty}(\Omega)}+\| v^{\vep,D}\|_{L^{\infty}(\Omega)}\leq C, \quad \forall \;\vep\in(0,\overline\vep] \text{ and }D\geq D_0.
\end{align*}
\item For each $\eta>0$, there exist two constants $K=K(\eta)
>0$ and $\overline\vep'=\overline\vep'(M, D_0,\eta)\leq \overline\vep$ such that for any $\vep\in (0, \overline\vep']$, $D\geq D_0$, there hold
\begin{align*}
|u^{\vep, D}(r)-h^-(v^*)|<\eta \; \text{ on } [0, \widehat R_*-\vep K] \text{ and }
|u^{\vep, D}(r)-h^+(v^*)|<\eta \; \text{on }  [\widehat R_*+\vep K, 1],
\end{align*}
where 
\begin{align}\label{eqn:hat R_*}
\widehat R_* = \widehat R_*(M):= \Big[\frac{v^*+h^+(v^*)-M}{h^+(v^*)-h^-(v^*)}\Big]^{1/N}\in(0,1).
\end{align}
\item $v^{\vep,D}(x)=\mathcal S^{\vep}[u^{\vep,D}]-\frac{\vep}{D} u^{\vep,D}(x)$ and
\begin{align*}
\lim\limits_{\vep\to0} \|v^{\vep, D} - v^*\|_{L^{\infty}(\Omega)} =0 \text{ uniformly with }D\geq D_0.
\end{align*}
\end{enumerate}
\end{theorem}

We conclude this section with the following remark.
\begin{remark}\label{rem:dependence on M}
\begin{enumerate}
\item For any fixed integer $k\geq2$, it follows directly from the proof of Theorem~\ref{thm:main-scalar} that the family of solutions $(u^{\vep, D}, v^{\vep, D})$ obtained in Theorem \ref{thm:main} depends smooth on the parameters $(M, D, \vep) \in I_{v^*}\times [D_0, \infty)\times (0,\overline\vep]$. 
\item An interesting question is to determine whether the family of solutions $(u^{\vep, D}, v^{\vep, D})$ obtained in Theorem \ref{thm:main} is independent of the choice of $k$. We can even ask more: For $\vep > 0$ small, is there {\it local uniqueness} of solutions to \eqref{eqn:MCRD-ss-T} near the limiting transition layer $(u^*, v^*)$ in suitable topology? Here, $v^*$ is defined as in condition {\upshape (A3)} and 
\begin{align*}
u^*(r) :=\begin{cases}
		h^+(v^*) &\text{ in } [0, R_*), \\
		h^-(v^*) &\text{ in } (R_*, 1].
		\end{cases}
\end{align*}
Recently, the local uniqueness of spike/bubble solutions has been widely investigated for singularly perturbed nonlinear elliptic problems; see \cite{CPY} and references therein. 
\end{enumerate}
\end{remark}

\section*{Acknowledgments}  The research of He is supported in part by NSFC grant No.~12071141. The research of Liu and He is supported in part by the Shanghai Science and Technology Innovation Action Plan in Basic Research Area (23JC1402600). He and Ye are partially supported by Science and Technology Commission of Shanghai Municipality (No. 22DZ2229014).

\end{document}